\documentclass[11pt, reqno, a4paper, twoside]{article}
\usepackage[utf8]{inputenc}
\usepackage[T1]{fontenc}
\usepackage{lmodern}
\usepackage[margin=2.5cm]{geometry}
\usepackage{microtype}
\usepackage[numbers,square,sort&compress]{natbib}

\usepackage{amsmath,amssymb,amsfonts,mathtools}
\usepackage{amsthm}
\usepackage{stmaryrd}
\usepackage{enumitem}
\usepackage{graphicx}
\usepackage{array}
\usepackage{multirow}
\usepackage[tableposition=top]{caption}
\usepackage{xcolor}
\usepackage{hyperref}
\definecolor{darkgreen}{rgb}{0.0,0.45,0.15}
\hypersetup{colorlinks=true,linkcolor=darkgreen,citecolor=darkgreen,urlcolor=darkgreen,
  pdftitle={Robust mixed methods for coupled Darcy--Forchheimer--transport problems},
  pdfauthor={Rishi Das, Harsha Hutridurga, Amiya K. Pani, Ricardo Ruiz-Baier}}

\theoremstyle{plain}
\newtheorem{theorem}{Theorem}[section]
\newtheorem{lemma}[theorem]{Lemma}

\theoremstyle{definition}
\newtheorem{remark}{Remark}[section]

\allowdisplaybreaks
\numberwithin{equation}{section}
\numberwithin{figure}{section}
\numberwithin{table}{section}

\newcommand*\tnorm[1]{|\!|\!| #1 |\!|\!|}
\newcommand{\bn}{\boldsymbol{n}}

\newcommand{\bu}{\boldsymbol{u}}
\newcommand{\bv}{\boldsymbol{v}}
\newcommand{\bw}{\boldsymbol{w}}

\renewcommand{\bf}{\boldsymbol{f}}
\newcommand{\bg}{\boldsymbol{g}}

\newcommand{\bx}{\boldsymbol{x}}

\newcommand{\bz}{\boldsymbol{z}}

\newcommand\bsigma{\boldsymbol{\sigma}}
\newcommand\bzeta{\boldsymbol{\zeta}}
\newcommand\bzero{\boldsymbol{0}}

\newcommand\bxi{\boldsymbol{\xi}}
\newcommand\btau{\boldsymbol{\tau}}

\newcommand{\rF}{\mathtt{F}}

\newcommand{\rH}{\mathrm{H}}
\newcommand{\rL}{\mathrm{L}}
\newcommand{\rW}{\mathrm{W}}
\newcommand{\rQ}{\mathrm{Q}}

\newcommand{\bB}{\mathbf{B}}
\newcommand{\bH}{\mathbf{H}}
\newcommand{\bX}{\mathbf{X}}

\newcommand{\bL}{\mathbf{L}}
\newcommand{\bW}{\mathbf{W}}
\newcommand{\bV}{\mathbf{V}}

\newcommand{\bZ}{\mathbf{Z}}

\newcommand{\bbA}{\mathbb{A}}
\newcommand{\bbD}{\mathbb{D}}
\newcommand{\bbB}{\mathbb{B}}
\newcommand{\bbC}{\mathbb{C}}

\newcommand{\cA}{\mathcal{A}}
\newcommand{\cB}{\mathcal{B}}

\newcommand{\cL}{\mathcal{L}}

\newcommand{\cI}{\mathcal{I}}
\newcommand{\cT}{\mathcal{T}}
\newcommand{\cS}{\mathcal{S}}
\newcommand{\cR}{\mathcal{R}}

\newcommand\vdiv{\mathop{\mathrm{div}}\nolimits}

\begin{document}

\title{Robust mixed methods for coupled Darcy--Forchheimer--transport problems}

\author{Rishi Das\thanks{IITB--Monash Research Academy, Indian Institute of Technology Bombay, Powai, Mumbai, Maharashtra 400076, India. Email: \texttt{rishi.das@monash.edu}.}
\and Harsha Hutridurga\thanks{Department of Mathematics, Indian Institute of Technology Bombay, Powai, Mumbai, Maharashtra 400076, India. Email: \texttt{hutridurga@iitb.ac.in}.}
\and Amiya K. Pani\thanks{Department of Mathematics, BITS-Pilani, K.K. Birla Goa Campus, Zuarinagar, Goa-403726, India. Email: \texttt{amiyap@goa.bits-pilani.ac.in}.}
\and Ricardo Ruiz-Baier\thanks{School of Mathematics, Monash University, 9 Rainforest Walk, Melbourne VIC 3800, Australia; and Universidad Adventista de Chile, Casilla 7-D, Chill\'an, Chile. Email: \texttt{ricardo.ruizbaier@monash.edu}. Corresponding author.}}

\date{}

\maketitle

\begin{abstract}
We adapt the recent theory for unique solvability of perturbed saddle-point problems in Banach spaces to the case of parameter-independent stability bounds. This also constitutes a slight modification of the Brezzi--Braess theory for parameter-robust stability of perturbed saddle-point problems from the Hilbert to the Banach setting. We apply the new abstract result to the mixed formulation of the advection-diffusion-reaction equation and tackle also its coupling with the Darcy--Forchheimer equations. The complete system is shown to be stable irrespective of the model parameters of permeability, Forchheimer coefficient, and reaction modulation. We discretise the problem with mixed finite element methods and show convergence in appropriately parameter-weighted norms. We also design operator-based preconditioners for the full system, utilising weighted norms that provide robustness with respect to some model parameters. We also present a few numerical examples that show the convergence and robustness of the proposed method in some regimes of practical interest.


\end{abstract}
\section{Introduction}\label{sec:intro}
This paper investigates the robust parameter stability of a system of the Darcy--Forchheimer equations coupled with the advection-diffusion-reaction (ADR) equation as presented in \cite{sayah2021finite}.
  Let $\Omega\subset \mathbb{R}^d$, $d\in \{2,3\}$, denote a  {bounded} domain with Lipschitz-continuous boundary $\partial\Omega$, on which its outward unit normal  is
denoted by $\bn$.  The boundary is decomposed into  two sub-boundaries $\Gamma_{\bu}$ and $\Gamma_p$, both with positive $(d-1)-$Hausdorff measure. The Darcy--Forchheimer equations are written in their discharge velocity and pressure head formulation $\bu,p$ featuring a nonlinear drag term, and the ADR transport using a conservative form in terms of total (diffusive plus advective) flux  $\bzeta$ and solute concentration  $\varphi$:
\begin{subequations}\label{eq:strong}
\begin{align}
    \kappa^{-1}\bu + {\rF|\bu|\bu} + \nabla p & = \bf(\varphi) \quad \text{in $\Omega$},\\
    \vdiv \bu & = 0 \quad \text{in $\Omega$},\\
    \alpha^{-1}\bzeta & = \nabla \varphi - { \alpha^{-1}} \bu \varphi \quad \text{in $\Omega$},\\
    \varrho_0 \varphi - \vdiv \bzeta & = m \quad \text{in $\Omega$},\\
    \bu\cdot \bn = \varphi & = 0 \quad \text{on $\Gamma_{\bu}$},\\
    p  = \bzeta \cdot \bn  & = 0 \quad \text{on $\Gamma_p$}.
\end{align}\end{subequations}
Here, $\kappa$ is the hydraulic conductivity (permeability of the porous matrix divided by the fluid viscosity),
$\bf$ is a forcing vector that depends on the concentration field and which is assumed Lipschitz continuous in the following sense:
\begin{equation}\label{eq:f.lipschitz}
\|\bf(\varphi_1)-\bf(\varphi_2)\|_{0,\Omega}\leq\alpha_1\|\varphi_1-\varphi_2\|_{0,\Omega}\qquad\forall\varphi_1,\varphi_2\in\rL^2(\Omega),
\end{equation}
where  \(\alpha_1\) is a positive constant. This holds, e.g., if $\bf(\bx,\eta) = \bf_{\bzero}(\bx)+\bf_1(\eta)$ with $\bf_{\bzero}\in\bL^2(\Omega)$ and $\bf_1$ Lipschitz continuous with constant $\alpha_1$ (as in \cite[Assumption~2.1]{sayah2021finite} and in our numerical examples). Taking $\varphi_2=0$ in \eqref{eq:f.lipschitz} and using H\"older's inequality, we also have
\begin{equation}\label{eq:lip.bound.f}\|\bf(\varphi)\|_{0,\Omega}\leq \|\bf(0)\|_{0,\Omega}+ \alpha_1\|\varphi\|_{0,\Omega}\leq \|\bf(0)\|_{0,\Omega}+ \alpha_1|\Omega|^{\frac13}\|\varphi\|_{0,6,\Omega}\qquad\forall\varphi\in\rL^6(\Omega).\end{equation}
 In addition,  $\rF>0$ is the Forchheimer coefficient, $\alpha>0$ is the diffusivity coefficient, $\varrho_0>0$ is the reaction coefficient  and $m$ is the concentration source/sink.

 The mathematical and numerical analysis of the Darcy--Forchheimer equations alone is reasonably  well developed in the literature.
 We may refer to \cite{fabrie1989regularity} for existence, uniqueness and regularity; \cite{park2005mixed,girault2008numerical,pan2012mixed} for mixed and primal finite elements; \cite{rui2012block} for block-centred finite difference method and to \cite{dhpr_sinum26} for robust stability and preconditioning in weighted norms.
For couplings of Darcy--Forchheimer (or Brinkman--Forchheimer) flow with other physical models including free-fluid/porous-medium interaction, refer to \cite{caucao2020conforming}, for heat transfer and double-diffusion, see, \cite{caucao2020fully,caucao2021fully}, and for transport of a solute  in \cite{sayah2021finite}. A variety of discretisations of the coupled flow-transport problem with finite differences and finite element methods can be found in, e.g., \cite{deugoue2022numerical,huang2025numerical,bonetti2025conforming}, or in \cite{allendes2023numerical} for the case of singular sources (see also \cite{caucao2021fully,gatica2022p} for analysis of similar models in the fully mixed case). Such problems find applications in reservoir modelling and geothermal energy optimisation, for example.

\smallskip
\noindent\textbf{Main contribution.}
Owing to the form of \eqref{eq:strong}, and  with the aim of preserving robustness with respect to the model parameters, both at the continuous and discrete levels, the variational formulation of the problem is cast involving sums and intersections of reflexive Banach spaces. The main tool in the theoretical analysis of well-posedness (as well as error estimates) is a new abstract result for perturbed saddle-point problems in reflexive Banach spaces that stands as an extension of \cite{correa2022continuous} (to the parameter-robust context) as well as of \cite{boon2021robust,braess1996stability} (to its non-Hilbert counterpart).
 A more closely related reference that deals with the design of preconditioning for flow-transport couplings, is the recent work \cite{cai2026parameter} addressing robust stability of a linearised model for coupled heat and Biot equations. The main differences here are that, due to the presence of the nonlinear terms and Banach spaces functional structure, we need to take into consideration sums of three non-Hilbert spaces and realise the action of duality maps that are non-standard when compared to the recent literature on operator preconditioning (see, e.g., \cite{baerland2020observation,cai2026parameter,hong2022robust}). It should be noted that, in the context of ADR problems, the presence of the advective term renders the underlying saddle-point structure non-symmetric. The treatment of this term within the design of preconditioners remains challenging for both preconditioning strategies and iterative solvers.

The preconditioners those  we explore here are  inspired by Schur complement and operator preconditioning frameworks for symmetric elliptic operators; however, their extension to the nonsymmetric and nonnormal setting lacks a firm theoretical basis  \cite{wathen2015preconditioning,ernst2000residual}. Additionally, while the individual Darcy--Forchheimer and reaction-diffusion subproblems are uniformly bounded and inf-sup stable, robustness with respect to $\alpha$ and $\rF$ is partially lost in the coupled system due to the fixed-point coupling of the nonlinear velocity and transport flux. This analytical loss of robustness does not manifest in our observed numerical convergence rates, suggesting it may be an artifact of sub-optimal scaling in the weighted norms or limitations inherent to the fixed-point framework.

\smallskip
\noindent\textbf{Common notation and preliminaries.} For $s\geq 0$ and $r\in[1,+\infty]$, we denote by $\rL^r(\Omega)$ and $\rW^{s,r}(\Omega)$ the usual Lebesgue and Sobolev spaces endowed with the norms $\|\bullet\|_{0,r,\Omega}$ and $\|\bullet\|_{s,r,\Omega}$, respectively.
Note that $\rW^{0,r}(\Omega)=\rL^r(\Omega)$.
If $r = 2$, we write $\rH^{s}(\Omega)$ in place of $\rW^{s,2}(\Omega)$, and denote the corresponding norm by $\|\bullet\|_{s,\Omega}$. {Recall also the \(\rL^r(\Omega)\)-\(\rL^s(\Omega)\) embedding for finite measure domains, i.e.,   \(\|\bullet\|_{0,r,\Omega}\leq|\Omega|^{\frac1r-\frac1s}\|\bullet\|_{0,s,\Omega},\) for \(1<r\leq s<\infty.\)} Also, $C_{\mathsf{emb}}$ denotes the constants arising from a Sobolev embedding, which depends only on $|\Omega|.$ We
denote by $\mathbf{H}$ the vectorial counterpart of a generic scalar functional space $\rH$. The $\rL^2(\Omega)$ inner product for scalar, vector, or tensor valued functions
is denoted by $(\bullet,\bullet)_{0,\Omega}$.
Moreover, we  use for $r,s \in [1,+\infty)$, the Banach space
$\bH^r(\vdiv_s,\Omega)=\{ \bv \in \bL^r(\Omega): \vdiv \bv \in \rL^s(\Omega)\},$ endowed with the norm $\|\bv\|_{r,\vdiv_s,\Omega} := \|\bv\|_{0,r,\Omega} + \|\vdiv \bv\|_{0,s,\Omega}$ (when $r=2$ we simply write $\bH(\vdiv_s,\Omega)$ and $\|\bullet\|_{\vdiv_s,\Omega}$), {and in particular, for $r=2, s=2,$ we call it $\bH(\vdiv,\Omega)$.}
Often, we  use a subscript $\Gamma_i$ on the functional space to denote that the (appropriate) traces vanish on the part of the boundary $\Gamma_i\subset \partial \Omega$.

For two Banach spaces $(X,\|\bullet\|_X)$ and $(Y,\|\bullet\|_Y),$ we denote by $\cL(X,Y)$ the space of bounded linear maps from $X$ to $Y$. Let $(Z,\|\bullet\|_Z)$ be another Banach space. The intersection and sum spaces are
\begin{gather*} X \cap Y \cap Z = \{ v : v\in X \ \text{and} \ v\in Y \ \text{and}\ v\in Z\}, \\ X+ Y + Z = \{u+v+w: u\in X \ \text{and}\ v\in Y\ \text{and} \ w \in Z\},\end{gather*}
and their respective norms are (see \cite{baerland2020observation}, noting that the same norm characterisation done there for Hilbert spaces holds also for the Banach case)
\[ \| r\|_{X\cap Y \cap Z} = \|r\|_X + \|r\|_Y + \|r\|_Z, \quad \| r\|_{X + Y + Z} = \inf_{\substack{r = u+v+w, \\ u \in X, v\in Y, w\in Z}}\,\bigl[\|u\|_X + \|v\|_Y+ \|w\|_Z\bigr].\]
Here, and in what follows, $X,Y,Z$ are assumed to be continuously embedded in a common Hausdorff topological vector space (in all our applications, $\rL^1(\Omega)$ or its vector counterpart), so that the sums and intersections above are meaningful. Under this  assumption, both quantities define norms and $X\cap Y\cap Z$ and $X+Y+Z$ are Banach spaces \cite[Lemma~2.3.1]{bergh1976interpolation}. For the sum, homogeneity and the triangle inequality follow directly from the definition of the infimum (if $r = u+v+w$ and $\tilde r = \tilde u+\tilde v+\tilde w$, then $r+\tilde r = (u+\tilde u)+(v+\tilde v)+(w+\tilde w)$ and $\lambda r = \lambda u+\lambda v + \lambda w$), whereas definiteness is a consequence of the continuous embeddings: if $\|r\|_{X+Y+Z}=0$, there exist decompositions $r = u_n+v_n+w_n$ with $\|u_n\|_X+\|v_n\|_Y+\|w_n\|_Z\to 0$, and therefore $r = u_n+v_n+w_n \to 0$ in $\rL^1(\Omega)$, that is, $r = 0$.
Furthermore, if $X \cap Y \cap Z$ is dense in $X$ and $Y$ and $Z$, then the following relation (understood as an identification under an isometry) holds
\begin{equation}\label{eq:duality} (X+Y+Z)' = X'\cap Y' \cap Z',\qquad \text{with}\quad \|\ell\|_{(X+Y+Z)'} = \max\bigl\{\|\ell\|_{X'},\|\ell\|_{Y'},\|\ell\|_{Z'}\bigr\}.\end{equation}
For a fixed $\alpha>0$, by $\alpha X$ we denote the Banach space whose elements coincide with those in $X$ but are measured with the weighted norm
$\| \bullet \|_{\alpha X}:= \alpha\|\bullet \|_X$; that is, multiplication by a (positive) parameter does not modify the underlying set of functions but only the way in which their size is measured. For instance, $\kappa^{-\frac12}\bL^2(\Omega)$ is the space $\bL^2(\Omega)$ endowed with the norm $\bv\mapsto\kappa^{-\frac12}\|\bv\|_{0,\Omega}$, and $\kappa^{\frac12}\rH^1_{\Gamma_p}(\Omega)$ is the space $\rH^1_{\Gamma_p}(\Omega)$ endowed with the norm $q\mapsto \kappa^{\frac12}\|\nabla q\|_{0,\Omega}$. Throughout the paper, and as in \cite{dhpr_sinum26}, the parameters $\kappa,\rF,\alpha,\varrho_0$ are positive constants and, after a suitable non-dimensionalisation, we assume that $\kappa\leq 1$ and $\varrho_0\leq 1$. These bounds cover the regimes of main interest in applications, in particular, in biological and biomedical flows through soft tissues (where hydraulic conductivities are typically very small and reaction/uptake rates are moderate), as well as in low-permeability geological formations.
For any scalar function \(\varphi,\) we recall that the sign function (denoted by \(\mathrm{sgn}(\varphi)\)) takes the value \(1,\) for non-negative \(\varphi,\) and \(-1\), otherwise.
Additionally, \(\varphi\, \mathrm{sgn}(\varphi) = |\varphi|.\)

\smallskip
\noindent\textbf{Outline.} The remainder of the paper is organised as follows. Section~\ref{sec:abstract} presents the weak formulation for \eqref{eq:strong}, specifying the required parameter-weighted spaces. Also we state and prove an abstract theory for robust stability of perturbed saddle-point problems in Banach spaces. In Section~\ref{sec:existence}, we formulate our problem in the context of abstract theory and show, using fixed-point arguments, that there exists a unique solution of the coupled system and that we have a parameter-independent continuous dependence on data. Section~\ref{sec:fe} deals with the finite element formulation, the robust discrete stability--via a discrete fixed-point argument--and the derivation of robust a priori error estimates. We also derive robust (in $\kappa,\rF,\varrho_0$) \textit{a priori} error estimates, which, owing to the cubic monotonicity of the Forchheimer term, are of H\"older type in the worst case. We construct variable block-diagonal operator preconditioners in Section \ref{sec:precond} and show their robustness with respect to all model parameters and the discretisation parameter,  except the diffusivity coefficient. Finally, Section~\ref{sec:numer} includes a series of numerical tests showing the  convergence of the method in 2D, also demonstrating the performance of the  preconditioners.

\section{Mixed weak formulation and abstract setting}\label{sec:abstract}
\subsection{Weak form of the Darcy--Forchheimer transport equations}
We now consider the weighted spaces proposed in \cite[Section~3]{dhpr_sinum26} for Darcy velocity and pressure. As sets, these are
\[\bV = \bigl\{\bv\in\bL^3(\Omega):\ \vdiv\bv\in\rL^2(\Omega),\ \bv\cdot\bn = 0 \ \text{on}\ \Gamma_{\bu}\bigr\}, \qquad \rQ = \rL^2(\Omega)\]
(note that $\rH^1_{\Gamma_p}(\Omega)$ and $\rW^{1,\frac32}_{\Gamma_p}(\Omega)$ are both contained in $\rL^2(\Omega)$ for $d\in\{2,3\}$), but they are understood as the following intersection and sum of weighted spaces, respectively,
\[\bV:=  \kappa^{-\frac{1}{2}}\bL^2(\Omega) \cap \bH_{\Gamma_{\bu}}(\vdiv,\Omega) \cap \rF^{\frac13} \bL^3(\Omega), \quad\text{and}\quad \rQ := \rL^2(\Omega) + \kappa^{\frac{1}{2}}\rH^1_{\Gamma_p}(\Omega) + \rF^{-\frac13}\rW^{1,\frac32}_{\Gamma_p}(\Omega),\]
which amounts to endowing them with the following norms
\begin{subequations}\label{eq:norms}
    \begin{align}
   \|\bv\|_{\bV} &:= \kappa^{-\frac12}\|\bv\|_{0,\Omega} + \|\vdiv\bv\|_{0,\Omega} + \rF^{\frac13}\|\bv\|_{0,3,\Omega}, \quad\text{and}\label{eq:V.norm}\\
   \|q\|_{\rQ} & := \inf_{\substack{q = q_1+q_2+q_3, \\ (q_1,q_2,q_3) \in \rL^2(\Omega) \times \rH^1_{\Gamma_{\!p}}\!(\Omega)\times \rW^{1,\frac32}_{\Gamma_p}(\Omega)}}\bigl(\|q_1\|_{0,\Omega} + \kappa^{\frac{1}{2}}\|\nabla q_2\|_{0,\Omega} + \rF^{-\frac13} \|\nabla q_3\|_{0,\frac32,\Omega}\bigr),\label{def:Qhat-norm}
    \end{align}
\end{subequations}
respectively. Both quantities are norms: \eqref{eq:V.norm} is a sum of (weighted) norms, and \eqref{def:Qhat-norm} is the sum norm discussed in Section~\ref{sec:intro}, since $\|\nabla\bullet\|_{0,\Omega}$ and $\|\nabla\bullet\|_{0,\frac32,\Omega}$ are norms in $\rH^1_{\Gamma_p}(\Omega)$ and $\rW^{1,\frac32}_{\Gamma_p}(\Omega)$, respectively, thanks to the Poincar\'e inequality ($\Gamma_p$ having positive measure). Let us also stress that, although for fixed parameters $\|\bv\|_{0,\Omega}\leq |\Omega|^{\frac16}\|\bv\|_{0,3,\Omega}$, the $\rL^2$-contribution cannot be dropped from \eqref{eq:V.norm} without losing robustness: one has $\kappa^{-\frac12}\|\bv\|_{0,\Omega}\leq |\Omega|^{\frac16}\kappa^{-\frac12}\rF^{-\frac13}\,\rF^{\frac13}\|\bv\|_{0,3,\Omega}$, and the constant $\kappa^{-\frac12}\rF^{-\frac13}$ is unbounded when $\kappa\to0$ or $\rF\to0$. The term $\kappa^{-\frac12}\|\bv\|_{0,\Omega}$ is the natural energy associated with the Darcy term (dominant in the regime $\rF\to0$), whereas $\rF^{\frac13}\|\bv\|_{0,3,\Omega}$ is associated with the Forchheimer term (dominant for large $\rF$).
Denote as follows the trial and test function spaces for flux and concentration (following \cite[Section 3]{boon2021robust})
\[ \bW := \alpha^{-\frac12} \bL^2(\Omega) \cap \bH_{\Gamma_p}(\vdiv_{\frac65},\Omega), \quad \text{and}\quad\Psi := \alpha^{\frac12}\rH^1_{\Gamma_{\bu}}(\Omega) + \rL^6(\Omega), \]
that is, as sets, $\bW = \{\bw\in\bL^2(\Omega):\ \vdiv\bw\in\rL^{\frac65}(\Omega),\ \bw\cdot\bn=0\ \text{on}\ \Gamma_p\}$ and $\Psi = \rL^6(\Omega)$ (since $\rH^1(\Omega)\subset\rL^6(\Omega)$ for $d\in\{2,3\}$), equipped with the norms
\begin{subequations}
\begin{align}\|\bw\|_{\bW} &:= \alpha^{-\frac12}\|\bw\|_{0,\Omega} + \| \vdiv \bw\|_{0,\frac65,\Omega},\quad\text{and} \\
\|\psi\|_{\Psi} &:= \inf_{\chi \in \rH^1_{\Gamma_{\bu}}(\Omega)}\bigl(\|\psi-\chi\|_{0,6,\Omega} + \alpha^{\frac12}\|\nabla \chi\|_{0,\Omega} \bigr), \label{eq:Psi.norm}\end{align}\end{subequations}
respectively. Again, both are norms (the second one is the sum norm for two spaces, written in terms of the single auxiliary function $\chi$). Moreover, for all $\psi\in\Psi$ and any $\chi\in \rH^1_{\Gamma_{\bu}}(\Omega)$, the Sobolev embedding $\|\chi\|_{0,6,\Omega}\leq C_{\mathsf{emb}}\|\nabla\chi\|_{0,\Omega}$ shows $\|\psi\|_{0,6,\Omega}\le \|\psi-\chi\|_{0,6,\Omega}+C_{\mathsf{emb}}\alpha^{-\frac12}\,\alpha^{\frac12}\|\nabla\chi\|_{0,\Omega}$, and  hence,
\begin{equation}\label{eq:L6-Psi}
\|\psi\|_{0,6,\Omega}\leq \bigl(1+C_{\mathsf{emb}}\alpha^{-\frac12}\bigr)\|\psi\|_{\Psi}\qquad\forall\psi\in\Psi.
\end{equation}
This explains the $\alpha$-dependence of some of the constants appearing below. Finally, the reaction term will be measured with the seminorm $\varrho_0^{\frac12}\|\psi\|_{0,\Omega}$, and we define
\[\tnorm{(\bw,\psi)}_{\bW\times\Psi} := \|\bw\|_{\bW}+\|\psi\|_{\Psi}+\varrho_0^{\frac12}\|\psi\|_{0,\Omega}\qquad\forall (\bw,\psi)\in\bW\times\Psi,\]
which is the norm \eqref{eq:triple.norm}.

The weak formulation of \eqref{eq:strong} consists in finding  $(\bu,p,\bzeta,\varphi)\in \bV\times \rQ\times \bW \times \Psi$ such that
\begin{subequations}\label{eq:weak}
    \begin{align}
    a_1(\bu,\bv) + d_1(\bu; \bu,\bv) +  b_1(\bv,p) & = F^{\varphi}(\bv) \quad \forall \bv \in \bV,\label{eq:weak,1}\\
    b_1(\bu,q) & = 0 \quad \forall q \in \rQ,\label{eq:weak,2}\\
    a_2(\bzeta,\bw) + b_2(\bw,\varphi) + d_2^{\bu}(\bw, \varphi) & = 0 \quad \forall \bw \in \bW, \label{eq:weak,3}\\
    b_2(\bzeta, \psi) - c_2(\varphi,\psi) & = H(\psi) \quad \forall \psi \in \Psi, \label{eq:weak,4}
    \end{align}
\end{subequations}
where the  bilinear forms $a_1:\bV\times\bV\to\mathbb{R}$, $b_1:\bV\times\rQ \to \mathbb{R}$,  $a_2:\bW\times\bW\to\mathbb{R}$, $b_2:\bW\times\Psi \to \mathbb{R}$, $c_2: \Psi \times \Psi \to \mathbb{R}$,  $d_2^{\hat{\bu}}: \bW \times \Psi \to \mathbb{R}$ (for a fixed $\hat{\bu}$),
linear functionals $F^{\hat{\varphi}}:\bV\to\mathbb{R}$ (for a fixed $\hat{\varphi}$), $H:\Psi \to \mathbb{R}$, and nonlinear form $d_1:\bV\times\bV\times\bV \to \mathbb{R}$, are defined as:
\begin{gather*}
  a_1(\bu,\bv):= \int_\Omega \kappa^{-1}\bu\cdot\bv, \quad {b_1(\bv,q):= -\int_\Omega q\,\vdiv\bv},  \quad d_1(\bw;\bu,\bv):=\int_\Omega \rF |\bw|\bu\cdot\bv, \\
  a_2(\bzeta,\bw):= \int_\Omega \alpha^{-1}\bzeta\cdot\bw, \quad  b_2(\bw,\psi):= \int_\Omega \psi\,\vdiv\bw, \quad   c_2(\varphi,\psi):=\int_\Omega { \varrho_0} \varphi\,\psi, \\
  d_2^{\hat{\bu}}(\bw,\psi):={\int_\Omega  \alpha^{-1}\hat{\bu}\cdot \bw \,\psi},\quad
  F^{\hat{\varphi}}(\bv):= \int_\Omega \bf(\hat{\varphi})\cdot\bv, \quad H(\psi) := - \int_\Omega m\,\psi.
\end{gather*}

\subsection{Abstract setting for robust well-posedness}
Let $\bX$ and $Y$ be reflexive Banach spaces equipped with the norms $\|\bullet\|_{\bX}$ and $\|\bullet\|_Y$, respectively.%
Further, let $Y_c$ be a dense linear subspace of $Y$, and let $c:Y_c\times Y_c\to\mathbb{R}$ be a symmetric and positive semi-definite bilinear form with associated seminorm $|v|_c := c(v,v)^{\frac12}$ (in our application, $Y_c=Y=\Psi$, $c(\varphi,\psi)=\int_\Omega\varphi\,\psi$, and the scaling parameter $t$ introduced below is $t=\varrho_0^{\frac12}$).  Following \cite{braess1996stability}, we assume that $Y_c$ is complete with respect to the norm $\|\bullet\|_Y+|\bullet|_c$ (it is then also reflexive, being isometric to a closed subspace of the product of $Y$ with the Hilbert completion of $Y_c/\{v: |v|_c=0\}$ endowed with $c(\bullet,\bullet)$). When $t=0$ the form $c(\bullet,\bullet)$ does not enter the problem, and we simply set $Y_c:=Y$. Consider the problem of finding $(\bsigma,u) \in \bX\times Y_c$ such that
\begin{subequations}\label{eq:abstract}
    \begin{align}
    a(\bsigma,\btau) + b(\btau,u) & = F(\btau) \quad \forall \btau \in \bX, \label{eq:weak.1}\\
    b(\bsigma, v) - t^2c(u,v) & = G(v)\label{eq:weak.2} \quad \forall v\in Y_c,
\end{align}\end{subequations}
where $0\leq t\leq 1$ is a fixed scaling parameter and  $a:\bX\times \bX \to \mathbb{R}$, $b:\bX\times Y \to \mathbb{R}$, and $c:Y_c \times Y_c \to \mathbb{R}$
are given bilinear forms. Here, also $(F,G) \in \bX'\times Y'$ are given linear forms. Let $\bB:\bX\to Y'$ be a bounded linear map induced by $b(\bullet,\bullet)$ as $\bX \ni \btau\mapsto \bB\btau$ with $\langle \bB \btau, v\rangle := b(\btau,v)$ for all $v\in Y$, and denote $\bZ:=\ker(\bB)$. We introduce the $t$-dependent norm
\begin{equation}\label{eq:triple.norm}
\tnorm{(\btau,v)}_{\bX\times Y}:=\|\btau\|_{\bX}+\|v\|_{Y}+t|v|_c\qquad\forall(\btau,v)\in\bX\times Y_c,
\end{equation}
which, for each fixed $t>0$, is equivalent to the norm $\|\btau\|_{\bX}+\|v\|_Y+|v|_c$ (so that $\bX\times Y_c$ endowed with \eqref{eq:triple.norm} is a reflexive Banach space), and which reduces to the product norm of $\bX\times Y$ when $t=0$. Note that \eqref{eq:triple.norm} also contains the energy associated with the perturbation term.

For each fixed $t\in[0,1]$, the unique solvability of \eqref{eq:abstract} follows from \cite[Theorem~3.4]{correa2022continuous}, under the assumptions that $Y_c=Y$, $a,b,c$ are bounded, $c$ is positive semi-definite, $a$ satisfies an inf-sup condition on $\bZ$, and $b$ satisfies an inf-sup condition. If the solution is measured in the norm $\|\bsigma\|_{\bX}+\|u\|_Y$ of $\bX\times Y$, the resulting bounds can be made uniform in $t\in[0,1]$ (since $\|t^2c\|\leq\|c\|$); however, they depend on the norm $\|c\|$ of $c$ on $Y\times Y$ and on the inverse of the inf-sup constant of $a$ on $\bZ$. In our application (see Section~\ref{sec:existence}) one has $t^2=\varrho_0$, $c(\varphi,\psi) = \int_\Omega\varphi\,\psi$, and, owing to \eqref{eq:L6-Psi}, $\|c\|\leq |\Omega|^{\frac23}(1+C_{\mathsf{emb}}\alpha^{-\frac12})^2$,
which degenerates as $\alpha\to0$. Moreover, such bounds do not control the reaction energy $t|u|_c$, which is an essential component of the norm whose Riesz map yields a parameter-robust preconditioner (cf. Section~\ref{sec:precond}). If, on the other hand, one endows $Y$ with the $t$-dependent norm $\|\bullet\|_Y+t|\bullet|_c$ in order to control this energy, then the inf-sup constant of $b$ with respect to this norm is only bounded below by $C_b(1+t\|c\|^{\frac12})^{-1}$,
and the corresponding stability bounds from \cite[Theorem~3.4]{correa2022continuous} are no longer uniform. Furthermore, if $c(\bullet,\bullet)$ is only defined on a proper dense subspace $Y_c$ and is not bounded with respect to $\|\bullet\|_Y$ (e.g., $c(p,q)=\int_\Omega\nabla p\cdot\nabla q$ with $Y_c=\rH^1(\Omega)$ dense in $Y=\rL^2(\Omega)$), then \cite[Theorem~3.4]{correa2022continuous} does not apply at all. For these reasons, in order to track the robustness of the stability estimates with respect to the parameters, we instead extend
the Brezzi--Braess approach \cite[Lemma~3]{braess1996stability} (see also \cite[Theorem~2.1]{boon2021robust}) to the setting of reflexive Banach spaces: the inf-sup condition for $b$ is only required with respect to $\|\bullet\|_Y$, the term $t|u|_c$ is controlled through the positive semi-definiteness of $c$, and the price to pay is a Braess-type condition on $a$ and $b$ (cf. \eqref{eq:Braess} below).

For this we adapt  \cite[Theorem~49.15]{ern2021galerkin} to the case of perturbed saddle-point problems. With
\begin{align}\label{eq:curl.A}\mathcal{A}(\bsigma,u;\btau,v):=a(\bsigma,\btau)+b(\btau,u)+b(\bsigma,v)-t^2c(u,v), \quad \forall (\bsigma,u),(\btau,v)\in \bX \times Y_c,\end{align}
problem \eqref{eq:abstract} reads: find $(\bsigma,u)\in\bX\times Y_c$ such that $\cA(\bsigma,u;\btau,v) = F(\btau)+G(v)$ for all $(\btau,v)\in\bX\times Y_c$. By the Banach--Ne\v{c}as--Babu\v{s}ka theorem \cite[Theorem~25.9]{ern2021galerkin}, applied in the reflexive Banach space $(\bX\times Y_c,\tnorm{(\bullet,\bullet)}_{\bX\times Y})$, the invertibility of $\cA$ will follow as a consequence of the following conditions
\begin{subequations}
   \begin{align}
\mathrm{BNB~1:}\qquad&\inf_{\substack{(\bsigma,u)\in\bX\times Y_c \\ (\bsigma,u)\neq(\bzero,0)}}
\ \sup_{\substack{(\btau,v)\in\bX\times Y_c \\ (\btau,v)\neq(\bzero,0)}}
\frac{\cA(\bsigma,u;\btau,v)}
     {\tnorm{(\bsigma,u)}_{\bX\times Y}\,
      \tnorm{(\btau,v)}_{\bX\times Y}}
 \geq C_{\natural}>0,\label{eq:BNB1}\\
\mathrm{BNB~2:}\qquad& \forall(\btau,v)\in\bX\times Y_c, \quad [\forall (\bsigma,u)\in\bX\times Y_c,\, \cA(\bsigma,u;\btau,v)=0]\implies[(\btau,v)={(\bzero, 0)}]\label{eq:BNB2},
\end{align}\end{subequations}
where \(C_{\natural}\) is a positive constant independent of \(t\), and the bound $\tnorm{(\bsigma,u)}_{\bX\times Y}\leq C_\natural^{-1}\bigl[\|F\|_{\bX'}+\|G\|_{Y'}\bigr]$ then holds.
The main result is discussed below.

\begin{theorem}[Robust well-posedness in Banach spaces]\label{th:abstract}
 Assume that
\begin{itemize}
\item $a(\bullet,\bullet)$ and $b(\bullet,\bullet)$ are bounded, and the {continuity constants} \(\|a\|\) and \(\|b\|\) are independent of the scaling parameter $t$ (no boundedness of $c(\bullet,\bullet)$ with respect to $\|\bullet\|_Y$ is assumed).
    \item $a(\bullet,\bullet)$ and $c(\bullet,\bullet)$ are symmetric and positive semi-definite.
    \item Inf-sup condition: there exists $C_b>0$, independent of $t$, such that
    \begin{align}\label{eq:inf-sup.b} \sup_{\btau \in \bX\setminus\{\boldsymbol{0}\}} \frac{b(\btau,v)}{\|\btau\|_{\bX}} \geq C_b \| v\|_Y \quad \forall v \in Y.\end{align}
    \item Braess condition: there exists $C_c>0$, independent of $t$, such that
    \begin{align}\label{eq:Braess}\sup_{(\btau,v) \in \bX\times Y_c \setminus\{(\boldsymbol{0},0)\}} \frac{a(\bsigma,\btau)+b(\bsigma,v)}{\tnorm{(\btau,v)}_{\bX\times Y}} \geq C_c \| \bsigma \|_{\bX} \quad \forall \bsigma \in \bX.\end{align}
\end{itemize}
Then, for each pair $(F,G) \in \bX'\times Y'$, the problem \eqref{eq:abstract} is uniquely solvable. Moreover, there holds
\[\tnorm{(\bsigma,u)}_{\bX\times Y} \leq C\bigl[ \|F\|_{\bX'} + \|G\|_{Y'}\bigr], \]
where
\begin{equation}\label{eq:C}
C:=\max\Bigl\{\frac{2(1+\delta)}{C_c},\ \frac{2(1+\delta^{-1})}{C_b}+(1+\delta^{-1})^2\Bigl(1+\frac{\|a\|^{\frac12}}{C_b}\Bigr)^2\Bigr\},\qquad \delta:=\frac{C_c}{2\max\{1,\|b\|\}},
\end{equation}
 does not depend on the scaling parameter \(t\) nor on $c(\bullet,\bullet)$.
\end{theorem}

\begin{remark}\label{rem:kernel}
Under the assumptions of Theorem~\ref{th:abstract}, $a(\bullet,\bullet)$ is coercive (and hence inf-sup stable) on $\bZ$. Indeed, if $\bsigma\in\bZ$, then $b(\bsigma,v)=0$ for all $v\in Y$, and since $\tnorm{(\btau,v)}_{\bX\times Y}\geq\|\btau\|_{\bX}$, the Braess condition \eqref{eq:Braess} and the Cauchy--Schwarz inequality for the positive semi-definite form $a(\bullet,\bullet)$ give
\[C_c\|\bsigma\|_{\bX}\leq\sup_{\btau\in\bX\setminus\{\bzero\}}\frac{a(\bsigma,\btau)}{\|\btau\|_{\bX}}\leq \|a\|^{\frac12}a(\bsigma,\bsigma)^{\frac12},\qquad\text{that is,}\qquad a(\bsigma,\bsigma)\geq\frac{C_c^2}{\|a\|}\|\bsigma\|_{\bX}^2\quad\forall\bsigma\in\bZ.\]
Hence the inf-sup condition of $a$ on $\bZ$ required in \cite[Theorem~3.4]{correa2022continuous} is a consequence of the Braess condition. Note also that \eqref{eq:Braess} is more restrictive than in the Hilbert context \cite[(3.1)]{braess1996stability}, since the supremum of the sum $a(\bsigma,\btau)+b(\bsigma,v)$ is taken jointly over $(\btau,v)$ (as there are no orthogonal decompositions in the Banach setting).
\end{remark}

\begin{proof}[Proof of Theorem~\ref{th:abstract}]
We follow closely  the analysis in \cite[Theorem~49.15]{ern2021galerkin} and \cite[Lemma 3]{braess1996stability}. Let $(\bsigma,u)\in\bX\times Y_c$ be fixed and denote
\[ \mathsf{S}:=\sup_{(\bzero,0)\neq(\btau,v)\in{\bX\times Y_c}}\frac{\mathcal{A}(\bsigma,u;\btau,v)}{\tnorm{(\btau,v)}_{\bX\times Y}}.\]
We first prove $\mathrm{BNB~1}$ by showing that $\tnorm{(\bsigma,u)}_{\bX\times Y}\leq C\,\mathsf{S}$, for which we  consider two cases.

\smallskip
\noindent\textbf{Case~1.} Assume that \begin{align}\label{eq:case.1}\|u\|_Y+t|u|_c\leq \delta\,\|\bsigma\|_{\bX}.\end{align}
For any $(\btau,v)\in\bX\times Y_c$, the definition \eqref{eq:curl.A}, the boundedness of $b(\bullet,\bullet)$, and the Cauchy--Schwarz inequality for $c(\bullet,\bullet)$ yield
\begin{align*}
a(\bsigma,\btau)+b(\bsigma,v) & = \cA(\bsigma,u;\btau,v) - b(\btau,u)+t^2c(u,v)\\
& \leq \cA(\bsigma,u;\btau,v) + \max\{1,\|b\|\}\bigl(\|u\|_Y+t|u|_c\bigr)\tnorm{(\btau,v)}_{\bX\times Y}.
\end{align*}
Dividing by $\tnorm{(\btau,v)}_{\bX\times Y}$, taking the supremum, and using the Braess condition \eqref{eq:Braess} together with \eqref{eq:case.1} and the definition of $\delta$, we obtain
\[ C_c\|\bsigma\|_{\bX}\leq \mathsf{S}+\max\{1,\|b\|\}\,\delta\,\|\bsigma\|_{\bX} = \mathsf{S}+\frac{C_c}{2}\|\bsigma\|_{\bX},\]
that is, $\|\bsigma\|_{\bX}\leq 2C_c^{-1}\mathsf{S}$. Using \eqref{eq:case.1} once more, we conclude that
\[\tnorm{(\bsigma,u)}_{\bX\times Y}\leq (1+\delta)\|\bsigma\|_{\bX}\leq \frac{2(1+\delta)}{C_c}\,\mathsf{S}.\]

\smallskip
\noindent\textbf{Case~2.} Assume now that $\|u\|_Y+t|u|_c> \delta\,\|\bsigma\|_{\bX}$, so that
\begin{align}\label{eq:case.2}\tnorm{(\bsigma,u)}_{\bX\times Y}\leq (1+\delta^{-1})\bigl(\|u\|_Y+t|u|_c\bigr).\end{align}
The positive semi-definiteness of $a(\bullet,\bullet)$ and $c(\bullet,\bullet)$ and the definition of $\mathsf{S}$ give
\begin{align}\label{eq:case.1.b}
a(\bsigma,\bsigma)+t^2|u|_c^2=\cA(\bsigma,u;\bsigma,-u)\leq \mathsf{S}\,\tnorm{(\bsigma,-u)}_{\bX\times Y} = \mathsf{S}\,\tnorm{(\bsigma,u)}_{\bX\times Y}.\end{align}
In turn, the inf-sup condition \eqref{eq:inf-sup.b} (applied to $u\in Y_c\subseteq Y$), the identity $b(\btau,u) = \cA(\bsigma,u;\btau,0)-a(\bsigma,\btau)$, the fact that $\tnorm{(\btau,0)}_{\bX\times Y}=\|\btau\|_{\bX}$, and the Cauchy--Schwarz inequality $a(\bsigma,\btau)\leq a(\bsigma,\bsigma)^{\frac12}\|a\|^{\frac12}\|\btau\|_{\bX}$ imply
\begin{align}\label{eq:case.1.c}
C_b\|u\|_Y\leq\sup_{\bzero\neq\btau\in\bX}\frac{b(\btau,u)}{\|\btau\|_{\bX}}\leq \mathsf{S}+\|a\|^{\frac12}a(\bsigma,\bsigma)^{\frac12}.
\end{align}
Combining \eqref{eq:case.2}, \eqref{eq:case.1.c}, and \eqref{eq:case.1.b} (which bounds both $a(\bsigma,\bsigma)^{\frac12}$ and $t|u|_c$), we arrive at
\[ \tnorm{(\bsigma,u)}_{\bX\times Y}\leq \underbrace{\frac{1+\delta^{-1}}{C_b}\,\mathsf{S}}_{=:y}+\underbrace{(1+\delta^{-1})\Bigl(1+\frac{\|a\|^{\frac12}}{C_b}\Bigr)\bigl(\mathsf{S}\,\tnorm{(\bsigma,u)}_{\bX\times Y}\bigr)^{\frac12}}_{=:z}.\]
We recall that for any three positive real numbers \(x, y, z\) such that \(x\leq y+z\), there holds \(x\leq2y+\frac{z^2}{x}\) (indeed, $x^2\leq xy+xz\leq xy+\frac{x^2}{2}+\frac{z^2}{2}$; see also \cite[Equation~3.4]{braess1996stability}). Choosing $x = \tnorm{(\bsigma,u)}_{\bX\times Y}$ we readily obtain
\[\tnorm{(\bsigma,u)}_{\bX\times Y}\leq \Bigl[\frac{2(1+\delta^{-1})}{C_b}+(1+\delta^{-1})^2\Bigl(1+\frac{\|a\|^{\frac12}}{C_b}\Bigr)^2\Bigr]\mathsf{S}.\]
Combining the two cases, we arrive at $\tnorm{(\bsigma,u)}_{\bX\times Y}\leq C\,\mathsf{S}$ with $C$ as in \eqref{eq:C}, which is $\mathrm{BNB~1}$ with $C_\natural = C^{-1}$.

\smallskip
We now turn to showing \(\mathrm{BNB}~2\). Since $a(\bullet,\bullet)$ and $c(\bullet,\bullet)$ are symmetric, so is $\cA$, that is, $\cA(\bsigma,u;\btau,v) = \cA(\btau,v;\bsigma,u)$. Hence, if $(\btau,v)\in\bX\times Y_c$ is such that $\cA(\bsigma,u;\btau,v)=0$ for all $(\bsigma,u)\in\bX\times Y_c$, then $\cA(\btau,v;\bsigma,u)=0$ for all $(\bsigma,u)$, and $\mathrm{BNB~1}$ (applied with $(\btau,v)$ in place of $(\bsigma,u)$) yields $\tnorm{(\btau,v)}_{\bX\times Y}\leq C\cdot 0$, that is, $(\btau,v) = (\bzero,0)$.
Finally, if $(\bsigma,u)$ solves \eqref{eq:abstract}, then $\cA(\bsigma,u;\btau,v) = F(\btau)+G(v)\leq (\|F\|_{\bX'}+\|G\|_{Y'})\tnorm{(\btau,v)}_{\bX\times Y}$, and the stated bound follows from $\mathrm{BNB~1}$. This completes the proof.
\end{proof}
In summary, compared with \cite[Theorem~3.4]{correa2022continuous}, Theorem~\ref{th:abstract} replaces the inf-sup condition of $a$ on $\bZ$ by the Braess condition \eqref{eq:Braess} (which, by Remark~\ref{rem:kernel}, is stronger), and in return provides stability bounds that are independent of $t$ and of $c(\bullet,\bullet)$ (which, moreover, need not be defined nor bounded on the whole of $Y$), in the stronger norm \eqref{eq:triple.norm} that also controls the energy $t|u|_c$ of the perturbation term. Compared with \cite[Theorem~2.1]{boon2021robust} and \cite[Lemma~3]{braess1996stability}, it extends the robust stability result (including the case of a perturbation $c(\bullet,\bullet)$ defined only on a dense subspace $Y_c$) from Hilbert to reflexive Banach spaces.

\section{Existence and uniqueness of weak solution}\label{sec:existence}
This section  employs a fixed-point framework to split the Darcy--Forchheimer equations from the ADR system. After showing that the two uncoupled systems are well-posed, we  construct a fixed-point operator and then prove that it maps a closed ball $K$ into itself and is continuous and compact, so that the Schauder fixed-point theorem yields existence. Then, uniqueness is addressed separately. The decoupling strategy is similar to the analysis in, e.g.,  \cite{caucao2021fully,correa2023new,gatica2022p} (where the Banach fixed-point theorem is used instead).

\subsection{Unique solvability of the decoupled sub-problems}
Now, for a fixed concentration $\hat{\varphi}\in \Psi$, consider the Darcy--Forchheimer sub-system in weak form as follows
\begin{subequations}\label{eq:weak.fixed.u.hat}
    \begin{align}
    a_1(\bu,\bv) + d_1(\bu; \bu,\bv) +  b_1(\bv,p) & = F^{\hat\varphi}(\bv) \quad \forall \bv \in \bV,\label{eq:weak,1.hat.u}\\
    b_1(\bu,q) & = 0 \quad \forall q \in \rQ.\label{eq:weak,2.hat.u}
    \end{align}
\end{subequations}
We recall from \cite[Theorem~3.3]{dhpr_sinum26} the following result for the well-posedness of the system \eqref{eq:weak,1.hat.u}-\eqref{eq:weak,2.hat.u}.

\begin{theorem}[Well-posedness of  Darcy--Forchheimer equations]\label{th:continuous-DF}
Assume that   $\bf(\hat{\varphi})\in \bL^2(\Omega)$ { for a fixed  $\hat{\varphi}\in \Psi.$}  Then there exists a unique  solution $(\bu,p)\in \bV\times\rQ$ to the problem \eqref{eq:weak,1.hat.u}-\eqref{eq:weak,2.hat.u}. Furthermore, there is a positive constant $C^{\mathsf{DF}}_{c},$ independent of the model parameters $\kappa, \rF$, such that
\begin{equation*}
 \|(\bu, p )\|_{\bV\times\rQ}  \leq C^{\mathsf{DF}}_{c}\max \bigl\{\|\bf(\hat{\varphi})\|_{0,\Omega}^{\frac12},\|\bf(\hat{\varphi})\|_{0,\Omega} ,\|\bf(\hat{\varphi})\|_{0,\Omega}^2\bigr\}.
    \end{equation*}
\end{theorem}
(The square root comes from \eqref{eq:u.estimate} below; the pressure is then bounded through the inf-sup condition of $b_1$ \cite[Lemma~3.2]{dhpr_sinum26} and \eqref{eq:loc.lip}, which gives $\|p\|_{\rQ}\lesssim\|\bf(\hat\varphi)\|_{0,\Omega}+\|\bu\|_{\bV}+\|\bu\|_{\bV}^2$.)
The velocity bound requires some care, since the Forchheimer term is only \emph{cubically} (and not quadratically) monotone in $\rL^3$. We denote by $\mathcal{N}:\bV\to\bV'$ the nonlinear operator $\langle\mathcal{N}(\bu),\bv\rangle := a_1(\bu,\bv)+d_1(\bu;\bu,\bv)$, and by $\bV^0:=\{\bv\in\bV:\ \vdiv\bv = 0 \ \text{in}\ \Omega\}$ the kernel of $b_1$ (recall that $\rL^2(\Omega)\subseteq\rQ$).
\begin{lemma}[Robust monotonicity and local Lipschitz continuity]\label{lem:monotone}
For all $\bu,\bv,\bw\in\bV$ there hold
\begin{subequations}
\begin{align}
\langle\mathcal{N}(\bu)-\mathcal{N}(\bv),\bu-\bv\rangle &\geq \kappa^{-1}\|\bu-\bv\|_{0,\Omega}^2+\frac{\rF}{2}\|\bu-\bv\|_{0,3,\Omega}^3,\label{eq:monotone}\\
\langle\mathcal{N}(\bu)-\mathcal{N}(\bv),\bw\rangle &\leq \bigl(1+\|\bu\|_{\bV}+\|\bv\|_{\bV}\bigr)\|\bu-\bv\|_{\bV}\|\bw\|_{\bV}.\label{eq:loc.lip}
\end{align}\end{subequations}
Consequently, with $\varrho(s):=\min\{\frac{s^2}{4},\frac{s^3}{16}\}$ for $s\geq0$,
\begin{equation}\label{eq:rho.monotone}
\langle\mathcal{N}(\bu)-\mathcal{N}(\bv),\bu-\bv\rangle\geq \varrho\bigl(\|\bu-\bv\|_{\bV}\bigr)\qquad\forall\,\bu,\bv\in\bV \ \text{with}\ \bu-\bv\in\bV^0.
\end{equation}
\end{lemma}
\begin{proof}
For $\boldsymbol{a},\boldsymbol{b}\in\mathbb{R}^d$ one readily verifies the identity
$(|\boldsymbol{a}|\boldsymbol{a}-|\boldsymbol{b}|\boldsymbol{b})\cdot(\boldsymbol{a}-\boldsymbol{b}) = \frac12(|\boldsymbol{a}|+|\boldsymbol{b}|)|\boldsymbol{a}-\boldsymbol{b}|^2+\frac12(|\boldsymbol{a}|-|\boldsymbol{b}|)^2(|\boldsymbol{a}|+|\boldsymbol{b}|)$, which, together with $|\boldsymbol{a}-\boldsymbol{b}|\leq|\boldsymbol{a}|+|\boldsymbol{b}|$, gives $(|\boldsymbol{a}|\boldsymbol{a}-|\boldsymbol{b}|\boldsymbol{b})\cdot(\boldsymbol{a}-\boldsymbol{b})\geq\frac12|\boldsymbol{a}-\boldsymbol{b}|^3$. Integrating over $\Omega$ yields \eqref{eq:monotone}. For \eqref{eq:loc.lip} we use the pointwise bound $\bigl||\boldsymbol{a}|\boldsymbol{a}-|\boldsymbol{b}|\boldsymbol{b}\bigr|\leq(|\boldsymbol{a}|+|\boldsymbol{b}|)|\boldsymbol{a}-\boldsymbol{b}|$ and the generalised H\"older inequality ($\frac13+\frac13+\frac13=1$):
\[\langle\mathcal{N}(\bu)-\mathcal{N}(\bv),\bw\rangle\leq \kappa^{-\frac12}\|\bu-\bv\|_{0,\Omega}\,\kappa^{-\frac12}\|\bw\|_{0,\Omega}+\rF^{\frac13}\bigl(\|\bu\|_{0,3,\Omega}+\|\bv\|_{0,3,\Omega}\bigr)\rF^{\frac13}\|\bu-\bv\|_{0,3,\Omega}\,\rF^{\frac13}\|\bw\|_{0,3,\Omega}.\]
Finally, if $\bu-\bv\in\bV^0$, then $\|\bu-\bv\|_{\bV} = A+B$ with $A:=\kappa^{-\frac12}\|\bu-\bv\|_{0,\Omega}$ and $B:=\rF^{\frac13}\|\bu-\bv\|_{0,3,\Omega}$, and \eqref{eq:monotone} reads $\langle\mathcal{N}(\bu)-\mathcal{N}(\bv),\bu-\bv\rangle\geq A^2+\frac12B^3$. If $A\geq B$, then $A\geq\frac12(A+B)$ and $A^2\geq\frac14(A+B)^2$; otherwise $B>\frac12(A+B)$ and $\frac12B^3\geq\frac1{16}(A+B)^3$. This proves \eqref{eq:rho.monotone}.
\end{proof}
Note that $\varrho$ does not depend on any model parameter, but it is not quadratic: \eqref{eq:rho.monotone} cannot be improved to a bound of the form $\alpha_{\mathsf{DF}}\|\bu-\bv\|_{\bV}^2$ with $\alpha_{\mathsf{DF}}>0$ independent of $\|\bu-\bv\|_{\bV}$, because of the different homogeneity of the Darcy and Forchheimer terms. We will repeatedly use the following elementary consequence: for $s,R\geq0$,
\begin{equation}\label{eq:rho.inverse}
\varrho(s)\leq R\,s\quad\Longrightarrow\quad s\leq 4\max\{R,R^{\frac12}\}.
\end{equation}
Indeed, for $s>0$, $\varrho(s)\leq Rs$ means $\min\{\frac s4,\frac{s^2}{16}\}\leq R$, that is, $s\leq 4R$ or $s\leq4R^{\frac12}$. (Note that $\varrho$ is increasing; the maximum appears because a bound on a minimum only bounds one of its two entries. The estimate is sharp: for $R=\frac14$ and $s=2$ one has $\varrho(2)=\frac12=Rs$ and $s = 4R^{\frac12}=4\max\{R,R^{\frac12}\}$.)
In particular, testing \eqref{eq:weak,1.hat.u} with $\bv=\bu\in\bV^0$ and using \eqref{eq:rho.monotone} with $\bv=\bzero$ (note that $\mathcal{N}(\bzero)=\bzero$), together with $F^{\hat\varphi}(\bv)\leq \kappa^{\frac12}\|\bf(\hat\varphi)\|_{0,\Omega}\,\kappa^{-\frac12}\|\bv\|_{0,\Omega}\leq\|\bf(\hat\varphi)\|_{0,\Omega}\|\bv\|_{\bV}$ (where we used $\kappa\leq1$), we obtain the following robust bound (this also amends the argument leading to \cite[Equation~(3.10)]{dhpr_sinum26}):
 \begin{align}\label{eq:u.estimate}
    \|\bu\|_{\bV}\leq \mathcal{M}\bigl(\|\bf(\hat{\varphi})\|_{0,\Omega}\bigr),\end{align}
where $\mathcal{M}(s):=4\max\{s,s^{\frac12}\}$ for $s\geq0$.

    In what follows, we state and prove the unique-solvability result for the ADR equations.
We first define an auxiliary operator, and then establish the properties that will later assist in proving the parameter-robust inf-sup stability. Given $\ell\in\Psi'$, let $\vartheta_\ell\in\rH^1_{\Gamma_{\bu}}(\Omega)$ be the unique solution of the mixed Poisson problem (with $\vartheta_\ell = 0$ on $\Gamma_{\bu}$ and $\nabla\vartheta_\ell\cdot\bn = 0$ on $\Gamma_p$, in the weak sense)
\begin{equation}\label{eq:aux.poisson}
\int_\Omega\nabla\vartheta_\ell\cdot\nabla\chi = -\langle\ell,\chi\rangle\qquad\forall\chi\in\rH^1_{\Gamma_{\bu}}(\Omega),
\end{equation}
which is well-defined by the Lax--Milgram lemma, since $\Psi'\subseteq(\rH^1_{\Gamma_{\bu}}(\Omega))'$ (cf. \eqref{eq:duality}) and $\|\nabla\bullet\|_{0,\Omega}$ is a norm on $\rH^1_{\Gamma_{\bu}}(\Omega)$ ($\Gamma_{\bu}$ having positive measure). We then define
\[\cS_c:\Psi'\to\bW,\qquad \ell\mapsto\cS_c(\ell):=\nabla\vartheta_\ell.\]
{\begin{lemma}\label{lem:inf-sup:ARD}
   The operator \(\cS_c\) defined above satisfies the following properties:
    \begin{gather}\nonumber
    \cS_c \in \cL(\Psi', \bW),\quad
    \|\cS_c\|_{\cL(\Psi', \bW)}\ \leq 2\ \text{is independent of } \alpha,{\varrho_0},\quad \text{and}\\
    {{\langle{\vdiv} [\cS_c(\ell)], \psi\rangle}} = \langle \ell, \psi \rangle \quad \text{for all}\ \ell \in \Psi',\psi \in \Psi.\label{S:prop}
\end{gather}
\end{lemma}
\begin{proof}
Let $\ell\in\Psi'$. Recalling \eqref{eq:duality} with $X=\alpha^{\frac12}\rH^1_{\Gamma_{\bu}}(\Omega)$, $Y=\rL^6(\Omega)$ and $Z$ being the null space, we have $\Psi' = \alpha^{-\frac12}(\rH^1_{\Gamma_{\bu}}(\Omega))'\cap(\rL^6(\Omega))'$ with $\|\ell\|_{\Psi'} = \max\{\alpha^{-\frac12}\|\ell\|_{(\rH^1_{\Gamma_{\bu}}(\Omega))'},\|\ell\|_{0,\frac65,\Omega}\}$, where $(\rH^1_{\Gamma_{\bu}}(\Omega))'$ is normed by duality with respect to $\|\nabla\bullet\|_{0,\Omega}$, and $(\rL^6(\Omega))'$ is identified with $\rL^{\frac65}(\Omega)$. Testing \eqref{eq:aux.poisson} with $\chi=\vartheta_\ell$ gives
\[\alpha^{-\frac12}\|\cS_c(\ell)\|_{0,\Omega} = \alpha^{-\frac12}\|\nabla\vartheta_\ell\|_{0,\Omega}\leq\alpha^{-\frac12}\|\ell\|_{(\rH^1_{\Gamma_{\bu}}(\Omega))'}\leq\|\ell\|_{\Psi'}.\]
Moreover, taking $\chi\in C_0^\infty(\Omega)$ in \eqref{eq:aux.poisson} shows that $\vdiv(\nabla\vartheta_\ell) = \ell$ in the sense of distributions, and since $\ell\in\rL^{\frac65}(\Omega)$, we deduce that $\vdiv\cS_c(\ell)\in\rL^{\frac65}(\Omega)$ with $\|\vdiv\cS_c(\ell)\|_{0,\frac65,\Omega} = \|\ell\|_{0,\frac65,\Omega}\leq\|\ell\|_{\Psi'}$. In addition, \eqref{eq:aux.poisson} (with test functions vanishing on $\Gamma_{\bu}$ but arbitrary on $\Gamma_p$) encodes the condition $\nabla\vartheta_\ell\cdot\bn=0$ on $\Gamma_p$, so that $\cS_c(\ell)\in\bH_{\Gamma_p}(\vdiv_{\frac65},\Omega)$. Altogether, $\cS_c(\ell)\in\bW$ and
\[\|\cS_c(\ell)\|_{\bW} = \alpha^{-\frac12}\|\nabla\vartheta_\ell\|_{0,\Omega}+\|\ell\|_{0,\frac65,\Omega}\leq 2\|\ell\|_{\Psi'},\]
which shows that $\cS_c\in\cL(\Psi',\bW)$ (it is clearly linear) with $\|\cS_c\|_{\cL(\Psi',\bW)}\leq2$, independently of $\alpha$ and $\varrho_0$. Finally, since $\vdiv\cS_c(\ell) = \ell$ in $\rL^{\frac65}(\Omega)$ and $\Psi\subseteq\rL^6(\Omega)$, \eqref{S:prop} holds. This concludes the rest of the proof.
\end{proof}

Let us stress that the argument above only uses the Lax--Milgram lemma for the mixed Poisson problem \eqref{eq:aux.poisson}; in particular, no elliptic regularity is needed, and the standing assumptions on $\Omega$ (bounded Lipschitz domain, with $\Gamma_{\bu}$ and $\Gamma_p$ of positive measure) suffice.

Next, for a fixed advecting velocity \(\hat{\bu}\in\bV\), let us consider the following ADR system in its weak form
\begin{subequations}\label{eq:weak.hat.phi}
    \begin{align}
    a_2(\bzeta,\bw) + b_2(\bw,\varphi) + d_2^{\hat\bu}(\bw, \varphi) & = 0 \quad \forall \bw \in \bW, \label{eq:weak,3.hat.phi}\\
    b_2(\bzeta, \psi) - c_2(\varphi,\psi) & = H(\psi) \quad \forall \psi \in \Psi, \label{eq:weak,4.hat.phi}
    \end{align}
\end{subequations}
and we state the following continuous stability result.
\begin{theorem}[Well-posedness of decoupled  equations] \label{th:continuous-ARD}
 Assume that
 $m \in \rL^{\frac65}(\Omega)$. Then, for a fixed $\hat{\bu}\in \bV,$ there are positive constants \(C_c, C^c_{\alpha}\) (depending on $\alpha$), such that, whenever {\(\|\hat\bu\|_{{0,3,\Omega}}\leq  \frac{C_c}{2C^c_{\alpha}},\)} there exists a unique  solution $(\bzeta,\varphi)\in \bW\times\Psi$ to the problem \eqref{eq:weak,3.hat.phi}-\eqref{eq:weak,4.hat.phi}.  Furthermore, there is a positive constant $C^{\mathsf{ADR}}_{c}$, independent of $\varrho_0,$ such that
\begin{equation}\label{eq:bound.phi.psi}
 \tnorm{(\bzeta, \varphi)}_{\bW\times\Psi}  \leq C^{\mathsf{ADR}}_{c} \| m \|_{0,\frac65,\Omega} .
    \end{equation}
\end{theorem}

\begin{proof}
We first use Theorem~\ref{th:abstract} to show the well-posedness of the following reaction-diffusion system, which is identical to \eqref{eq:weak.hat.phi} without the advection part, i.e., one seeks \((\bzeta,\varphi)\in\bW\times \Psi\) such that
\begin{subequations}\label{eq:weak.form.reac-diff}
\begin{align}
    a_2(\bzeta,\bw) + b_2(\bw,\varphi) & = 0 \quad \forall \bw \in \bW, \label{eq:weak,3_1}\\
    b_2(\bzeta, \psi) - c_2(\varphi,\psi) & = H(\psi) \quad \forall \psi \in \Psi. \label{eq:weak,4_1}
\end{align}
\end{subequations}

\smallskip
\noindent\textbf{Boundedness of \(a_2(\bullet,\bullet),\,b_2(\bullet,\bullet)\) and \(c_2(\bullet,\bullet)\).} Observe that a use of  the {Cauchy--Schwarz} inequality gives the
 boundedness of the bilinear forms \(a_2(\bullet,\bullet)\) and \(c_2(\bullet,\bullet)\) as
 \begin{align*}
a_2(\bzeta,\btau)&=\int_{\Omega}\alpha^{-1}\bzeta\cdot\btau\leq\alpha^{-\frac12}\|\bzeta\|_{0,\Omega}\,\alpha^{-\frac12}\|\btau\|_{0,\Omega} \leq\|\bzeta\|_{\bW}\|\btau\|_{\bW}\quad \forall \bzeta,\btau\in \bW,\\
c_2(\varphi,\psi)&=\int_{\Omega}{\varrho_0}\varphi\psi\leq \varrho_0^{\frac12}\|\varphi\|_{0,\Omega}\,\varrho_0^{\frac12}\|\psi\|_{0,\Omega}\leq \varrho_0|\Omega|^{\frac23}\bigl(1+C_{\mathsf{emb}}\alpha^{-\frac12}\bigr)^2\|\varphi\|_{\Psi}\|\psi\|_{\Psi}\qquad \forall \varphi, \psi \in \Psi,
\end{align*}
where we have used \eqref{eq:L6-Psi} (as discussed in Section~\ref{sec:abstract}, the size of $\|c_2\|$ plays no role in the stability constant of Theorem~\ref{th:abstract}; in the notation of that theorem, $Y_c=Y=\Psi$, $t^2=\varrho_0\leq1$, $c(\varphi,\psi)=\int_\Omega\varphi\,\psi$ and $|\varphi|_c = \|\varphi\|_{0,\Omega}$). In order to show the boundedness of \(b_2(\bullet,\bullet)\), we follow a decomposition argument over \(\Psi.\) Let \(\psi=\psi_0+\psi_1,\) with \(\psi_0\in\rL^6(\Omega)\) and \(\psi_1\in\rH^1_{\Gamma_{\bu}}(\Omega).\) Then, integrating by parts (recall that $\psi_1=0$ on $\Gamma_{\bu}$ and $\btau\cdot\bn=0$ on $\Gamma_p$),
\(b_2(\btau,\psi)=b_2(\btau,\psi_0+\psi_1)=\langle\vdiv\btau,\psi_0\rangle-(\alpha^{-\frac12}\btau,\alpha^{\frac12}\nabla\psi_1).\)
    Next, a use of the H\"{o}lder's  with  the {Cauchy--Schwarz} inequality, {the fact that $a^2+b^2\leq(a+b)^2,$\, for any two  {nonnegative} real numbers $a,b,$} and then taking the infimum over all decompositions of \(\psi\in\Psi\) yield the desired boundedness as follows
    \begin{align*}
   b_2&(\btau,\psi) \leq \|\vdiv\btau\|_{0,\frac65,\Omega}\|\psi_0\|_{0,6,\Omega}+\|\alpha^{-\frac12}\btau\|_{0,\Omega}\|\alpha^{\frac12}\nabla\psi_1\|_{0,\Omega}\\
    &\leq\bigl(\|\vdiv\btau\|^2_{0,\frac65,\Omega}+\|\alpha^{-\frac12}\btau\|^2_{0,\Omega}\bigr)^{\frac12}\bigl(\|\psi_0\|^2_{0,6,\Omega}+\|\alpha^{\frac12}\nabla\psi_1\|^2_{0,\Omega}\bigr)^{\frac12} \\
    &\leq\bigl(\|\vdiv\btau\|_{0,\frac65,\Omega}+\|\alpha^{-\frac12}\btau\|_{0,\Omega}\bigr)\bigl(\|\psi_0\|_{0,6,\Omega}+\|\alpha^{\frac12}\nabla\psi_1\|_{0,\Omega}\bigr)=\|\btau\|_{\bW}\|\psi\|_{\Psi}\quad {\forall \btau \in \bW, \psi \in \Psi.}
\end{align*}

\smallskip
\noindent\textbf{Symmetry and positive-semidefiniteness of \(a_2(\bullet,\bullet)\) and \(c_2(\bullet,\bullet)\).} Clearly, \(a_2(\bullet,\bullet)\) and \(c_2(\bullet,\bullet)\) are symmetric, as well as positive-semidefinite bilinear forms, as  \[a_2(\bzeta,\bzeta)={\alpha^{-1}\int_\Omega \bzeta\cdot\bzeta} \geq0 \quad \text{and}\quad  c_2(\varphi,\varphi)= \varrho_0\int_{\Omega}\varphi^2\geq0 \qquad \forall \bzeta\in\bW, \varphi\in\Psi.\]

\smallskip
\noindent\textbf{Inf-sup of \(a_2(\bullet,\bullet)\) on $\bZ$.} Choosing $\btau = \bzeta$ and using the kernel characterization  definition of the $\bW$-norm, we arrive at
\[ \sup_{\btau \in \bZ\setminus\{\boldsymbol{0}\}} \frac{a_2(\bzeta,\btau)}{\|\btau\|_{\bW}} {{\geq \alpha^{-1}\frac{\int_\Omega \bzeta\cdot\bzeta}{\|\bzeta\|_{\bW}}} =}  \| \bzeta\|_{\bW} \quad \forall \bzeta \in \bZ.\]
Then, this readily
gives the inf-sup condition of $a_2(\bullet,\bullet)$ on $\bZ$ with constant 1 (here we have used that $\bzeta\in\bZ$ implies $\vdiv\bzeta=0$, so that $\|\bzeta\|_{\bW}=\alpha^{-\frac12}\|\bzeta\|_{0,\Omega}$). We recall that, by Remark~\ref{rem:kernel}, this property is in fact implied by the Braess condition established below.

\smallskip
\noindent\textbf{Inf-sup stability of \(b_2(\bullet,\bullet)\).}
Let $\psi\in\Psi$. Since $\cS_c(\Psi')\subseteq\bW$ (cf. Lemma~\ref{lem:inf-sup:ARD}), restricting the supremum to elements of the form $\btau = \cS_c(\ell)$, with $\ell\in\Psi'$, and using \eqref{S:prop}, we first obtain
\begin{align*}
\sup_{\bzero\neq\btau \in \bW} \frac{\int_\Omega \psi\,\vdiv \btau}{\|\btau\|_\bW} \geq \sup_{0\neq\ell \in \Psi'} \frac{\langle\vdiv [\cS_c(\ell)], \psi\rangle}{\|\cS_c(\ell)\|_\bW} = \sup_{0\neq\ell \in \Psi'} \frac{\langle \ell, \psi \rangle}{\|\cS_c(\ell)\|_\bW}.
\end{align*}
Next, the boundedness $\|\cS_c(\ell)\|_{\bW}\leq\|\cS_c\|_{\cL(\Psi',\bW)}\|\ell\|_{\Psi'}$ (for those $\ell$ with $\langle\ell,\psi\rangle\geq0$, which suffices for the supremum) yields
\begin{align*}
\sup_{0\neq\ell \in \Psi'} \frac{\langle \ell, \psi \rangle}{\|\cS_c(\ell)\|_\bW}\geq \|\cS_c\|^{-1}_{\cL(\Psi', \bW)} \sup_{0\neq\ell \in \Psi'} \frac{\langle \ell, \psi \rangle}{\|\ell\|_{\Psi'}} = \|\cS_c\|^{-1}_{\cL(\Psi', \bW)}\|\psi\|_{\Psi},
\end{align*}
where the last equality is a consequence of the Hahn--Banach theorem (the norm of an element of a normed space is attained by duality, see e.g. \cite[Corollary~1.4]{brezis2011functional}). Hence
\[\sup_{\bzero\neq\btau \in \bW} \frac{\int_\Omega \psi\,\vdiv \btau}{\|\btau\|_\bW}\geq C_{b_2} \|\psi\|_{\Psi}\qquad\forall\psi\in\Psi,\]
where $C_{b_2} = \|\cS_c\|^{-1}_{\cL(\Psi', \bW)}\geq\frac12$ is independent of \(\alpha,\varrho_0\), and hence, the inf-sup condition is satisfied.

\smallskip
\noindent\textbf{Braess condition.} Now, fixing  \(\btau=\bzeta,\psi=0\), and observing that \(\tnorm{(\bzeta,0)}_{\bW\times \Psi}=\|\bzeta\|_{\bW}\), we arrive at
\begin{align}\label{eq:Braess.1st}
\sup_{(\bzero,0)\neq(\btau,\psi)\in\bW\times\Psi}\frac{a_2(\bzeta,\btau)+b_2(\bzeta,\psi)}{\tnorm{(\btau,\psi)}_{\bW\times\Psi}} \geq \alpha^{-1}\frac{\int_\Omega \bzeta\cdot \bzeta}{\tnorm{(\bzeta,0)}_{\bW\times\Psi}}=\alpha^{-1}\frac{\|\bzeta\|^2_{0,\Omega}}{\|\bzeta\|_{\bW}}.
\end{align}
 Next, taking \(\btau = \bzero,\) and \(\psi=\mathrm{sgn}(\vdiv\bzeta)|\vdiv\bzeta|^{\frac15}\in\rL^{6}(\Omega)\),
we evaluate the following quantities
\begin{align*}
   b_2(\boldsymbol{\zeta}, \psi)
  &= \int_{\Omega} |\operatorname{div} \boldsymbol{\zeta}|^{\frac{6}{5}}
  = \|\operatorname{div} \boldsymbol{\zeta}\|_{0,\frac{6}{5},\Omega}^{\frac{6}{5}},
  \quad \quad \|\psi\|^6_{0,6,\Omega}
  = \int_\Omega|\vdiv\bzeta|^{\frac65}=\|\operatorname{div} \boldsymbol{\zeta}\|^{\frac65}_{0,\frac65,\Omega},
  \\
    \tnorm{(\btau, \psi)}_{\boldsymbol{W} \times \Psi}
  &= \tnorm{(\boldsymbol{0}, \psi)}_{\boldsymbol{W} \times \Psi}
  = \|\psi\|_{\Psi} + \varrho_0^{\frac12}\|\psi\|_{0,\Omega}
  \leq \bigl(1+\varrho_0^{\frac12}|\Omega|^{\frac13}\bigr)\|\psi\|_{0,6,\Omega}\leq 2\|\psi\|_{0,6,\Omega},
\end{align*}
where {the first inequality follows from \eqref{eq:Psi.norm} (taking $\chi=0$) and the $\rL^2$-$\rL^6$ embedding; the last one from $\varrho_0\leq 1$ and, for the sake of simplicity, $|\Omega|\leq1$ (otherwise, the factor $\frac12$ below is replaced by $(1+|\Omega|^{\frac13})^{-1}$)}.
 Furthermore using {the {above} expressions} for \(b_2(\boldsymbol{\zeta}, \psi),\,\|\psi\|_{0,6,\Omega}\) and \(\tnorm{(\btau, \psi)}_{\boldsymbol{W} \times \Psi}\), we immediately obtain
\begin{align}\label{eq:Braess.2nd} \sup_{(\bzero,0)\neq(\btau,\psi)\in\bW\times\Psi}\frac{a_2(\bzeta,\btau)+b_2(\bzeta,\psi)}{\tnorm{(\btau,\psi)}_{\bW\times\Psi}} \geq{\frac12} \frac{\|\vdiv \bzeta\|^{\frac65}_{0,\frac65,\Omega}}{\|\vdiv \bzeta\|^{\frac15}_{0,\frac65,\Omega}}={\frac12}\|\vdiv \bzeta\|_{0,\frac65,\Omega}.
\end{align}
Now, combining \eqref{eq:Braess.1st} and \eqref{eq:Braess.2nd}, we arrive at
\begin{align}\label{eq:Braess.3rd}
\sup_{(\bzero,0)\neq(\btau,\psi)\in\bW\times\Psi}\frac{a_2(\bzeta,\btau)+b_2(\bzeta,\psi)}{\tnorm{(\btau,\psi)}_{\bW\times\Psi}}\geq\max\bigl\{\alpha^{-1}{\frac{\|\bzeta\|^2_{0,\Omega}}{\|\bzeta\|_{\bW}},{\frac12}\|\vdiv \bzeta\|_{0,\frac65,\Omega}}\bigr\}.
\end{align}
For $\bzeta\neq\bzero$, set \(\varpi:=\alpha^{-\frac12}\frac{\|\bzeta\|_{0,\Omega}}{\|\bzeta\|_{\bW}}\in[0,1]\), and note that, by definition of the $\bW$-norm, $\frac{\|\vdiv\bzeta\|_{0,\frac65,\Omega}}{\|\bzeta\|_{\bW}} = 1-\varpi$. Then, \eqref{eq:Braess.3rd} shows
\begin{gather*}\max\bigl\{{\alpha^{-1}\frac{\|\bzeta\|^2_{0,\Omega}}{\|\bzeta\|_{\bW}},\frac12\|\vdiv \bzeta\|_{0,\frac65,\Omega}}\bigr\}=\max\bigl\{\varpi^2,\tfrac12(1-\varpi)\bigr\} \|\bzeta\|_{\bW}\geq\gamma_c\|\bzeta\|_{\bW}, \\
\qquad\text{with}\quad
\gamma_c:=\min_{\varpi\in[0,1]}\max\bigl\{\varpi^2,\tfrac12(1-\varpi)\bigr\}=\frac14
\end{gather*}
(the minimum being attained at $\varpi=\frac12$, where $\varpi^2=\frac12(1-\varpi)$). We, therefore, conclude from \eqref{eq:Braess.3rd} that \[\sup_{(\bzero,0)\neq(\btau,\psi)\in\bW\times\Psi}\frac{a_2(\bzeta,\btau)+b_2(\bzeta,\psi)}{\tnorm{(\btau,\psi)}_{\bW\times\Psi}}\geq \gamma_c \|\bzeta\|_{\bW}\qquad\forall\bzeta\in\bW.\]
An appeal to Theorem~\ref{th:abstract} now yields the well-posedness of \eqref{eq:weak.form.reac-diff}, and additionally, we obtain the following inf-sup type estimate for $\cA$
\begin{align}\label{eq:A.stab}
\sup_{(\bzero,0)\neq(\tau,\psi)\in\bW\times\Psi}\frac{\cA(\bzeta,\varphi;\btau,\psi)}{\tnorm{(\btau,\psi)}_{\bW\times \Psi}}\geq C_c\tnorm{(\bzeta,\varphi)}_{\bW\times \Psi},
\end{align}
where $C_c:=C^{-1}$, with $C$ given by \eqref{eq:C} with $\|a\|=\|b\|=1$, $C_b = C_{b_2}$, and Braess constant $\gamma_c$; in particular, $C_c$ is independent of $\alpha$, $\varrho_0$ (and of $\kappa,\rF$).

Next, in order to show the well-posedness of the uncoupled ADR equations, for a fixed \(\hat{\bu}\in\bV,\) we consider another bilinear form \(\mathcal{A}^{\mathsf{ADR}}:\bW\times\Psi\to \mathbb{R}\) defined by
\begin{align}\label{eq:A:ARD}\mathcal{A}^{\mathsf{ADR}}(\bzeta,\varphi;\btau,\psi):=\mathcal{A}(\bzeta,\varphi;\btau,\psi)+d_2^{\hat{\bu}}(\btau,\varphi),\quad \forall (\bzeta,\varphi), (\btau,\psi)\in\bW\times \Psi.\end{align}
Thus,  showing the well-posedness of the system~\eqref{eq:weak,3}-\eqref{eq:weak,4} boils down to analysing the problem of finding \((\bzeta,\varphi)\in\bW\times \Psi\) satisfying \[\mathcal{A}^{\mathsf{ADR}}(\bzeta,\varphi;\btau,\psi)=H(\psi),\quad \forall {\btau\in\bW, \text{and }} \psi\in \Psi.\]

Note that, by the generalised Hölder's inequality with {1/3+1/2+1/6=1}, the definition of $\|\bullet\|_{\bW}$, and \eqref{eq:L6-Psi}, the following estimate readily follows:\begin{align}\label{eq:d2.bound}
d_2^{\hat{\bu}}(\btau,\varphi)\leq \alpha^{-1}\|\hat\bu\|_{0,3,\Omega}\|\btau\|_{0,\Omega}\|\varphi\|_{0,6,\Omega}\leq C^c_{\alpha}\|\hat{\bu}\|_{0,3,\Omega}\,\tnorm{(\btau,\psi)}_{\bW\times\Psi}\,\tnorm{(\bzeta,\varphi)}_{\bW\times\Psi},\end{align} where {$C^c_{\alpha}:={{\alpha^{-\frac12}}}(1+C_{\mathsf{emb}}\alpha^{-\frac12}).$}
Invoking the bound of \(d_2^{\hat{\bu}}(\bullet,\bullet)\) (cf. \eqref{eq:d2.bound}), it follows from \eqref{eq:A.stab} and \eqref{eq:A:ARD} that
\[\sup_{(\bzero,0)\neq(\btau,\psi)\in {\bW\times\Psi}}\frac{\mathcal{A}^{\mathsf{ADR}}(\bzeta,\varphi;\btau,\psi)}{\tnorm{(\btau,\psi)}_{\bW\times\Psi}}\geq\bigl[C_c-C^c_{\alpha}{\|\hat{\bu}\|_{0,3,\Omega}}\bigr]\tnorm{(\bzeta,\varphi)}_ {\bW\times\Psi}, \qquad  \forall(\bzeta,\varphi)\in \bW\times \Psi.\]
Hence, under the assumption that { for each $\hat{\bu} \in \bV$ such that  {$\|\hat{\bu}\|_{{0,3,\Omega}}\leq \frac{C_c}{2C^c_{\alpha}}$}, we arrive at
\begin{equation}\label{eq:ARD}
\sup_{(\bzero,0)\neq(\btau,\psi)\in {\bW\times\Psi}}\frac{\mathcal{A}^{\mathsf{ADR}}(\bzeta,\varphi;\btau,\psi)}{\tnorm{(\btau,\psi)}_ {\bW\times\Psi}}\geq\frac{C_c}{2} \tnorm{(\bzeta,\varphi)}_ {\bW\times\Psi}.
\end{equation}
Therefore,} the ADR system (cf. \eqref{eq:weak,3}-\eqref{eq:weak,4}) is well-posed.
Additionally, using \eqref{eq:A:ARD} in \eqref{eq:ARD}, and using the H\"{o}lder's inequality to see that \(H(\psi)\leq\|m\|_{0,\frac65,\Omega}\|\psi\|_{0,6,\Omega},\) we obtain the continuous dependence on the data as follows: \begin{align}\label{eq:C.c.ard}
\tnorm{(\bzeta,\varphi)}_{\bW\times \Psi}\leq C^{\mathsf{ADR}}_{c}\|m\|_{0,\frac65,\Omega},
\end{align} with {$C^{\mathsf{ADR}}_{c}:={2}{C_c}^{-1}(1+C_{\mathsf{emb}}\alpha^{-\frac12})$} being independent of~\(\varrho_0\). This concludes the rest of the proof.
\end{proof}

\subsection{Fixed-point analysis}
Now, define the following  map
\[ \cS^{\mathsf{DF}}:\Psi\to \bV \times \rQ, \quad \hat{\varphi} \mapsto \cS^{\mathsf{DF}}(\hat{\varphi})= (\cS^{\mathsf{DF}}_1(\hat{\varphi}), \cS^{\mathsf{DF}}_2(\hat{\varphi})):= (\bu, p), \]
where $( \bu,p) \in \bV \times \rQ$ is the unique solution of the Darcy--Forchheimer equations {{corresponding to a fixed} $\hat\varphi\in \ \Psi$}, as stated in Theorem~\ref{th:continuous-DF}, {and $\cS_1^{\mathsf{DF}}:\Psi\to\bV$,  $\cS_2^{\mathsf{DF}}:\Psi\to\rQ.$}
Further, {for a fixed $\hat\bu\in\bV$ with $\|\hat\bu\|_{0,3,\Omega}\leq r_{\mathsf{ADR}}:=\frac{C_c}{2C^c_\alpha}$,} set the solution operator associated with the mixed ADR equations as
\begin{gather*}
\cS^{\mathsf{ADR}}: \bV_{\mathsf{ADR}}:=\{\hat\bu\in\bV:\ \|\hat\bu\|_{0,3,\Omega}\leq r_{\mathsf{ADR}}\} \to  \bW\times \Psi,\\
\hat{\bu} \mapsto \cS^{\mathsf{ADR}}(\hat{\bu})=(\cS^{\mathsf{ADR}}_1(\hat{\bu}),\cS^{\mathsf{ADR}}_2(\hat{\bu})):= ( \bzeta, \varphi),
\end{gather*}
where $(\bzeta, \varphi)$ is the unique solution of the ADR equations, as stated in Theorem~\ref{th:continuous-ARD}. Clearly, $(\bu,p,\bzeta,\varphi)$ solves \eqref{eq:weak} if and only if $\varphi$ is a fixed point of
\begin{equation}\label{Operator-S.V->V}
\mathcal{T}:=\cS^{\mathsf{ADR}}_2\circ\cS^{\mathsf{DF}}_1,
\end{equation}
and $(\bu,p) = \cS^{\mathsf{DF}}(\varphi)$, $\bzeta = \cS^{\mathsf{ADR}}_1(\bu)$.  Owing to Lemma~\ref{lem:monotone},  $\cS^{\mathsf{DF}}_1$ is only H\"older continuous with respect to the $\bV$-norm, see Lemma~\ref{lemma:continuity-S} below, so a contraction argument is not available. We, therefore, separate existence (via the Schauder fixed-point theorem) from uniqueness.

Recall that the forcing term satisfies the Lipschitz condition \eqref{eq:f.lipschitz}. Now, set
\[ \mathsf{M}_0:=\|\bf(\bzero)\|_{0,\Omega}+\alpha_1\bigl(1+C_{\mathsf{emb}}\alpha^{-\frac12}\bigr)C^{\mathsf{ADR}}_{c}\|m\|_{0,\frac65,\Omega},\qquad K:=\bigl\{\hat\varphi\in\Psi:\ \|\hat\varphi\|_{\Psi}\leq C^{\mathsf{ADR}}_{c}\|m\|_{0,\frac65,\Omega}\bigr\}.\]
\begin{lemma}[Ball mapping property]\label{lemma:ball}
Under the small data assumption
\begin{equation}\label{data-bounded-map-closed}
\rF^{-\frac13}\,\mathcal{M}(\mathsf{M}_0) = 4\rF^{-\frac13}\max\bigl\{\mathsf{M}_0,\mathsf{M}_0^{\frac12}\bigr\}\leq r_{\mathsf{ADR}}=\frac{C_c}{2C^c_\alpha},
\end{equation}
the map $\mathcal{T}$ is well-defined on $K$, and $\mathcal{T}(K)\subseteq K$. Moreover, $\|\cS^{\mathsf{DF}}_1(\hat\varphi)\|_{\bV}\leq\mathcal{M}(\mathsf{M}_0)$ for all $\hat\varphi\in K$.
\end{lemma}
\begin{proof}
Let $\hat\varphi\in K$ and $\hat\bu:=\cS^{\mathsf{DF}}_1(\hat\varphi)$. By \eqref{eq:lip.bound.f} (with $|\Omega|\leq1$) and \eqref{eq:L6-Psi}, $\|\bf(\hat\varphi)\|_{0,\Omega}\leq\|\bf(\bzero)\|_{0,\Omega}+\alpha_1(1+C_{\mathsf{emb}}\alpha^{-\frac12})\|\hat\varphi\|_{\Psi}\leq \mathsf{M}_0$, and since $\mathcal{M}$ is nondecreasing, \eqref{eq:u.estimate} gives $\|\hat\bu\|_{\bV}\leq\mathcal{M}(\mathsf{M}_0)$. Hence $\|\hat\bu\|_{0,3,\Omega}\leq\rF^{-\frac13}\|\hat\bu\|_{\bV}\leq r_{\mathsf{ADR}}$ by \eqref{data-bounded-map-closed}, that is, $\hat\bu\in\bV_{\mathsf{ADR}}$, and Theorem~\ref{th:continuous-ARD} implies that $\mathcal{T}(\hat\varphi) = \cS^{\mathsf{ADR}}_2(\hat\bu)$ is well-defined with $\|\mathcal{T}(\hat\varphi)\|_{\Psi}\leq\tnorm{\cS^{\mathsf{ADR}}(\hat\bu)}_{\bW\times\Psi}\leq C^{\mathsf{ADR}}_c\|m\|_{0,\frac65,\Omega}$, i.e., $\mathcal{T}(\hat\varphi)\in K$.
\end{proof}

 \begin{lemma}[Continuity of the solution operators]\label{lemma:continuity-S}
(a) For all $\varphi_1,\varphi_2\in\Psi$, the velocities $\bu_i:=\cS^{\mathsf{DF}}_1(\varphi_i)$ satisfy, with $\delta:=\kappa^{\frac12}\alpha_1\|\varphi_1-\varphi_2\|_{0,\Omega}$,
 \begin{equation}\label{continuity-S}
 \kappa^{-\frac12}\|\bu_1-\bu_2\|_{0,\Omega}\leq\delta,\qquad \rF^{\frac13}\|\bu_1-\bu_2\|_{0,3,\Omega}\leq (2\delta^2)^{\frac13},\qquad\text{and hence,}\;\; \|\bu_1-\bu_2\|_{\bV}\leq\delta+(2\delta^2)^{\frac13}.
 \end{equation}
 (b) For all $\bu_1,\bu_2\in\bV_{\mathsf{ADR}}$, with $(\bzeta_i,\varphi_i):=\cS^{\mathsf{ADR}}(\bu_i)$,
 \begin{equation}\label{eq:Lip.ARD}
 \tnorm{(\bzeta_1-\bzeta_2,\varphi_1-\varphi_2)}_{\bW\times\Psi}\leq L_{\mathsf{ADR}}\|\bu_1-\bu_2\|_{0,3,\Omega},\qquad L_{\mathsf{ADR}}:=\frac{2C^c_\alpha C^{\mathsf{ADR}}_c}{C_c}\|m\|_{0,\frac65,\Omega}.
 \end{equation}
 \end{lemma}
\begin{proof}
(a) Subtracting   \eqref{eq:weak.fixed.u.hat} satisfied by $(\bu_i,p_i)$, $i=1,2$, testing with $\bv = \bu_1-\bu_2\in\bV^0$ (so that the pressure terms vanish), and using \eqref{eq:monotone} and \eqref{eq:f.lipschitz}, we obtain, with $A:=\kappa^{-\frac12}\|\bu_1-\bu_2\|_{0,\Omega}$ and $B:=\rF^{\frac13}\|\bu_1-\bu_2\|_{0,3,\Omega}$,
\[A^2+\frac12B^3\leq\int_\Omega(\bf(\varphi_1)-\bf(\varphi_2))\cdot(\bu_1-\bu_2)\leq\kappa^{\frac12}\alpha_1\|\varphi_1-\varphi_2\|_{0,\Omega}\,A = \delta A.\]
Hence, $A\leq\delta$ and $B^3\leq2\delta A\leq2\delta^2$, which is \eqref{continuity-S} (recall that $\vdiv(\bu_1-\bu_2)=0$).

(b) Subtracting   \eqref{eq:weak.hat.phi} with $\hat\bu = \bu_1$ and $\hat\bu = \bu_2$, we see that $(\bzeta_1-\bzeta_2,\varphi_1-\varphi_2)$ satisfies
\[\cA^{\mathsf{ADR}}_{\bu_1}(\bzeta_1-\bzeta_2,\varphi_1-\varphi_2;\btau,\psi) = -d_2^{\bu_1-\bu_2}(\btau,\varphi_2)\qquad\forall(\btau,\psi)\in\bW\times\Psi,\]
where $\cA^{\mathsf{ADR}}_{\bu_1}$ denotes the form \eqref{eq:A:ARD} with $\hat\bu=\bu_1$. Since $\|\bu_1\|_{0,3,\Omega}\leq r_{\mathsf{ADR}}$, the inf-sup estimate \eqref{eq:ARD} applies, and together with \eqref{eq:d2.bound} and \eqref{eq:bound.phi.psi} there holds
\[\frac{C_c}{2}\tnorm{(\bzeta_1-\bzeta_2,\varphi_1-\varphi_2)}_{\bW\times\Psi}\leq C^c_\alpha\|\bu_1-\bu_2\|_{0,3,\Omega}\tnorm{(\bzeta_2,\varphi_2)}_{\bW\times\Psi}\leq C^c_\alpha C^{\mathsf{ADR}}_c\|m\|_{0,\frac65,\Omega}\|\bu_1-\bu_2\|_{0,3,\Omega}.\]
\end{proof}
We stress that \eqref{eq:Lip.ARD} does not require any restriction relating $\alpha$ and $\rF$, and that \eqref{continuity-S} expresses a H\"older (rather than Lipschitz) continuity of $\cS^{\mathsf{DF}}_1$ in the $\bV$-norm. For the uniqueness of solutions we will use a duality argument, which requires the following regularity property of the mixed Poisson problem. We stress that it only concerns the domain $\Omega$ and the boundary partition $\partial\Omega=\overline{\Gamma_{\bu}}\cup\overline{\Gamma_p}$, and not the solution of \eqref{eq:weak}.

\smallskip
\noindent\textbf{Assumption (R).} There exists $q\in(2,3]$, with $q=3$ if $d=3$, such that for every $g\in(\rW^{1,q'}_{\Gamma_{\bu}}(\Omega))'$, $\frac1q+\frac1{q'}=1$ (where $\rW^{1,q'}_{\Gamma_{\bu}}(\Omega)$ is normed by $\|\nabla\bullet\|_{0,q',\Omega}$), the solution $\vartheta\in\rH^1_{\Gamma_{\bu}}(\Omega)$ of $\int_\Omega\nabla\vartheta\cdot\nabla\chi = \langle g,\chi\rangle$ for all $\chi\in\rH^1_{\Gamma_{\bu}}(\Omega)$ belongs to $\rW^{1,q}_{\Gamma_{\bu}}(\Omega)$ and $\|\nabla\vartheta\|_{0,q,\Omega}\leq C_{\mathsf{R}}\|g\|_{(\rW^{1,q'}_{\Gamma_{\bu}}(\Omega))'}$.

\smallskip
\noindent In 2D, (R) holds for some $q>2$ on bounded Lipschitz domains whose boundary partition is regular in the sense of Gr\"oger (e.g., polygonal domains where $\Gamma_{\bu}$ and $\Gamma_p$ meet at finitely many points) \cite{groger1989w1p}. In 3D, the requirement $q=3$ is a genuine restriction on the geometry; it holds in a number of relevant polyhedral configurations with mixed boundary conditions, see \cite{disser2015optimal} and the references therein. Given $q$ as in (R), we set $s:=\frac{2q}{q-2}$ (so that $\frac12+\frac1s+\frac1q=1$, and $s=6$ if $d=3$), denote by $C_q$ the constant of the Sobolev embedding $\|\chi\|_{0,\frac{3q}{2q-3},\Omega}\leq C_q\|\nabla\chi\|_{0,q',\Omega}$ for all $\chi\in\rW^{1,q'}_{\Gamma_{\bu}}(\Omega)$ (note that $\frac{3q}{2q-3}\geq 3$, and it is the Sobolev exponent of $\rW^{1,q'}$ if $d=3$), and define
\begin{gather*}
r_{\mathsf{dual}}:=\frac{\alpha}{2\max\{C_{\mathsf{emb}},C_{\mathsf{R}}C_q\}},\\
\mathsf{M}_s:=\begin{cases}\bigl(1+C_{\mathsf{emb}}\alpha^{-\frac12}\bigr)C^{\mathsf{ADR}}_c\|m\|_{0,\frac65,\Omega}&\text{if } d=3,\\[1ex] C_{\mathsf{emb},s}\,\alpha^{-1}\bigl(\alpha^{\frac12}+r_{\mathsf{ADR}}(1+C_{\mathsf{emb}}\alpha^{-\frac12})\bigr)C^{\mathsf{ADR}}_c\|m\|_{0,\frac65,\Omega}&\text{if } d=2,\end{cases}
\end{gather*}
where, if $d=2$, $C_{\mathsf{emb},s}$ is the constant of the Sobolev embedding $\|\chi\|_{0,s,\Omega}\leq C_{\mathsf{emb},s}\|\nabla\chi\|_{0,\Omega}$ for all $\chi\in\rH^1_{\Gamma_{\bu}}(\Omega)$.

As a consequence of the previous analysis, we can state the following main result.
\begin{theorem}[Well-posedness of the  coupled continuous problem]\label{th:wellposedness}
Assume that \eqref{data-bounded-map-closed} holds. Then,  \eqref{eq:weak} has at least one solution $(\bu,p,\bzeta,\varphi) \in \bV \times \rQ\times \bW \times \Psi$ with $\varphi\in K$, and any such solution satisfies
\begin{equation}\label{estimate su}
    \|\bu\|_{\bV}\leq\mathcal{M}(\mathsf{M}_0),\qquad \|p\|_{\rQ}\leq C^{\mathsf{DF}}_c\max\bigl\{\mathsf{M}_0^{\frac12},\mathsf{M}_0,\mathsf{M}_0^2\bigr\},\qquad \tnorm{(\bzeta,\varphi)}_{\bW\times\Psi}\leq C^{\mathsf{ADR}}_{c}\|m\|_{0,\frac65,\Omega}.
\end{equation}
If, in addition, Assumption (R) holds and the data satisfy
\begin{equation}\label{eq:uniqueness.condition}
\rF^{-\frac13}\mathcal{M}(\mathsf{M}_0)\leq r_{\mathsf{dual}}\qquad\text{and}\qquad L_*:=4\,C_{\mathsf{R}}C_q\,\kappa\,\alpha_1\,\alpha^{-1}\,\mathsf{M}_s<1,
\end{equation}
then \eqref{eq:weak} has exactly one solution with concentration in $K$. Here, the constants $\mathcal{M}$, $C^{\mathsf{DF}}_c$, $C_c$ and $C^{\mathsf{ADR}}_c$ do not depend on the model parameters $\kappa,\rF,\varrho_0$.
\end{theorem}
\begin{proof}
For the existence we apply the Schauder fixed-point theorem to $\mathcal{T}$ on $K$, regarded as a subset of  $\rL^2(\Omega)$. The set $K$ is nonempty, convex, bounded, and closed in $\rL^2(\Omega)$. Indeed, let $(\varphi_n)\subseteq K$ with $\varphi_n\to\varphi$ in $\rL^2(\Omega)$. By \eqref{eq:L6-Psi}, $(\varphi_n)$ is bounded in $\rL^6(\Omega)$, and hence a subsequence converges weakly in $\rL^6(\Omega)$; its weak limit must coincide with $\varphi$ (since $\rL^6(\Omega)\hookrightarrow\rL^2(\Omega)$ continuously). Moreover, as a set $\Psi=\rL^6(\Omega)$, and $\|\bullet\|_{\Psi}$ is a convex function which is continuous with respect to $\|\bullet\|_{0,6,\Omega}$ (since $\|\psi\|_{\Psi}\leq\|\psi\|_{0,6,\Omega}$, taking $\chi=0$ in \eqref{eq:Psi.norm}), hence weakly lower semicontinuous in $\rL^6(\Omega)$. Therefore $\|\varphi\|_{\Psi}\leq\liminf_{n}\|\varphi_n\|_{\Psi}\leq C^{\mathsf{ADR}}_c\|m\|_{0,\frac65,\Omega}$, that is, $\varphi\in K$.

By Lemma~\ref{lemma:ball}, $\mathcal{T}(K)\subseteq K$. The map $\mathcal{T}:K\to K$ is continuous with respect to the $\rL^2(\Omega)$-topology: if $\varphi_n\to\varphi$ in $\rL^2(\Omega)$, then \eqref{continuity-S} gives $\cS^{\mathsf{DF}}_1(\varphi_n)\to\cS^{\mathsf{DF}}_1(\varphi)$ in $\bL^3(\Omega)$, and \eqref{eq:Lip.ARD} together with \eqref{eq:L6-Psi} gives $\mathcal{T}(\varphi_n)\to\mathcal{T}(\varphi)$ in $\rL^6(\Omega)$, hence in $\rL^2(\Omega)$. Finally, $\mathcal{T}(K)$ is relatively compact in $\rL^2(\Omega)$: if $(\bzeta,\varphi)=\cS^{\mathsf{ADR}}(\hat\bu)$ with $\hat\bu\in\bV_{\mathsf{ADR}}$, taking $\btau\in [C_0^\infty(\Omega)]^d$ in \eqref{eq:weak,3.hat.phi} shows that $\nabla\varphi = \alpha^{-1}(\bzeta+\hat\bu\varphi)$ in the sense of distributions, and therefore
\begin{equation}\label{eq:H1.phi}
\|\nabla\varphi\|_{0,\Omega}\leq\alpha^{-1}\bigl(\|\bzeta\|_{0,\Omega}+\|\hat\bu\|_{0,3,\Omega}\|\varphi\|_{0,6,\Omega}\bigr)\leq \alpha^{-1}\bigl(\alpha^{\frac12}+r_{\mathsf{ADR}}(1+C_{\mathsf{emb}}\alpha^{-\frac12})\bigr)C^{\mathsf{ADR}}_c\|m\|_{0,\frac65,\Omega},
\end{equation}
so that $\mathcal{T}(K)$ is bounded in $\rH^1(\Omega)$, and the Rellich--Kondrachov theorem gives the claim. (In addition, testing \eqref{eq:weak,3.hat.phi} with smooth $\btau$ such that $\btau\cdot\bn=0$ on $\Gamma_p$ and integrating by parts shows that $\varphi = 0$ on $\Gamma_{\bu}$, so that $\mathcal{T}(K)\subseteq\rH^1_{\Gamma_{\bu}}(\Omega)$.) Hence $\mathcal{T}$ has a fixed point $\varphi\in K$, and $(\bu,p) := \cS^{\mathsf{DF}}(\varphi)$, $\bzeta:=\cS^{\mathsf{ADR}}_1(\bu)$ solve \eqref{eq:weak}. The bounds \eqref{estimate su} follow from Lemma~\ref{lemma:ball}, Theorem~\ref{th:continuous-DF}, and Theorem~\ref{th:continuous-ARD}.

Regarding uniqueness, let $(\bu,p,\bzeta,\varphi)$ and $(\bu',p',\bzeta',\varphi')$ be two solutions of \eqref{eq:weak} with $\varphi,\varphi'\in K$, and set $\varepsilon:=\varphi'-\varphi$. By Lemma~\ref{lemma:ball} and \eqref{eq:uniqueness.condition}, $\bu,\bu'\in\bV_{\mathsf{ADR}}$ and $\|\bu'\|_{0,3,\Omega}\leq\rF^{-\frac13}\mathcal{M}(\mathsf{M}_0)\leq r_{\mathsf{dual}}$. Since $\varphi=\mathcal{T}(\varphi)$ and $\varphi'=\mathcal{T}(\varphi')$, the existence part of the proof shows that $\varphi,\varphi'\in\rH^1_{\Gamma_{\bu}}(\Omega)$ and that $\bzeta = \alpha\nabla\varphi-\bu\varphi$, $\bzeta' = \alpha\nabla\varphi'-\bu'\varphi'$; together with \eqref{eq:L6-Psi} (if $d=3$) or \eqref{eq:H1.phi} (if $d=2$), this gives $\|\varphi\|_{0,s,\Omega}\leq\mathsf{M}_s$. We proceed in three steps.

\emph{Step 1 (primal form of the difference).} For $\chi\in\rH^1_{\Gamma_{\bu}}(\Omega)\subseteq\Psi$, testing \eqref{eq:weak,4} with $\psi=\chi$ and integrating by parts (recall that $\chi=0$ on $\Gamma_{\bu}$ and $\bzeta\cdot\bn=0$ on $\Gamma_p$) gives $\int_\Omega\bzeta\cdot\nabla\chi+\varrho_0\int_\Omega\varphi\,\chi = \int_\Omega m\,\chi$, and similarly for $(\bzeta',\varphi')$. Since $\bzeta'-\bzeta = \alpha\nabla\varepsilon-\bu'\varepsilon-(\bu'-\bu)\varphi$, subtracting yields
\begin{equation}\label{eq:primal.diff}
\int_\Omega(\alpha\nabla\varepsilon-\bu'\varepsilon)\cdot\nabla\chi+\varrho_0\int_\Omega\varepsilon\,\chi = \int_\Omega\varphi\,(\bu'-\bu)\cdot\nabla\chi\qquad\forall\chi\in\rH^1_{\Gamma_{\bu}}(\Omega).
\end{equation}

\emph{Step 2 (adjoint problem).} Consider the problem of finding $z\in\rH^1_{\Gamma_{\bu}}(\Omega)$ such that
\begin{equation}\label{eq:adjoint}
\alpha\int_\Omega\nabla z\cdot\nabla\chi-\int_\Omega\chi\,\bu'\cdot\nabla z+\varrho_0\int_\Omega z\,\chi = \int_\Omega\varepsilon\,\chi\qquad\forall\chi\in\rH^1_{\Gamma_{\bu}}(\Omega),
\end{equation}
which is the weak form of $-\alpha\Delta z-\bu'\cdot\nabla z+\varrho_0 z=\varepsilon$ in $\Omega$, $z=0$ on $\Gamma_{\bu}$, $\alpha\nabla z\cdot\bn=0$ on $\Gamma_p$. Since $|\int_\Omega\chi\,\bu'\cdot\nabla\chi|\leq\|\bu'\|_{0,3,\Omega}\|\chi\|_{0,6,\Omega}\|\nabla\chi\|_{0,\Omega}\leq C_{\mathsf{emb}}r_{\mathsf{dual}}\|\nabla\chi\|^2_{0,\Omega}\leq\frac\alpha2\|\nabla\chi\|^2_{0,\Omega}$, the bilinear form in \eqref{eq:adjoint} is coercive on $\rH^1_{\Gamma_{\bu}}(\Omega)$, and the Lax--Milgram lemma provides a unique solution $z$. Testing with $\chi=z$ gives $\frac\alpha2\|\nabla z\|_{0,\Omega}^2+\varrho_0\|z\|^2_{0,\Omega}\leq\|\varepsilon\|_{0,\Omega}\|z\|_{0,\Omega}$, and in particular $\varrho_0\|z\|_{0,\Omega}\leq\|\varepsilon\|_{0,\Omega}$. Next, $z$ is a fixed point of the map $\Phi$ that assigns to $w\in\rH^1_{\Gamma_{\bu}}(\Omega)$ the solution $\Phi(w)\in\rH^1_{\Gamma_{\bu}}(\Omega)$ of the mixed Poisson problem with right-hand side $\langle g_w,\chi\rangle := \alpha^{-1}\int_\Omega(\varepsilon+\bu'\cdot\nabla w-\varrho_0 z)\,\chi$. By H\"older's inequality (with $\frac13+\frac1q+\frac{2q-3}{3q}=1$), the Sobolev embedding defining $C_q$, and $|\Omega|\leq1$, we have
\[\|g_{w_1}-g_{w_2}\|_{(\rW^{1,q'}_{\Gamma_{\bu}}(\Omega))'}\leq \alpha^{-1}C_q\|\bu'\|_{0,3,\Omega}\|\nabla(w_1-w_2)\|_{0,q,\Omega},\qquad \|g_{0}\|_{(\rW^{1,q'}_{\Gamma_{\bu}}(\Omega))'}\leq 2\alpha^{-1}C_q\|\varepsilon\|_{0,\Omega}.\]
Hence, by (R) and $\|\bu'\|_{0,3,\Omega}\leq r_{\mathsf{dual}}\leq\frac{\alpha}{2C_{\mathsf{R}}C_q}$, $\Phi$ is a contraction (with constant $\frac12$) in $\rW^{1,q}_{\Gamma_{\bu}}(\Omega)$; similarly, owing to $\|\bu'\|_{0,3,\Omega}\leq\frac{\alpha}{2C_{\mathsf{emb}}}$, it is also a contraction in $\rH^1_{\Gamma_{\bu}}(\Omega)$. Since $\rW^{1,q}_{\Gamma_{\bu}}(\Omega)\subseteq\rH^1_{\Gamma_{\bu}}(\Omega)$, the fixed point in $\rW^{1,q}_{\Gamma_{\bu}}(\Omega)$ coincides with $z$, and
\begin{equation}\label{eq:adjoint.reg}
\|\nabla z\|_{0,q,\Omega}\leq 2\,C_{\mathsf{R}}\|g_0\|_{(\rW^{1,q'}_{\Gamma_{\bu}}(\Omega))'}\leq 4\,C_{\mathsf{R}}C_q\,\alpha^{-1}\|\varepsilon\|_{0,\Omega}.
\end{equation}

\emph{Step 3 (conclusion).} Taking $\chi=\varepsilon$ in \eqref{eq:adjoint} and $\chi=z$ in \eqref{eq:primal.diff}, the left-hand sides coincide, and therefore, using H\"older's inequality ($\frac1s+\frac12+\frac1q=1$), \eqref{continuity-S} (which gives $\|\bu'-\bu\|_{0,\Omega}\leq\kappa\,\alpha_1\|\varepsilon\|_{0,\Omega}$), and \eqref{eq:adjoint.reg}, we arrive at
\begin{align*}
\|\varepsilon\|^2_{0,\Omega} = \int_\Omega\varphi\,(\bu'-\bu)\cdot\nabla z & \leq\|\varphi\|_{0,s,\Omega}\|\bu'-\bu\|_{0,\Omega}\|\nabla z\|_{0,q,\Omega}\\
& \leq \mathsf{M}_s\,\kappa\,\alpha_1\,4\,C_{\mathsf{R}}C_q\,\alpha^{-1}\|\varepsilon\|_{0,\Omega}^2 = L_*\|\varepsilon\|^2_{0,\Omega}.
\end{align*}
Since $L_*<1$, $\varphi'=\varphi$; then \eqref{continuity-S} gives $\bu'=\bu$, Theorem~\ref{th:continuous-ARD} gives $\bzeta'=\bzeta$, and Theorem~\ref{th:continuous-DF} shows $p'=p$. This completes the rest of the proof.
\end{proof}
Let us briefly remark that, due to the coupling between the Darcy--Forchheimer velocity and the transport equation, our  fixed-point framework does not allow us to readily obtain $\alpha$-independent estimates {(nor $\rF$-independent smallness conditions). This manifests itself in the hypothesis of Theorem~\ref{th:continuous-ARD}, in the small data assumption \eqref{data-bounded-map-closed}, and in \eqref{eq:uniqueness.condition}. We also stress that robustness refers to the fact that the constants $\mathcal{M}$, $C^{\mathsf{DF}}_c$, $C_c$, $C^{\mathsf{ADR}}_c$ in \eqref{estimate su} do not depend on $\kappa,\rF,\varrho_0$. The smallness assumptions are formulated in terms of the same constants, and therefore the class of admissible data does not shrink when $\kappa$ or $\varrho_0$ vary. The only parameters entering the smallness conditions are $\alpha$ (through $C^c_\alpha$, $C^{\mathsf{ADR}}_c$, $r_{\mathsf{dual}}$ and $L_*$) and $\rF$ (through the factor $\rF^{-\frac13}$ in \eqref{data-bounded-map-closed} and \eqref{eq:uniqueness.condition}, which reflects that the advecting velocity must be small in $\rL^3$); this is  the partial loss of robustness mentioned in the introduction, and it does not contradict the robustness of the stability. (The permeability enters $L_*$ only through the factor $\kappa\leq1$, so that the uniqueness condition becomes less restrictive as $\kappa\to0$.)

We also emphasise that Theorem~\ref{th:wellposedness} does not require any regularity of the solution beyond the natural $\rL^p$ setting, nor any assumption on the data other than \eqref{eq:f.lipschitz}, $m\in\rL^{\frac65}(\Omega)$, and smallness. The only additional ingredient for uniqueness is Assumption~(R), which concerns the geometry and is automatically satisfied in 2D. The duality argument is needed for the following reason: in the natural setting, the advective coupling of two solutions is controlled by $\|\bu_1-\bu_2\|_{0,3,\Omega}$ (cf. \eqref{eq:Lip.ARD}, which uses $\frac13+\frac12+\frac16=1$), whereas the Forchheimer term only provides control of $\rF\|\bu_1-\bu_2\|_{0,3,\Omega}^3$ (cf. \eqref{eq:monotone}), and the resulting cross term, being quadratic, cannot be absorbed by a cubic quantity for small differences. Measuring $\varphi'-\varphi$ in $\rL^2(\Omega)$ by duality moves the derivative onto the adjoint solution, so that the coupling is controlled by $\kappa^{-\frac12}\|\bu_1-\bu_2\|_{0,\Omega}$ instead, which depends linearly on $\|\varphi'-\varphi\|_{0,\Omega}$ (cf. \eqref{continuity-S}).}

\section{Finite element discretisation}\label{sec:fe}
Let \( {\cT}_h \) be a regular simplicial mesh defined in a bounded Lipschitz domain \( \Omega \), consisting of tetrahedral elements (or triangular elements in two dimensions) \( K \) with diameter \( h_K \). We define the mesh size as \( h := \max\bigl(h_K : K \in {\cT}_h\bigr) \). Given an integer \( k \geq 0 \) and a generic element \( K \in {\cT}_h \), we denote by \( \mathbb{P}_k(K) \) the space of polynomials defined locally in \( K \) with a degree at most \( k \), and denote by \( \mathbb{P}_k(\cT_h) \) its global counterpart. In addition, let ${\mathbb{RT}_k(\cT_h)}$ denote the \( \mathbf{H}_{\Gamma_{\bullet}}({\vdiv}, \Omega) \)-discrete conforming subspace defined by the Raviart--Thomas  elements of order \( k \), i.e.,
\[{\mathbb{RT}_k(\cT_h)} := \{ \bu \in \mathbf{H}_{\Gamma_{\bullet}}({\vdiv_s}, \Omega):\bu|_{K} = [\mathbb{P}_k(K)]^d \oplus \widetilde{\mathbb{P}}_k(K) \bx, \ \bx \in \Omega \subset \mathbb{R}^d, K \in {\cT}_h \},\]
where \( \widetilde{\mathbb{P}}_k(K) \) is the space of all polynomials over \( K \) of degree exactly equal to \( k.\)
{We follow \cite[Section~4]{dhpr_sinum26} for constructing discrete spaces for the system~\eqref{eq:weak}. Define conforming discrete weighted subspaces, \(\bV_h\subseteq\bV\) and \(\bW_h\subseteq\bW\) as follows
\begin{subequations}
\begin{align}
 \bV_h&:=  \kappa^{-\frac{1}{2}}\bL_h^2(\Omega) \cap \bH_{h,\Gamma_{\bu}}(\vdiv,\Omega) \cap \rF^{\frac13} \bL_h^3(\Omega),\label{eq:discrete.Vh}\\
  \bW_h&:= \alpha^{-\frac12}\bL^2_h(\Omega)\cap\bH_{h,\Gamma_p}(\vdiv_{\frac65},\Omega), \label{eq:discrete.Wh}
 \end{align}\end{subequations}
 where \( {\bL}_h^r(\Omega), \,\,{\bH}_{{h,\Gamma_{\bullet}}}({\vdiv}, \Omega), \,{\bH}_{{h,\Gamma_{\bullet}}}({\vdiv}_s, \Omega) \)-norms correspond to  functions in ${\mathbb{RT}_k(\cT_h)}$ equipped with  ${\bL}^r(\Omega)$, for \(r=2,3,\) {$\bH_{\Gamma_{\bullet}}(\vdiv,\Omega),$ and $\bH_{\Gamma_{\bullet}}(\vdiv_{\frac65},\Omega)$}. More explicitly, as sets
\[\bV_h = \{\bv_h\in\mathbb{RT}_k(\cT_h):\ \bv_h\cdot\bn = 0\ \text{on}\ \Gamma_{\bu}\},\qquad \bW_h = \{\bw_h\in\mathbb{RT}_k(\cT_h):\ \bw_h\cdot\bn = 0\ \text{on}\ \Gamma_{p}\},\]
and they are endowed with the (restrictions of the) norms $\|\bullet\|_{\bV}$ and $\|\bullet\|_{\bW}$ defined in Section~\ref{sec:abstract}, respectively. In particular, $\bV_h\subseteq\bV$ and $\bW_h\subseteq\bW$ with equality of norms.
Next, define the discrete gradient operator \( \nabla_h: {\mathbb{P}_k(\cT_h)} \to {\mathbb{RT}_k(\cT_h)} \) (cf. \cite[Equation~3.1]{arnold1997preconditioning}) by
\begin{equation}\label{eq:disc.grad}
\sum_{K\in\cT_h}\int_K(\nabla_h q_h)\cdot\bv_h = -\sum_{K\in\cT_h}\int_K(\vdiv\bv_h)\,q_h \quad \forall q_h \in {\mathbb{P}_k(\cT_h)}, \text{ and } \forall \bv_h \in {\mathbb{RT}_k(\cT_h)}.
\end{equation}
Additionally, define the following broken Sobolev spaces
\begin{align*}
    {\rW^{1,r}}(\mathcal{T}_h)&:=\{p_h\in\rL^2(\Omega):p_h|_K\in\rW^{1,r}(K) \quad \forall K\in \mathcal{T}_h\}, \quad \text{for } r=2,\frac32,
\end{align*}
and whenever needed, we add a subscript $\Gamma_{\bullet}$ denoting a part on the boundary, where the value of the function is prescribed to zero.
These spaces are equipped with DG (broken Sobolev) seminorms
\(|\bullet|_{{1,\mathcal{T}_h}}\) and
\(|\bullet|_{{1, r, \mathcal{T}_h}}\),
defined as
\begin{align*}
{\rH^1}(\mathcal{T}_h) \ni \zeta & \mapsto  |\zeta|^2_{{1,\mathcal{T}_h}} := \biggl(
\sum_{K\in\mathcal{T}_h}\| {\nabla} \zeta\|^{2}_{0,K}
+ \sum_{F\in \mathcal{F}_h} h_F^{-1}\|[\![\zeta]\!]\|^2_{0,F}
\biggr)^{\frac12},
\\
\rW^{1,r}(\mathcal{T}_h) \ni \xi & \mapsto
{|\xi|_{{1,r,\mathcal{T}_h}} := \biggl(\sum_{K\in\mathcal{T}_h}\|{\nabla} \xi\|^{r}_{0,r,K}
+ \sum_{F\in \mathcal{F}_h} h_F^{1-r}\|[\![\xi]\!]\|^{r}_{0,r,F}
\biggr)^{\frac{1}{r}}},\end{align*}
respectively, where $\mathcal{F}_h$ denotes the set of facets in the mesh $\cT_h$.
Further, define the following weighted spaces for the pressure, as the sum of the infinite-dimensional (but broken) spaces $\rL^2(\Omega)$, \(\rH^1_{\Gamma_p}(\mathcal{T}_h)\) and \(\rW^{1,\frac32}_{\Gamma_p}(\mathcal{T}_h),\) and the concentration, as the sum of \(\rL^6(\Omega)\) and \(\rH^1_{\Gamma_{\bu}}(\cT_h):\)
\begin{subequations}
\begin{align}\label{eq:broken.Qhat}
{\widehat{\rQ}}(\cT_h)&:=\rL^2(\Omega)+\kappa^\frac12\rH^1_{\Gamma_p}(\mathcal{T}_h)+\rF^{-\frac13}\rW^{1,\frac32}_{\Gamma_p}(\mathcal{T}_h),\\
\widehat{\Psi}(\cT_h)&:=\rL^6(\Omega)+\alpha^{\frac12}\rH^1_{\Gamma_{\bu}}(\cT_h).
\end{align}\end{subequations}
Define broken weighted norms on \({\widehat{\rQ}}(\cT_h)\) and \(\widehat{\Psi}(\cT_h)\) (see, \cite[Section~4]{dhpr_sinum26}) as
\begin{align}
\|q\|_{{\widehat{\rQ}(\cT_h)}}&:=
\inf_{\substack{
    q = q_{1} + q_{2} + q_{3} \\
    (q_{1},q_{2},q_{3}) \in \\ \rL^2(\Omega)
     \times\rH_{\Gamma_p}^1(\mathcal{T}_h)
    \times \rW_{\Gamma_p}^{1,\frac32}(\mathcal{T}_h)
}}
\left(
    \|q_{1}\|_{0,\Omega}
    +\kappa^{\frac12}\, |q_{2}|_{1,\mathcal{T}_h}
    + \rF^{-\frac13} |q_{3}|_{1,\frac32,\mathcal{T}_h}
\right),\label{norm:Q.hat}\\
\|\psi\|_{\widehat{\Psi}(\cT_h)}&:=\inf_{\substack{
    \psi = \psi_{1} + \psi_{2} \\
    (\psi_{1},\psi_{2}) \in \\ \rL^6(\Omega)
     \times\rH_{\Gamma_{\bu}}^1(\mathcal{T}_h)
}}
\left(
    \|\psi_{1}\|_{0,6,\Omega}
    +\alpha^{\frac12}\, |\psi_{2}|_{1,\mathcal{T}_h}
\right).\label{norm:Psi.hat}
\end{align}
Finally, we set  discrete spaces \(\rQ_h\) and \(\Psi_h\) (non-conforming subspaces of   \(\rQ\) and \(\Psi\),  respectively) as
\begin{subequations}
\begin{align}
\rQ_h &:=  \rL_h^2(\Omega) +\kappa^{\frac12} \rH^1_{h,{\Gamma_p}}(\Omega)+\rF^{-\frac13}\rW^{1,\frac32}_{h,\Gamma_p}(\Omega),\label{eq:disc.Ph}\\
\Psi_h&:=\rL^6_h(\Omega)+\alpha^{\frac12}\rH^1_{h,\Gamma_{\bu}}(\Omega),\label{eq:disc.Psih}
\end{align}
\end{subequations}
where \( \rL_h^r(\Omega), \rH^1_h(\Omega), \) and \(\rW^{1,\frac32}_h(\Omega)\) are polynomial spaces having discrete functions from  \( {\mathbb{P}_k(\cT_h)} \), and equipped with the \( \rL^r(\Omega)\) norm (for \(r=2,6\))  and the broken \(\rH^1(\cT_h),\,\rW^{1,\frac32}(\cT_h)\) seminorms, respectively.
That is, as sets, $\rQ_h = \Psi_h = \mathbb{P}_k(\cT_h)$, and we equip \(\rQ_h\) and \(\Psi_h\) with the $\|\bullet\|_{\widehat{\rQ}(\cT_h)}$ and $\|\bullet\|_{\widehat{\Psi}(\cT_h)}$ norms, respectively (these are norms on the larger spaces $\widehat{\rQ}(\cT_h)$, $\widehat{\Psi}(\cT_h)$ by the same arguments given in Section~\ref{sec:intro}, using the discrete Poincar\'e and Sobolev inequalities for broken spaces, see \cite[Chapter~5]{dipietro2012mathematical}). Note that $\rQ_h\subseteq\rQ\subseteq\widehat{\rQ}(\cT_h)$ and $\Psi_h\subseteq\Psi\subseteq\widehat{\Psi}(\cT_h)$, but the norms on $\rQ_h,\Psi_h$ are the broken ones, which makes these discrete spaces non-conforming with respect to the norms of $\rQ$ and $\Psi$.
Next, set discrete bilinear forms
$b_{1,h}: \bV_h\times \rQ_h \to \mathbb{R}, b_{2,h}: \bW_h\times \Psi_h \to \mathbb{R}$
 and the discrete linear functional \(H_h:\Psi_h\to\mathbb{R}\) as follows
\begin{align*}
b_{1,h}(\bv_h,q_h)&:=-\int_{\Omega}\vdiv\bv_h\,q_h,\quad
b_{2,h}(\bzeta_h,\psi_h):=\int_{\Omega}\vdiv\bzeta_h\psi_h, \quad H_h(\psi_h):=-\int_{\Omega}m\psi_h,
\end{align*}
for all \( (\bv_h\,  q_h, \bzeta_h,\psi_h)\in \bV_h\times {\rQ_h} \times \bW_h\times\Psi_h.\)
Then, with the  FE spaces \eqref{eq:discrete.Vh}, \eqref{eq:discrete.Wh}, \eqref{eq:disc.Ph} and \eqref{eq:disc.Psih}
 the mixed FE formulation of \eqref{eq:weak} consists in  finding $(\bu_h,p_h,\bzeta_h,\varphi_h)\in \bV_h\times \rQ_h\times\bW_h\times\Psi_h$ such that
\begin{subequations}\label{eq:weak_Galerkin}
    \begin{align}
     a_1(\bu_h,\bv_h) + d_1(\bu_h; \bu_h,\bv_h) +  b_{1,h}(\bv_h,p_h) & = F^{\varphi_h}(\bv_h) \qquad \forall \bv_h \in \bV_h,\label{eq:weak,1h}\\
    b_{1,h}(\bu_h,q_h) & = 0 \qquad\qquad \forall q_h \in \rQ_h,\label{eq:weak,2h}\\
    a_2(\bzeta_h,\bw_h) + b_{2,h}(\bw_h,\varphi_h) + d_2^{\bu_h}(\bw_h, \varphi_h) & = 0\qquad \qquad \forall \bw_h \in \bW_h, \label{eq:weak,3h}\\
    b_{2,h}(\bzeta_h, \psi_h) - c_{2,h}(\varphi_h,\psi_h) & = H_h(\psi_h) \qquad \forall \psi_h \in \Psi_h. \label{eq:weak,4h}
    \end{align}
\end{subequations}

\noindent
{A direct discrete analysis based on
 \(\bV_h:=(\mathbb{RT}_k(\cT_h),\|\bullet\|_{\vdiv,\Omega}),\) \(\bW_h:=(\mathbb{RT}_k(\cT_h),\|\bullet\|_{\vdiv_{\frac65},\Omega}),\) \(\rQ_h:=(\mathbb{P}_k(\cT_h),\|\bullet\|_{0,\Omega}),\) and \(\Psi_h:=(\mathbb{P}_k(\cT_h),\|\bullet\|_{0,6,\Omega})\),
endowed with the standard norms, does not lead to parameter-robust inf-sup and Braess conditions.
The parameter robustness is instead obtained by adopting the weighted discrete spaces \eqref{eq:discrete.Vh}-\eqref{eq:discrete.Wh}, \eqref{eq:disc.Ph}-\eqref{eq:disc.Psih}.}

\subsection{Robust discrete solvability}
 Consider the following system for a fixed $\hat{\varphi}_h\in \Psi_h$
\begin{subequations}\label{eq:weak.fixed.u.hat.disc}
    \begin{align}
    a_1(\bu_h,\bv_h) + d_1(\bu_h; \bu_h,\bv_h) +  b_{1,h}(\bv_h,p_h) & = F^{\hat\varphi_h}(\bv_h) \quad \forall \bv_h \in \bV_h,\label{eq:weak,1.hat.u.disc}\\
    b_{1,h}(\bu_h,q_h) & = 0 \quad \forall q_h \in \rQ_h,\label{eq:weak,2.hat.u.disc}
    \end{align}
\end{subequations}
and recall from \cite[Theorem~4.3]{dhpr_sinum26} the following result.
\begin{theorem}[Well-posedness of discrete  Darcy--Forchheimer equations]\label{th:discrete-DF}
 Assume that
 $\bf(\hat{\varphi}_h)\in \bL^2(\Omega).$ Then there exists a unique  solution $(\bu_h,p_h)\in \bV_h\times\rQ_h$ to   \eqref{eq:weak.fixed.u.hat.disc}. Furthermore, there is a positive constant ${C^{\mathsf{DF}}_{\mathsf{d}}},$ independent of the parameters $\kappa, \rF,h$,  such that
\begin{equation*}
 \|(\bu_h, p_h )\|_{\bV\times\widehat{\rQ}(\cT_h)}  \leq {C^{\mathsf{DF}}_{\mathsf{d}}}\max \bigl(\|\bf(\hat{\varphi}_h)\|_{0,\Omega}^{\frac12},\|\bf(\hat{\varphi}_h)\|_{0,\Omega},\|\bf(\hat{\varphi}_h)\|_{0,\Omega} ^2\bigr).
    \end{equation*}
Additionally, there holds
    \begin{align}\label{est:disc.uh.bound}
        \|\bu_h\|_{\bV}\leq \mathcal{M}\bigl(\|\bf(\hat{\varphi}_h)\|_{0,\Omega}\bigr) = 4\max\bigl\{\|\bf(\hat{\varphi}_h)\|_{0,\Omega},\|\bf(\hat{\varphi}_h)\|_{0,\Omega}^{\frac12}\bigr\}.
    \end{align}
\end{theorem}

\noindent We next prove the following lemma, which is crucial in deducing the discrete inf-sup stability of \(b_{2,h}(\bullet,\bullet)\). It relies on two tools. First, we use a family of projections $\pi_h:\bW\to\bW_h$ and $\pi^{\mathrm{L}}_h:\rL^{\frac65}(\Omega)\to\mathbb{P}_k(\cT_h)$ satisfying the commuting property $\vdiv\pi_h\btau = \pi^{\mathrm{L}}_h\vdiv\btau$ for all $\btau\in\bW$, and that are bounded uniformly in $h$:
\begin{equation}\label{eq:smoothed.proj}
\|\pi_h\btau\|_{0,\Omega}\leq C_\pi\|\btau\|_{0,\Omega},\qquad \|\pi^{\mathrm{L}}_h q\|_{0,\frac65,\Omega}\leq C_\pi\|q\|_{0,\frac65,\Omega}\qquad\forall\btau\in\bW,\ q\in\rL^{\frac65}(\Omega).
\end{equation}
Such smoothed   commuting projections, which preserve the boundary condition $\btau\cdot\bn=0$ on $\Gamma_p$ and are bounded in $\rL^p$ for all $p\in[1,\infty]$ on shape-regular meshes, are constructed in \cite{licht2019smoothed} for mixed boundary conditions (see also \cite{ern2016mollification} for the case of a single type of boundary condition). Second, we use the $h$-uniform stability of the $\rL^2(\Omega)$-orthogonal projection $P_h:\rL^2(\Omega)\to\mathbb{P}_k(\cT_h)$ in the broken norm, $\|P_h\psi\|_{\widehat{\Psi}(\cT_h)}\leq C_P\|\psi\|_{\widehat\Psi(\cT_h)}$, which follows from its $\rL^6$-stability and its stability in the broken $\rH^1$-seminorm on shape-regular meshes \cite[Chapter~1 and Chapter~5]{dipietro2012mathematical}.
\begin{lemma}\label{lem:disc.inf-sup:ARD}
For each $\ell_h\in\mathbb{P}_k(\cT_h)$ (identified with an element of $\Psi_h'$ via $\langle\ell_h,\psi_h\rangle:=\int_\Omega\ell_h\psi_h$), define
\[\cS_{\mathsf{d}}(\ell_h):=\pi_h\bigl(\cS_c(\ell_h)\bigr)\in\bW_h,\]
where $\cS_c$ is the operator of Lemma~\ref{lem:inf-sup:ARD}. Then $\cS_{\mathsf{d}}$ is linear, and
\begin{gather}\nonumber
\|\cS_{\mathsf{d}}(\ell_h)\|_{\bW}\leq 2\max\{1,C_\pi\}\,\|\ell_h\|_{\Psi'},\quad\text{with a constant independent of } \alpha,\varrho_0, h,\quad \text{and}\\
    \langle\vdiv [\cS_{\mathsf{d}}(\ell_h)], \psi_h\rangle = \int_\Omega\ell_h\psi_h \quad \text{for all}\ \ell_h \in \mathbb{P}_k(\cT_h),\ \psi_h \in \Psi_h.\label{S:disc.prop}
\end{gather}
\end{lemma}
\begin{proof}
By Lemma~\ref{lem:inf-sup:ARD}, $\btau:=\cS_c(\ell_h)\in\bW$ satisfies $\vdiv\btau = \ell_h$ and $\|\btau\|_{\bW}\leq2\|\ell_h\|_{\Psi'}$. Since $\pi^{\mathrm{L}}_h$ is a projection onto $\mathbb{P}_k(\cT_h)$, commuting gives $\vdiv\cS_{\mathsf{d}}(\ell_h) = \pi^{\mathrm{L}}_h\ell_h = \ell_h$, which implies \eqref{S:disc.prop}. Moreover, by \eqref{eq:smoothed.proj},
\begin{align*}
\|\cS_{\mathsf{d}}(\ell_h)\|_{\bW} & = \alpha^{-\frac12}\|\pi_h\btau\|_{0,\Omega}+\|\ell_h\|_{0,\frac65,\Omega}\leq C_\pi\alpha^{-\frac12}\|\btau\|_{0,\Omega}+\|\vdiv\btau\|_{0,\frac65,\Omega}\\
& \leq\max\{1,C_\pi\}\|\btau\|_{\bW}\leq 2\max\{1,C_\pi\}\|\ell_h\|_{\Psi'}.\end{align*}
No additional regularity of the auxiliary problem, and hence no additional assumption on $\Omega$ (beyond those required in \cite{licht2019smoothed} for the boundary partition $\partial\Omega = \overline{\Gamma_{\bu}}\cup\overline{\Gamma_p}$), is needed.
\end{proof}

\noindent
We now consider the following {discrete ADR} system for a fixed $\hat{\bu}_h \in \bV_h$
\begin{subequations}\label{eq:weak.hat.phi.disc}
    \begin{align}
    a_2(\bzeta_h,\bw_h) + b_{2,h}(\bw_h,\varphi_h) + d_2^{\hat\bu_h}(\bw_h, \varphi_h) & = 0 \quad \forall \bw_h \in \bW_h, \label{eq:weak,3.hat.phi.disc}\\
    b_{2,h}(\bzeta_h, \psi_h) - c_{2,h}(\varphi_h,\psi_h) & = H_h(\psi_h) \quad \forall \psi_h \in \Psi_h, \label{eq:weak,4.hat.phi.disc}
    \end{align}
\end{subequations}
and employ the discrete version of Theorem~\ref{th:abstract} to discuss its parameter-robust unique solvability.
\begin{theorem}[Well-posedness of discrete  {ADR} equations]\label{th:discrete-ARD}
Assume that
 $m \in \rL^{\frac65}(\Omega)$. {Then, for a fixed $\hat{\bu}_h\in \bV_h,$ there {exist}  positive constants \(C_d, C^d_{\alpha}\) (depending on $\alpha$), such that, whenever {\(\|\hat\bu_h\|_{{0,3,\Omega}}\leq
 \frac{C_d}{2C^d_{\alpha}},\)}} there exists a unique  solution $(\bzeta_h,\varphi_h)\in \bW_h\times\Psi_h$ to the problem \eqref{eq:weak.hat.phi.disc}.
 Furthermore, there is a positive constant $C^{\mathsf{ADR}}_{\mathsf{d}}$ independent of \(h,\) and  $\varrho_0$, such that
\begin{equation*}
 \|(\bzeta_h, \varphi_h)\|_{\bW\times\widehat{\Psi}(\cT_h)}  \leq C^{\mathsf{ADR}}_{\mathsf{d}} \| m \|_{0,\frac65,\Omega} .
    \end{equation*}
\end{theorem}
\begin{proof}
     The boundedness of \(a_2(\bullet,\bullet),b_{2,h}(\bullet,\bullet),\,c_{2,h}(\bullet,\bullet),\) and  symmetry and positive semi-definiteness of \(a_2(\bullet,\bullet)\) and \(c_{2,h}(\bullet,\bullet)\) (with $c_{2,h}:=c_2$ restricted to $\Psi_h\times\Psi_h$), and the inf-sup stability of \(a_2(\bullet,\bullet)\) in \(\bZ_h:=\{\btau_h\in\bW_h:b_{2,h}(\btau_h,\psi_h)=0,\,\forall \psi_h\in\Psi_h\}\) follow analogously as in Theorem~\ref{th:continuous-ARD}. Therefore, to complete the rest of the proof, it is enough to prove the discrete inf-sup stability of \(b_{2,h}(\bullet,\bullet)\) and the discrete Braess condition.

\smallskip
\noindent\textbf{Discrete inf-sup condition.}
Let $\psi_h\in\Psi_h$. Restricting the supremum to elements of the form $\btau_h=\cS_{\mathsf{d}}(\ell_h)$, $\ell_h\in\mathbb{P}_k(\cT_h)$, and using Lemma~\ref{lem:disc.inf-sup:ARD}, we obtain
\begin{align*}
\sup_{\bzero\neq\btau_h \in \bW_h} \frac{\int_\Omega \psi_h\vdiv \btau_h}{\|\btau_h\|_\bW} \geq \sup_{0\neq\ell_h \in \mathbb{P}_k(\cT_h)} \frac{\int_\Omega\ell_h\psi_h}{\|\cS_{\mathsf{d}}(\ell_h)\|_\bW}\geq \frac{1}{2\max\{1,C_\pi\}}\sup_{0\neq\ell_h \in \mathbb{P}_k(\cT_h)} \frac{\int_\Omega\ell_h\psi_h}{\|\ell_h\|_{\Psi'}}.
\end{align*}
Next, since $\Psi\subseteq\widehat\Psi(\cT_h)$ with $\|\psi\|_{\widehat\Psi(\cT_h)}\leq\|\psi\|_{\Psi}$ (an $\rH^1$-function has no jumps), there holds  $\|\ell_h\|_{\Psi'}\leq\|\ell_h\|_{\widehat\Psi(\cT_h)'}$. Moreover, by the Hahn--Banach theorem there exists $L\in\widehat\Psi(\cT_h)'$ with $\|L\|_{\widehat\Psi(\cT_h)'}=1$ and $L(\psi_h) = \|\psi_h\|_{\widehat\Psi(\cT_h)}$. Let $\ell_h\in\mathbb{P}_k(\cT_h)$ be the $\rL^2$-representative of $L$ restricted to $\mathbb{P}_k(\cT_h)$, so that $\int_\Omega\ell_h\psi = L(P_h\psi)$ for all $\psi\in\widehat\Psi(\cT_h)$. Then $\int_\Omega\ell_h\psi_h = \|\psi_h\|_{\widehat\Psi(\cT_h)}$ and $\|\ell_h\|_{\widehat\Psi(\cT_h)'}\leq C_P$, by the stability of $P_h$. Altogether,
\[\sup_{\bzero\neq\btau_h \in \bW_h} \frac{\int_\Omega \psi_h\vdiv \btau_h}{\|\btau_h\|_\bW}\geq\tilde{C}_b \|\psi_h\|_{\widehat{\Psi}(\cT_h)}\qquad\forall\psi_h\in\Psi_h,\qquad\text{with}\quad \tilde C_b:=\frac{1}{2C_P\max\{1,C_\pi\}},\]
which is independent of \(\alpha,\varrho_0,h\), and hence the discrete stability of inf-sup is achieved.

\smallskip
\noindent\textbf{Discrete Braess condition.} Setting \(\bw_h=\bzeta_h\), \(\psi_h=0,\) and exploiting \(\|\bzeta_h\|_{\bW}=\tnorm{(\bzeta_h,0)}_{\bW\times\widehat{\Psi}(\cT_h)},\) one obtains
\begin{align}\label{eq:Braess.disc.1st} \sup_{(\bzero,0)\neq(\bw_h,\psi_h)\in\bW_h\times\Psi_h}\frac{a_2(\bzeta_h,\bw_h)+b_{2,h}(\bzeta_h,\psi_h)}{\tnorm{(\bw_h,\psi_h)}_{\bW\times\widehat{\Psi}(\cT_h)}}\geq\frac{\alpha^{-1}\int_\Omega\bzeta_h\cdot\bzeta_h}{\tnorm{(\bzeta_h,0)}_{\bW\times\widehat{\Psi}(\cT_h)}}=\frac{\alpha^{-1}\|\bzeta_h\|_{0,\Omega}^2}{\|\bzeta_h\|_{\bW}}.
    \end{align}
Next, {taking the particular values} \(\bw_h=\bzero,\) and \(\psi_h=P_h\bigl(\mathrm{sgn}(\vdiv\bzeta_h)|\vdiv\bzeta_h|^{\frac15}\bigr)\in\Psi_h\) (the projection is needed since $\mathrm{sgn}(\vdiv\bzeta_h)|\vdiv\bzeta_h|^{\frac15}$ is not a polynomial), and noting that $b_{2,h}(\bzeta_h,\psi_h) = \|\vdiv\bzeta_h\|^{\frac65}_{0,\frac65,\Omega}$ because $\vdiv\bzeta_h\in\mathbb{P}_k(\cT_h)$, while $\tnorm{(\bzero,\psi_h)}_{\bW\times\widehat\Psi(\cT_h)}\leq2\|\psi_h\|_{0,6,\Omega}\leq 2C_P\|\vdiv\bzeta_h\|^{\frac15}_{0,\frac65,\Omega}$ (as in the continuous case, using also the $\rL^6$-stability of $P_h$), one obtains the following
\begin{align}\label{eq:Braess.disc.2nd}
(\mathrm{LHS})_{\eqref{eq:Braess.disc.1st}}\geq \frac{a_2(\bzeta_h,\bzero)+b_{2,h}(\bzeta_h,\psi_h)}{\tnorm{(\bzero,\psi_h)}_{\bW\times\widehat{\Psi}(\cT_h)}}\geq\frac{1}{2C_P} \frac{\|\vdiv \bzeta_h\|^{\frac65}_{0,\frac65,\Omega}}{\|\vdiv \bzeta_h\|^{\frac15}_{0,\frac65,\Omega}}=\frac{1}{2C_P}\|\vdiv \bzeta_h\|_{0,\frac65,\Omega}.
\end{align}
Now, let us combine \eqref{eq:Braess.disc.1st} and \eqref{eq:Braess.disc.2nd} to obtain the following
\begin{align}\label{eq:Braess.disc.3rd}
\sup_{(\bzero,0)\neq(\bw_h,\psi_h)\in\bW_h\times\Psi_h}\frac{a_2(\bzeta_h,\bw_h)+b_{2,h}(\bzeta_h,\psi_h)}{\tnorm{(\bw_h,\psi_h)}_{\bW\times\widehat{\Psi}(\cT_h)}}\geq\max\biggl\{{\frac{\alpha^{-1}\|\bzeta_h\|^2_{0,\Omega}}{\|\bzeta_h\|_{\bW}},\frac{1}{2C_P}\|\vdiv \bzeta_h\|_{0,\frac65,\Omega}}\biggr\}.
\end{align}
For $\bzeta_h\neq\bzero$, denoting \(\varpi:=\frac{\alpha^{-\frac12}\|\bzeta_h\|_{0,\Omega}}{\|\bzeta_h\|_{\bW}}\in[0,1]\) (so that $\|\vdiv\bzeta_h\|_{0,\frac65,\Omega} = (1-\varpi)\|\bzeta_h\|_{\bW}$), we deduce from \eqref{eq:Braess.disc.3rd} that
\begin{gather*}
\sup_{(\bzero,0)\neq(\bw_h,\psi_h)\in\bW_h\times\Psi_h}\frac{a_2(\bzeta_h,\bw_h)+b_{2,h}(\bzeta_h,\psi_h)}{\tnorm{(\bw_h,\psi_h)}_{\bW\times\widehat{\Psi}(\cT_h)}}\geq \gamma_{\mathsf{d}} \|\bzeta_h\|_{\bW},\\
\gamma_{\mathsf{d}}:=\min_{\varpi\in[0,1]}\max\Bigl\{\varpi^2,\frac{1-\varpi}{2C_P}\Bigr\}>0,
\end{gather*}
where $\gamma_{\mathsf{d}}$ is independent of $\alpha,\varrho_0$ and $h$.
An appeal to  Theorem~\ref{th:abstract} now yields the well-posedness of the discrete reaction-diffusion system~\eqref{eq:weak,3h} and \eqref{eq:weak,4h} without the advection part. Additionally, the following estimate holds
\begin{align}\label{eq:A.stab.disc}
\sup_{(\bzero,0)\neq(\tau_h,\psi_h)\in\bW_h\times\Psi_h}\frac{\cA(\bzeta_h,\varphi_h;\btau_h,\psi_h)}{\tnorm{(\btau_h,\psi_h)}_{\bW\times \widehat{\Psi}(\cT_h)}}\geq C_d\tnorm{(\bzeta_h,\varphi_h)}_{\bW\times \widehat{\Psi}(\cT_h)},
\end{align}
where \(C_d\) is the positive constant from
\eqref{eq:C},   independent of the model and   discretisation parameters.

Finally, for the well-posedness of the uncoupled discrete ADR equations, for a fixed \(\hat{\bu}_h\in\bV_h,\) we consider another bilinear form \(\mathcal{A}^{\mathsf{ADR}}_{\mathsf{d}}:\bW_h\times\Psi_h\to \mathbb{R}\) defined by
\begin{align}\label{eq:A:ARD.disc}\mathcal{A}_{\mathsf{d}}^{\mathsf{ADR}}(\bzeta_h,\varphi_h;\btau_h,\psi_h):=\mathcal{A}(\bzeta_h,\varphi_h;\btau_h,\psi_h)+d_2^{\hat{\bu}_h}(\btau_h,\varphi_h),\quad \forall (\bzeta_h,\varphi_h), (\btau_h,\psi_h)\in\bW_h\times \Psi_h.\end{align}
Thus,  to show the well-posedness of \eqref{eq:weak,3h}-\eqref{eq:weak,4h}, we now analyse  the problem: Find \((\bzeta_h,\varphi_h)\in\bW_h\times \Psi_h\) such that \[\mathcal{A}_{\mathsf{d}}^{\mathsf{ADR}}(\bzeta_h,\varphi_h;\btau_h,\psi_h)=H_h(\psi_h),\quad \forall \psi_h\in \Psi_h.\]
From the boundedness of \(d_2^{\hat{\bu}_h}(\bullet,\bullet)\),{\begin{align}\label{eq:d2.bound.disc}d_2^{\hat{\bu}_h}(\btau_h,\varphi_h)\leq C^d_{\alpha}\|\hat{\bu}_h\|_{0,3,\Omega}\,\tnorm{(\btau_h,\psi_h)}_{\bW\times\widehat{\Psi}(\cT_h)}\tnorm{(\bzeta_h,\varphi_h)}_{\bW\times\widehat{\Psi}(\cT_h)},\end{align}}
{where $C^d_{\alpha}=\alpha^{-\frac12}(1+{C_{\mathsf{emb}}}\alpha^{-\frac12})$}, combined with \eqref{eq:A.stab.disc} and \eqref{eq:A:ARD.disc}, we can assert that
\[\sup_{(\bzero,0)\neq(\btau_h,\psi_h)\in {\bW_h\times\Psi_h}}\frac{\mathcal{A}_{\mathsf{d}}^{\mathsf{ADR}}(\bzeta_h,\varphi_h;\btau_h,\psi_h)}{\tnorm{(\btau_h,{\psi_h})}_{\bW\times\widehat{\Psi}(\cT_h)}}\geq\bigl[C_d-C^d_{\alpha}\|\hat{\bu}_h\|_{0,3,\Omega}\bigr]\tnorm{(\bzeta_h,\varphi_h)}_ {\bW\times\widehat{\Psi}(\cT_h)}, \]
for all \((\bzeta_h,\varphi_h)\in \bW_h\times \Psi_h.\) Hence, {under the assumption that
$\|\hat{\bu}_h\|_{0,3,\Omega}\leq \frac{C_d}{2C^d_{\alpha}}$, we} arrive at
\begin{equation}\label{eq:ARD.disc}
\sup_{(\bzero,0)\neq(\btau_h,\psi_h)\in {\bW_h\times\Psi_h}}\frac{\mathcal{A}_{\mathsf{d}}^{\mathsf{ADR}}(\bzeta_h,\varphi_h;\btau_h,\psi_h)}{\tnorm{(\btau_h,\psi_h)}_ {\bW\times\widehat{\Psi}(\cT_h)}}\geq\frac{C_d}{2} \tnorm{(\bzeta_h,\varphi_h)}_ {\bW\times \widehat{\Psi}(\cT_h)}.
\end{equation}
Therefore, the discrete $\mathsf{ADR}$ system~\eqref{eq:weak,3h}-\eqref{eq:weak,4h} is well-posed.
Additionally, invoking \eqref{eq:A:ARD.disc} in \eqref{eq:ARD.disc}, and  applying H\"{o}lder's inequality,
gives the dependence on data {\(\tnorm{(\bzeta_h,\varphi_h)}_{\bW\times\widehat{\Psi}(\cT_h)}\leq C^{\mathsf{ADR}}_{\mathsf{d}}\|m\|_{0,\frac65,\Omega},\)} where {\(C^{\mathsf{ADR}}_{\mathsf{d}}:={2}{C_d}^{-1}(1+C_{\mathsf{emb}}\alpha^{-\frac12})\)}, which concludes the proof.
\end{proof}
}
\subsection{Fixed-point analysis in discrete framework}
{Analogous to the continuous case (cf. \eqref{Operator-S.V->V}), we now define the fixed-point operator in the discrete framework as follows
\begin{equation}\label{Operator-S_h.Vh->Vh}
\cS_h: \bV_h  \to  \bV_h, \quad \hat{\bu}_h \mapsto \cS_h(\hat{\bu}_h):= \cS_{1,h}^{\mathsf{DF}}(\cS_{2,h}^{\mathsf{ADR}}(\hat{\bu}_h)),\quad \hat{\bu}_h\in\bV_h,
\end{equation}
where the maps~\(\cS_{1,h}^{\mathsf{DF}_1}\) and \( \cS_{2,h}^{\mathsf{ADR}_2}\) are defined over the respective discrete spaces as follows
\[ \cS_h^{\mathsf{DF}}:\Psi_h\to \bV_h \times \rQ_h, \quad \hat{\varphi}_h \mapsto \cS_h^{\mathsf{DF}}(\hat{\varphi}_h)= (\cS^{\mathsf{DF}}_{1,h}(\hat{\varphi}_h), \cS^{\mathsf{DF}}_{2,h}(\hat{\varphi}_h)):= (\bu_h, p_h), \]
where $( \bu_h,p_h) \in \bV_h \times \rQ_h$ is the unique solution of the Darcy--Forchheimer equations as stated in Theorem~\ref{th:discrete-DF}.
Further, we define the solution operator associated with the mixed ADR equations as
\[ \cS_h^{\mathsf{ADR}}: \bV_h \to  \bW_h\times \Psi_h, \quad \hat{\bu}_h \mapsto \cS_h^{\mathsf{ADR}}(\hat{\bu}_h)=(\cS^{\mathsf{ADR}}_{1,h}(\hat{\bu}_h),\cS^{\mathsf{ADR}}_{2,h}(\hat{\bu}_h)):= ( \bzeta_h, \varphi_h),\]
where $(\bzeta_h, \varphi_h)$ is the unique solution of the ADR equations, as stated in Theorem~\ref{th:discrete-ARD}.

For $\hat\bu_h\in\bV_h$ with $\|\hat\bu_h\|_{0,3,\Omega}\leq r^{\mathsf{d}}_{\mathsf{ADR}}:=\frac{C_d}{2C^d_\alpha}$, the map $\cS^{\mathsf{ADR}}_h$ is well-defined by Theorem~\ref{th:discrete-ARD}, and, as in the continuous case, the discrete problem \eqref{eq:weak_Galerkin} is equivalent to finding a fixed point $\varphi_h$ of $\mathcal{T}_h:=\cS^{\mathsf{ADR}}_{2,h}\circ\cS^{\mathsf{DF}}_{1,h}$. Denote by $C^{\mathsf{d}}_{\mathsf{emb}}$ the constant of the discrete Sobolev embedding $\|\chi_h\|_{0,6,\Omega}\leq C^{\mathsf{d}}_{\mathsf{emb}}|\chi_h|_{1,\cT_h}$ for broken $\rH^1$ functions \cite[Theorem~5.3]{dipietro2012mathematical}, so that $\|\psi\|_{0,6,\Omega}\leq(1+C^{\mathsf{d}}_{\mathsf{emb}}\alpha^{-\frac12})\|\psi\|_{\widehat\Psi(\cT_h)}$, and set
\begin{gather*}
\mathsf{M}_0^{\mathsf{d}}:=\|\bf(\bzero)\|_{0,\Omega}+\alpha_1\bigl(1+C^{\mathsf{d}}_{\mathsf{emb}}\alpha^{-\frac12}\bigr)C^{\mathsf{ADR}}_{\mathsf{d}}\|m\|_{0,\frac65,\Omega},\\
K_h:=\bigl\{\hat\varphi_h\in\Psi_h:\ \|\hat\varphi_h\|_{\widehat\Psi(\cT_h)}\leq C^{\mathsf{ADR}}_{\mathsf{d}}\|m\|_{0,\frac65,\Omega}\bigr\}.
\end{gather*}
Exactly as in Lemmas~\ref{lemma:ball} and \ref{lemma:continuity-S} (with Lemma~\ref{lem:monotone} applied on the discrete kernel $\{\bv_h\in\bV_h: b_{1,h}(\bv_h,q_h)=0\ \forall q_h\in\rQ_h\}$, which consists of divergence-free functions since $\vdiv\mathbb{RT}_k(\cT_h)\subseteq\mathbb{P}_k(\cT_h)$), we can first show that
 under the small data assumption $4\rF^{-\frac13}\max\{\mathsf{M}_0^{\mathsf{d}},(\mathsf{M}_0^{\mathsf{d}})^{\frac12}\}\leq r^{\mathsf{d}}_{\mathsf{ADR}}$, the map $\mathcal{T}_h$ is well-defined on $K_h$ and $\mathcal{T}_h(K_h)\subseteq K_h$. Secondly, we can also show that
$\mathcal{T}_h$ is continuous on $K_h$ (the discrete counterparts of \eqref{continuity-S} and \eqref{eq:Lip.ARD} hold with the same structure, $C_c$ and $C^c_\alpha$ being replaced by $C_d$ and $C^d_\alpha$).

Since $K_h$ is a nonempty, convex and compact subset of the finite-dimensional space $\Psi_h$, the Brouwer fixed-point theorem provides a fixed point of $\mathcal{T}_h$.

We next state the discrete stability theorem for the fully-coupled discrete $\mathsf{DF}$-$\mathsf{ADR}$ system (whose proof follows similar to the continuous case as in Theorem~\ref{th:wellposedness}).
\begin{theorem}[Well-posedness of the coupled discrete problem]\label{th:wellposedness.disc}
Assume \eqref{eq:f.lipschitz} and $4\rF^{-\frac13}\max\{\mathsf{M}_0^{\mathsf{d}},(\mathsf{M}_0^{\mathsf{d}})^{\frac12}\}\leq r^{\mathsf{d}}_{\mathsf{ADR}}$. Then  \eqref{eq:weak_Galerkin} has at least one solution $(\bu_h,p_h,\bzeta_h,\varphi_h) \in \bV_h \times \rQ_h\times \bW_h \times \Psi_h$ with $\varphi_h\in K_h$, and
\begin{gather*}
   \|\bu_h\|_{\bV}\leq\mathcal{M}(\mathsf{M}_0^{\mathsf{d}}),\quad \|p_h\|_{\widehat{\rQ}(\cT_h)}\leq C^{\mathsf{DF}}_{\mathsf{d}}\max\bigl\{(\mathsf{M}^{\mathsf{d}}_0)^{\frac12},\mathsf{M}^{\mathsf{d}}_0,(\mathsf{M}^{\mathsf{d}}_0)^2\bigr\},\\
   \tnorm{(\bzeta_h,\varphi_h)}_{\bW\times\widehat\Psi(\cT_h)}\leq C^{\mathsf{ADR}}_{\mathsf{d}}\|m\|_{0,\frac65,\Omega},
\end{gather*}
where the constants do not depend on the model parameters $\kappa,\rF,\varrho_0,$ nor the discretisation parameter \(h.\)
\end{theorem}

The uniqueness of discrete solutions  is not needed for the convergence analysis.  The error estimates of Theorem~\ref{th:Cea} below hold for every discrete solution provided by Theorem~\ref{th:wellposedness.disc}, and they only involve the regularity of the continuous solution.

In the next section, we estimate the error between the exact and approximate solutions in the proposed weighted-norms, showing that the convergence is robust with respect to the material parameters.
\subsection{ {\it{A priori}} error analysis}\label{sec:apriori}
We start this section by proving the following parameter-robust \textit{a priori} error estimate for the coupled system \eqref{eq:strong}. We measure the best-approximation errors in the norms of the analysis, namely; $\mathrm{dist}(\bu,\bV_h):=\inf_{\bv_h\in\bV_h}\|\bu-\bv_h\|_{\bV}$, $\mathrm{dist}(p,\rQ_h):=\inf_{q_h\in\rQ_h}\|p-q_h\|_{\widehat\rQ(\cT_h)}$, $\mathrm{dist}(\bzeta,\bW_h):=\inf_{\bw_h\in\bW_h}\|\bzeta-\bw_h\|_{\bW}$, and $\mathrm{dist}(\varphi,\Psi_h):=\inf_{\psi_h\in\Psi_h}\bigl(\|\varphi-\psi_h\|_{\widehat\Psi(\cT_h)}+\varrho_0^{\frac12}\|\varphi-\psi_h\|_{0,\Omega}\bigr)$, and we set
\[E_h:=\mathrm{dist}(\bu,\bV_h)+\mathrm{dist}(\bzeta,\bW_h)+\mathrm{dist}(\varphi,\Psi_h),\qquad \widehat{E}_h:=E_h+\mathrm{dist}(p,\rQ_h).\]
We also recall that $b_{1,h}(\bullet,\bullet)$ and $b_{2,h}(\bullet,\bullet)$ are bounded on $\bV_h\times\widehat\rQ(\cT_h)$ and on $\bW_h\times\widehat\Psi(\cT_h)$, respectively, with constants independent of $h$ and of the model parameters, as a consequence of an element-wise integration by parts and the discrete trace inequality (see \cite[Section~4]{dhpr_sinum26}), and that $\beta_{\mathsf{d}}>0$ denotes the ($h$- and parameter-independent) discrete inf-sup constant of $b_{1,h}$ \cite[Lemma~4.1]{dhpr_sinum26}.

\begin{theorem}[Quasi-optimality]\label{th:Cea}
Let $(\bu,p,\bzeta,\varphi)$ be a solution of \eqref{eq:weak} given by Theorem~\ref{th:wellposedness}, with $\varphi\in\rL^\infty(\Omega)$, and let $(\bu_h,p_h,\bzeta_h,\varphi_h)$ be a solution of \eqref{eq:weak_Galerkin} given by Theorem~\ref{th:wellposedness.disc}. Assume, in addition, that
\begin{equation}\label{eq:KE}
K_E:=\frac{2}{C_d}\,\kappa^{\frac12}\,\alpha_1\,\alpha^{-\frac12}\bigl(1+C^{\mathsf{d}}_{\mathsf{emb}}\alpha^{-\frac12}\bigr)\|\varphi\|_{0,\infty,\Omega}\leq\frac12.
\end{equation}
Then, there exists  \(C_{\mathsf{C\acute{e}a}}>0\), independent of $\kappa,\rF,\varrho_0$ and $h$, such that, for every such discrete solution,
  \begin{subequations}  \begin{align}\label{eq:Cea}
        \|\bu-\bu_h\|_{\bV}+\tnorm{(\bzeta-\bzeta_h,\varphi-\varphi_h)}_{\bW\times\widehat{\Psi}(\cT_h)} &\leq C_{\mathsf{C\acute{e}a}}\bigl(E_h+E_h^{\frac12}\bigr),\\
        \|p-p_h\|_{\widehat{\rQ}(\cT_h)}&\leq C_{\mathsf{C\acute{e}a}}\bigl(\widehat{E}_h+E_h^{\frac12}\bigr).\label{eq:Cea.p}
    \end{align}\end{subequations}
\end{theorem}
The boundedness of the continuous concentration required in Theorem~\ref{th:Cea} is not an additional assumption as far as convergence rates are concerned: under the regularity hypotheses of Theorem~\ref{th:convergence} below (with $l\geq1$), one has $\varphi=\varphi_1+\varphi_2$ with $\varphi_1\in\rH^2(\Omega)\hookrightarrow\rL^\infty(\Omega)$ and $\varphi_2\in\rW^{1,6}(\Omega)\hookrightarrow\rL^\infty(\Omega)$ for $d\leq3$. Note also that only the concentration of the continuous problem needs to be bounded. This is the only place in the analysis where a bound beyond the natural setting is used. The reason is the same as discussed after Theorem~\ref{th:wellposedness}: the advective coupling of $\bu-\bu_h$ has to be controlled through $\kappa^{-\frac12}\|\bu-\bu_h\|_{0,\Omega}$ rather than through $\|\bu-\bu_h\|_{0,3,\Omega}$. At the continuous level this was achieved by duality, whereas a discrete (Aubin--Nitsche-type) counterpart of that argument would only control $\|\varphi-\varphi_h\|_{0,\Omega}$ up to the approximation error of the adjoint problem times the energy error of the concentration, which again contains $\|\bu-\bu_h\|_{0,3,\Omega}$, thus spoiling the convergence rates.
\begin{proof}
Since $b_{2,h}$ and $H_h$ coincide with $b_2$ and $H$ on discrete arguments, subtracting \eqref{eq:weak,3h}--\eqref{eq:weak,4h} from \eqref{eq:weak,3}--\eqref{eq:weak,4} (tested with discrete functions), and using $d_2^{\bu}(\bw_h,\varphi)-d_2^{\bu_h}(\bw_h,\varphi_h) = d_2^{\bu-\bu_h}(\bw_h,\varphi)+d_2^{\bu_h}(\bw_h,\varphi-\varphi_h)$, we find that for any $(\bzeta_h^*,\varphi_h^*)\in\bW_h\times\Psi_h$,
\begin{multline*}
\cA^{\mathsf{ADR}}_{\mathsf{d}}(\bzeta_h-\bzeta_h^*,\varphi_h-\varphi_h^*;\bw_h,\psi_h) = \cA^{\mathsf{ADR}}_{\mathsf{d}}(\bzeta-\bzeta_h^*,\varphi-\varphi_h^*;\bw_h,\psi_h)+d_2^{\bu-\bu_h}(\bw_h,\varphi)\\
\forall(\bw_h,\psi_h)\in\bW_h\times\Psi_h,
\end{multline*}
where $\cA^{\mathsf{ADR}}_{\mathsf{d}}$ is the form \eqref{eq:A:ARD.disc} with $\hat\bu_h = \bu_h$ (note that $\|\bu_h\|_{0,3,\Omega}\leq r^{\mathsf{d}}_{\mathsf{ADR}}$). The discrete inf-sup estimate \eqref{eq:ARD.disc}, the boundedness of $\cA^{\mathsf{ADR}}_{\mathsf{d}}$ on $(\bW\times\widehat\Psi(\cT_h))\times(\bW_h\times\Psi_h)$ (with a constant $C_{\mathsf{bd}}$ depending only on $\alpha$ and the data, through \eqref{eq:d2.bound.disc} and $\|\bu_h\|_{0,3,\Omega}\leq r^{\mathsf{d}}_{\mathsf{ADR}}$), and the bound $d_2^{\bu-\bu_h}(\bw_h,\varphi)\leq\alpha^{-\frac12}\|\varphi\|_{0,\infty,\Omega}\|\bu-\bu_h\|_{0,\Omega}\|\bw_h\|_{\bW}$ show
\begin{multline*}
\frac{C_d}{2}\tnorm{(\bzeta_h-\bzeta_h^*,\varphi_h-\varphi_h^*)}_{\bW\times\widehat\Psi(\cT_h)}\leq C_{\mathsf{bd}}\tnorm{(\bzeta-\bzeta_h^*,\varphi-\varphi_h^*)}_{\bW\times\widehat\Psi(\cT_h)}\\
+\alpha^{-\frac12}\|\varphi\|_{0,\infty,\Omega}\,\kappa^{\frac12}\,\kappa^{-\frac12}\|\bu-\bu_h\|_{0,\Omega}.
\end{multline*}
By the triangle inequality and the arbitrariness of $(\bzeta_h^*,\varphi_h^*)$, we arrive at
\begin{multline}\label{eq:stab.ard.cea}
\tnorm{(\bzeta-\bzeta_h,\varphi-\varphi_h)}_{\bW\times\widehat\Psi(\cT_h)}\leq C^{\mathsf{ADR}}_{\mathsf{C\acute{e}a}}\bigl(\mathrm{dist}(\bzeta,\bW_h)+\mathrm{dist}(\varphi,\Psi_h)\bigr)\\
+\frac{2}{C_d}\alpha^{-\frac12}\|\varphi\|_{0,\infty,\Omega}\,\kappa^{\frac12}\,\kappa^{-\frac12}\|\bu-\bu_h\|_{0,\Omega},
\end{multline}
with $C^{\mathsf{ADR}}_{\mathsf{C\acute{e}a}}:=1+2C_{\mathsf{bd}}C_d^{-1}$.

Regarding the velocity, let $\bV_h^0:=\{\bv_h\in\bV_h:\ b_{1,h}(\bv_h,q_h)=0\ \forall q_h\in\rQ_h\}$; since $\vdiv\bV_h\subseteq\rQ_h$, its elements are divergence-free. We first note that there exists $\bu_h^*\in\bV_h^0$ such that
\begin{equation}\label{eq:kernel.approx}
\|\bu-\bu_h^*\|_{\bV}\leq\Bigl(1+\frac{\|b_{1,h}\|}{\beta_{\mathsf{d}}}\Bigr)\mathrm{dist}(\bu,\bV_h).
\end{equation}
Indeed, given any $\bv_h\in\bV_h$, the discrete inf-sup condition implies that the (finite-dimensional) operator $\bv_h\mapsto b_{1,h}(\bv_h,\bullet)$ is surjective onto $\rQ_h'$ and admits a right inverse with norm bounded by $\beta_{\mathsf{d}}^{-1}$ (see, e.g., \cite[Part~V]{ern2021galerkin}; the argument holds in reflexive Banach spaces); hence, there exists $\bz_h\in\bV_h$ with $b_{1,h}(\bz_h,q_h) = b_{1,h}(\bu-\bv_h,q_h)$ for all $q_h\in\rQ_h$ and $\|\bz_h\|_{\bV}\leq\beta_{\mathsf{d}}^{-1}\|b_{1,h}\|\|\bu-\bv_h\|_{\bV}$. Since $\vdiv\bu=0$, the function $\bu_h^*:=\bv_h+\bz_h$ belongs to $\bV_h^0$, and $\|\bu-\bu_h^*\|_{\bV}\leq(1+\beta_{\mathsf{d}}^{-1}\|b_{1,h}\|)\|\bu-\bv_h\|_{\bV}$; taking the infimum over $\bv_h$ implies \eqref{eq:kernel.approx}. This argument only concerns the linear constraint and does not use the nonlinearity, nor any Hilbert structure (it is the Banach counterpart of \cite[Theorem~2.6]{gatica2014simple}).

Let now $\boldsymbol{e}_h:=\bu_h-\bu_h^*\in\bV_h^0$, and set $a:=\kappa^{-\frac12}\|\boldsymbol{e}_h\|_{0,\Omega}$, $b:=\rF^{\frac13}\|\boldsymbol{e}_h\|_{0,3,\Omega}$, so that $\|\boldsymbol{e}_h\|_{\bV}=a+b$. Testing the continuous and discrete momentum equations with $\boldsymbol{e}_h$ and subtracting, the pressure terms disappear, because $b_{1,h}(\boldsymbol{e}_h,p_h)=0$ and $b_1(\boldsymbol{e}_h,p) = -\int_\Omega p\,\vdiv\boldsymbol{e}_h = 0$. Therefore,
\[\langle\mathcal{N}(\bu_h)-\mathcal{N}(\bu_h^*),\boldsymbol{e}_h\rangle = \langle\mathcal{N}(\bu)-\mathcal{N}(\bu_h^*),\boldsymbol{e}_h\rangle-\int_\Omega(\bf(\varphi)-\bf(\varphi_h))\cdot\boldsymbol{e}_h,\]
and \eqref{eq:monotone}, \eqref{eq:loc.lip}, and \eqref{eq:f.lipschitz} yield
\[a^2+\frac12b^3\leq L_{\mathcal{N}}\|\bu-\bu_h^*\|_{\bV}(a+b)+\kappa^{\frac12}\alpha_1\|\varphi-\varphi_h\|_{0,\Omega}\,a,\qquad L_{\mathcal{N}}:=1+\|\bu\|_{\bV}+\|\bu_h^*\|_{\bV}.\]
Next, by $|\Omega|\leq1$, the discrete Sobolev embedding, \eqref{eq:stab.ard.cea}, and $\kappa^{-\frac12}\|\bu-\bu_h\|_{0,\Omega}\leq\|\bu-\bu_h^*\|_{\bV}+a$, we have
\begin{gather*}
\kappa^{\frac12}\alpha_1\|\varphi-\varphi_h\|_{0,\Omega}\leq c_\varphi\bigl(\mathrm{dist}(\bzeta,\bW_h)+\mathrm{dist}(\varphi,\Psi_h)\bigr)+K_E\bigl(\|\bu-\bu_h^*\|_{\bV}+a\bigr),\\
c_\varphi:=\alpha_1\bigl(1+C^{\mathsf{d}}_{\mathsf{emb}}\alpha^{-\frac12}\bigr)C^{\mathsf{ADR}}_{\mathsf{C\acute{e}a}},
\end{gather*}
where we have used $\kappa\leq1$ and $K_E$ is defined in \eqref{eq:KE}. Since $K_E\leq\frac12$, we can absorb $K_Ea^2$ into the left-hand side and obtain
\[\frac12\bigl(a^2+b^3\bigr)\leq R\,(a+b),\qquad R:=(L_{\mathcal{N}}+K_E)\|\bu-\bu_h^*\|_{\bV}+c_\varphi\bigl(\mathrm{dist}(\bzeta,\bW_h)+\mathrm{dist}(\varphi,\Psi_h)\bigr).\]
As in the proof of Lemma~\ref{lem:monotone}, $a^2+b^3\geq\min\{\frac14(a+b)^2,\frac18(a+b)^3\}$, so that $\min\{\frac{a+b}{8},\frac{(a+b)^2}{16}\}\leq R$, that is, $\|\boldsymbol{e}_h\|_{\bV}\leq\max\{8R,4R^{\frac12}\}$. Together with \eqref{eq:kernel.approx} and the triangle inequality, this shows
\begin{equation}\label{eq:final.err.est.for.u}
\|\bu-\bu_h\|_{\bV}\leq C_{\bu}\bigl(E_h+E_h^{\frac12}\bigr),
\end{equation}
where $C_{\bu}$ depends only on $\beta_{\mathsf{d}}$, $\|b_{1,h}\|$, $L_{\mathcal{N}}$ (which is bounded in terms of $\mathcal{M}(\mathsf{M}_0)$, see Theorem~\ref{th:wellposedness}), $c_\varphi$, and $K_E$. In particular, $C_{\bu}$ does not depend on $\kappa,\rF,\varrho_0$, $h$, nor on $\mathrm{dist}(p,\rQ_h)$.

To conclude the estimate associated with the ADR variables, inserting \eqref{eq:final.err.est.for.u} in \eqref{eq:stab.ard.cea} (with $\kappa\leq1$) gives the bound for $\tnorm{(\bzeta-\bzeta_h,\varphi-\varphi_h)}_{\bW\times\widehat\Psi(\cT_h)}$ in \eqref{eq:Cea}.

For the pressure, for any $q_h\in\rQ_h$, the discrete inf-sup condition of $b_{1,h}$ and the error equation (tested with $\bv_h\in\bV_h$),
yield
\begin{align*}\beta_{\mathsf{d}}\|p_h-q_h\|_{\widehat\rQ(\cT_h)}& \leq\sup_{\bzero\neq\bv_h\in\bV_h}\frac{b_{1,h}(\bv_h,p_h-q_h)}{\|\bv_h\|_{\bV}}\\
& \leq\|b_{1,h}\|\|p-q_h\|_{\widehat\rQ(\cT_h)}+(1+\|\bu\|_{\bV}+\|\bu_h\|_{\bV})\|\bu-\bu_h\|_{\bV}+\alpha_1\|\varphi-\varphi_h\|_{0,\Omega},\end{align*}
where we used $\kappa\le1$ in the last term. The triangle inequality, Steps 2--3, and the a priori bounds of Theorems~\ref{th:wellposedness} and \ref{th:wellposedness.disc} yield \eqref{eq:Cea.p}. This concludes the proof.
\end{proof}

\begin{remark}\label{rem:pressure,flux.robust}
(i) The velocity estimate \eqref{eq:final.err.est.for.u} is pressure-robust, in the sense that it does not involve $\mathrm{dist}(p,\rQ_h)$ (this is a consequence of $\vdiv\bV_h\subseteq\rQ_h$, cf. \cite[Corollary~4.5]{dhpr_sinum26}). (ii) The square root $E_h^{\frac12}$ is a consequence of the cubic (rather than quadratic) monotonicity of the Forchheimer term, cf. Lemma~\ref{lem:monotone}; this type of suboptimal bound is classical for Darcy--Forchheimer and related $p$-type nonlinearities (see, e.g., \cite{girault2008numerical} and references,therein). For sufficiently small $E_h$ (i.e., asymptotically), the term $E_h^{\frac12}$ dominates. Our numerical results (cf. Section~\ref{sec:numer}) nevertheless show optimal convergence in the weighted norms, which suggests that the estimate is pessimistic. If, for instance, the term $\kappa^{-1}\|\bullet\|_{0,\Omega}^2$ dominates the monotonicity, that is, $b\lesssim a$ in Step~2, one recovers the optimal bound $\|\bu-\bu_h\|_{\bV}\lesssim E_h$.
\end{remark}

\noindent
In order to provide a theoretical convergence rate, we recall some useful approximation properties.
Let $t \in (1,\infty)$ and denote by  \(\Pi_h\) the Raviart--Thomas interpolation operator onto \(\mathbb{RT}_k(\cT_h)\) (well-defined on $\bW^{1,t}(\Omega)$) and by \(P_h:\rW^{l,t}(K)\to \mathbb{P}_k(K)\),  the local orthogonal \(\rL^2\) projection. Then, as a standard consequence of the Bramble--Hilbert lemma, optimal approximation, and inverse inequalities on polynomial spaces, we can assert that
\begin{subequations}
     \begin{align}
\label{eq:approx-v}
  \|\bv- \Pi_h \bv\|_{m,t,\Omega}&\leq C_1 h^{l-m}|\bv|_{l,t,\Omega}\quad \forall \bv\in \mathbf{W}^{l,t}(\Omega), \ 1\leq l \leq k+1,\\
  \label{eq:approx-div}   \|\vdiv(\bv-\Pi_h\bv)\|_{m,t,\Omega}& \leq C_2h^{l-m}|\vdiv\bv|_{l,t,\Omega} \quad \forall \bv \in \mathbf{W}^{1,t}(\Omega),\ \text{with}\\
    \nonumber           &  \qquad \qquad \qquad \qquad \qquad  {\vdiv(\bv)\in \rW^{l,t}(\Omega)\ \text{and} \ 0\leq l \leq k+1,}\\
  \label{eq:approx-p}   \|q-P_hq\|_{m,t,\Omega}& \leq C_2 h^{l-m}|q|_{l,t,\Omega}\quad \forall q\in \rW^{l,t}(\Omega),\ 0\leq l \leq k+1, \\
    \label{eq:approx-p-K}    \|\nabla(q-P_hq)\|_{0,t,{K}}&\leq C_3 h_K^{l}|q|_{l+1,t,K}\quad \forall q\in \rW^{l+1,t}({K}),\  K\in\mathcal{T}_h,\ l \leq k,
     \end{align}
     \end{subequations}
     {for $ 0\leq m\leq l$ (with $m=0$ in \eqref{eq:approx-p} and \eqref{eq:approx-div} being the only cases needed below), where $C_1,C_2>0$ are  independent of \(h\) and $C_3>0$ is independent of $h_K$
(see \cite[Sections 4.1 and 4.5]{gatica2022p} for the first three, and \cite{ern2021finite} for the fourth  one).}

To conclude this section, the following theorem provides the theoretical rate of convergence of the mixed finite element scheme \eqref{eq:weak_Galerkin}, under suitable regularity on the exact solution.
 \begin{theorem}[Convergence rates]\label{th:convergence}
        Let the assumptions of Theorem~\ref{th:Cea} hold, and assume that the solution $(\bu,p,\bzeta,\varphi) \in \bV \times \rQ \times \bW\times \Psi$ of \eqref{eq:weak} satisfies the following regularity conditions
{\begin{gather*}\bu\in \kappa^{-\frac12}{\bH^{l}}(\Omega)\cap{\rF^{\frac13}\mathbf{W}^{l,3}(\Omega)},  \quad \vdiv\bu\in {\rH^l}(\Omega), \quad \text{and} \quad {p=p_1+p_2+p_3}, \\
\text{with} \quad {p_1\in \rH^{l}(\Omega), \quad p_2\in \kappa^{\frac12}\rH^{l+1}(\Omega), \quad \text{and}\quad  p_3 \in \rF^{-\frac13}\rW^{l+1,\frac32}(\Omega)},\\
\bzeta\in \alpha^{-\frac12}\bH^{l}(\Omega)\cap\bH(\vdiv_{\frac65},\Omega), \quad \vdiv\bzeta\in {\rW}^{l,\frac65}(\Omega), \quad \text{and} \quad \varphi=\varphi_1+\varphi_2,\\
\text{with}\quad\varphi_1\in \alpha^{\frac12}\rH^{l+1}(\Omega),\quad \varphi_2\in\rW^{l,6}(\Omega),
\end{gather*}
where \(1\leq l\leq k+1\).  }
Denote
\begin{align*}
\Theta_{\bu}&:=\kappa^{-\frac12}|\bu|_{l,\Omega}+|\vdiv\bu|_{l,\Omega}+\rF^{\frac13}|\bu|_{l,3,\Omega},\qquad \Theta_p:=|p_1|_{l,\Omega}+\kappa^{\frac12}|p_2|_{l+1,\Omega}+\rF^{-\frac13}|p_3|_{l+1,\frac32,\Omega},\\
\Theta_{\bzeta}&:=\alpha^{-\frac12}|\bzeta|_{l,\Omega}+|\vdiv\bzeta|_{l,\frac65,\Omega},\qquad \Theta_{\varphi}:=\alpha^{\frac12}|\varphi_1|_{l+1,\Omega}+|\varphi_2|_{l,6,\Omega}+\varrho_0^{\frac12}\bigl(|\varphi_1|_{l,\Omega}+|\varphi_2|_{l,\Omega}\bigr).
\end{align*}
Then, there exists a positive constant $C_\sharp$, independent of \(h, \kappa, \rF\) and \(\varrho_0,\) such that, with $\Theta:=\Theta_{\bu}+\Theta_{\bzeta}+\Theta_\varphi$,
\begin{align*}
 \|\bu-\bu_h\|_{\bV}+\tnorm{(\bzeta-\bzeta_h,\varphi-\varphi_h)}_{\bW\times\widehat{\Psi}(\cT_h)}&\leq C_\sharp\bigl(h^{l}\,\Theta+h^{\frac l2}\,\Theta^{\frac12}\bigr),\\
 \|p-p_h\|_{\widehat{\rQ}(\cT_h)}&\leq C_\sharp\bigl(h^{l}\,(\Theta+\Theta_p)+h^{\frac l2}\,\Theta^{\frac12}\bigr).
\end{align*}
    \end{theorem}
    \begin{proof}
It suffices to bound the best-approximation errors in Theorem~\ref{th:Cea}. Taking $\bv_h=\Pi_h\bu$ and using \eqref{eq:approx-v} (with $m=0$ and $t\in\{2,3\}$) and \eqref{eq:approx-div} (with $m=0$, $t=2$) gives $\mathrm{dist}(\bu,\bV_h)\leq\max\{C_1,C_2\}h^l\Theta_{\bu}$. Similarly, taking $\bw_h = \Pi_h\bzeta$ and using \eqref{eq:approx-v} with $m=0$, $t=2$, and \eqref{eq:approx-div} with $m=0$, $t=\frac65$, gives $\mathrm{dist}(\bzeta,\bW_h)\leq\max\{C_1,C_2\}h^l\Theta_{\bzeta}$. For the pressure, we employ the decomposition $p = p_1+p_2+p_3$, the definition of the sum norm \eqref{norm:Q.hat} with the (admissible) decomposition $p-P_hp = (p_1-P_hp_1)+(p_2-P_hp_2)+(p_3-P_hp_3)$, \eqref{eq:approx-p} with $m=0$, $t=2$, and \eqref{eq:approx-p-K} with $t\in\{2,\frac32\}$, combined with the trace inequality below to control the jump terms; this gives $\mathrm{dist}(p,\rQ_h)\leq C h^l\Theta_p$ (see \cite[Theorem~4.6]{dhpr_sinum26}). Analogously, for the concentration we use the decomposition $\varphi = \varphi_1+\varphi_2$, the sum norm \eqref{norm:Psi.hat} with the decomposition $\varphi-P_h\varphi = (\varphi_2-P_h\varphi_2)+(\varphi_1-P_h\varphi_1)$, \eqref{eq:approx-p} with $m=0$ and $t=6$ (for $\varphi_2$) and $t=2$ (for the reaction seminorm), and \eqref{eq:approx-p-K} with $t=2$ (for $\varphi_1$), summed over $K\in\cT_h$. We combine that with the following trace inequality to control the jump term in the broken norms
\[h_F^{1-t} \|[\![\varphi_{1} - P_h \varphi_{1}]\!] \|^t_{0,t,F} \leq C_5 h^{lt}|\varphi_{1}|^t_{l+1,t,\omega_F}, \]
where $\omega_F$ stands for the usual face patch, and $C_5$ is independent of all the parameters. This gives $\mathrm{dist}(\varphi,\Psi_h)\leq C h^l\Theta_\varphi$. Inserting these bounds in \eqref{eq:Cea}--\eqref{eq:Cea.p} concludes the proof.
    \end{proof}

 In the next section we define a variable operator preconditioner for the Newton linearisation of  \eqref{eq:weak}.
\section{Operator preconditioning}\label{sec:precond}
This section discusses preconditioners for the monolithic coupled system arising from the weighted norms introduced in the previous sections which is robust with respect to all the model parameters except for the diffusion coefficient.

First, for \(1<s<\infty\), we denote by \(\mathcal{J}_s:\rL^s(\Omega)\to\bigl(\rL^s(\Omega)\bigr)'\)  the canonical duality map defined as
\begin{equation}
\langle{\mathcal{J}}_s(\theta_1),\theta_2\rangle:=\int_{\Omega}|\theta_1|^{s-2}\theta_1\theta_2,\quad\forall\theta_1,\theta_2\in\rL^s(\Omega).\end{equation}
Similarly, mimicking the operator notation $-\nabla \vdiv$ that induces the div-div bilinear form, we define $-\nabla \mathcal{J}_{s}(\vdiv)$ as}
\begin{equation}\label{eq:duality2} {\langle-\nabla \mathcal{J}_{s}(\vdiv)(\bzeta_1),\bzeta_2\rangle:={\langle} \mathcal{J}_{s}(\vdiv\bzeta_1{),\vdiv\bzeta_2\rangle},\quad \forall \bzeta_1,\bzeta_2 \ \text{such that} \ \vdiv\bzeta_i \in \rL^s(\Omega).}\end{equation}

Denote by \(\mathcal{H}_1(\bu)\) the G\^ateaux derivative of the nonlinear operator $d_1(\bullet;\bullet,\bullet)$, i.e., \(|\bu|\bu\) at the state \(\bu\) in the direction of \(\delta\bu\), and \(\mathcal{H}_2(\ast)\) by the linearisation of \(d_2^{\ast}(\bullet,\bullet)\) around \(\ast\). The linearised system corresponding to \eqref{eq:strong} takes the following form
\begin{equation}\label{eq:strong.linearised}
\begin{aligned}
[\kappa^{-1}+\rF\mathcal{H}_1(\hat\bu)]\bu + \nabla p & = \bf(\hat\varphi) \quad \text{in $\Omega$},\\
    \vdiv \bu & = 0 \quad \text{in $\Omega$},\\
    \alpha^{-1}\bzeta- \nabla \varphi + { \alpha^{-1}} \mathcal{H}_2(\hat\bu) \varphi &=\bzero \quad \text{in $\Omega$},\\
    \vdiv \bzeta - \varrho_0 \varphi & = -m \quad \text{in $\Omega$},
\end{aligned}
 \end{equation}
with the same boundary conditions as in \eqref{eq:strong}.
In addition, consider the Newton--Raphson linearisation, which starting from an initial guess $\bu^0,p^0,\bzeta^0,\varphi^0$, and iterating on  $m = 1, 2, \ldots$ until convergence, seeks velocity, pressure, flux, and concentration increments $(\delta \bu, \delta p, \delta\bzeta, \delta\varphi)$ in the space \begin{align}\label{eq:linearised-spaces}\underline{\bV}\times \underline{\rQ}\times \underline{\bW}\times\underline{\Psi} &:=\bigl[\kappa^{-\frac12}\bL^2(\Omega)\cap\bH_{\Gamma_{\bu}}(\vdiv,\Omega)\cap\{\rF\mathcal{H}_1(\hat{\bu})\}^\frac13\bL^3(\Omega)\bigr]\nonumber\\ &\quad \times \bigl[\rL^2(\Omega)+\kappa^\frac12\rH_{\Gamma_p}^1(\Omega)+\{\rF{\mathcal{H}_1(\hat{\bu})}\}^{-\frac13}\rW_{\Gamma_p}^{1,\frac32}(\Omega)\bigr]\nonumber\\&\quad\times \bigl[(\mathcal{H}_2^{-1}(\hat{\bu})\alpha)^{-\frac12}\bL^2\cap\bH_{\Gamma_p}(\vdiv_{\frac65},\Omega)\bigr]\times \bigl[(\mathcal{H}^{-1}_2(\hat{\bu})\alpha)^{\frac12}\rH^1(\Omega)+\rL^6(\Omega)\bigr]
,\end{align}
such that the following holds in $\underline{\bV}'\times \underline{\rQ}'\times\underline{\bW}'\times \underline{\Psi}'$
{\begin{equation}\label{eq:newton}
\begin{pmatrix}[\kappa^{-1}+\rF
\mathcal{H}({\bu}^m)]\mathcal{I} & \nabla&0&0 \\
    -\vdiv  &0 &0&0\\\alpha^{-1}\varphi^m&0&\alpha^{-1}\mathcal{I}&-\nabla+\alpha^{-1}\bu^m\\0&0&\vdiv&-\varrho_0\mathcal{I}
    \end{pmatrix}\begin{pmatrix}\delta \bu \\ \delta p\\\delta\bzeta\\\delta\varphi\end{pmatrix}= \begin{pmatrix}
      F - \{\kappa^{-1}+|\bu^m|\}({\bu}^m) - \nabla{p}^m \\ G +\vdiv{\bu}^m  \\-\alpha^{-1}\bzeta^m+\nabla\varphi^m-\alpha^{-1}\bu^m\varphi^m\\H-\vdiv\bzeta^m-\varrho_0\varphi^m
    \end{pmatrix}. \end{equation}}

Here, $\mathcal{I}$ denotes the identity operator  at the continuous level, but for the discrete case it stands for the corresponding mass matrix (either for velocity or pressure or flux or concentration). Further, we note that, in weak form, the action of the operator $\rF\mathcal{H}({\bu}^m)\mathcal{I}:\underline{\bV}\to \underline{\bV}'$ reads
\[ \langle \rF \mathcal{H}({\bu}^m)(\delta \bu), \bv \rangle := \int_\Omega \rF |\bu^m|\delta\bu \cdot \bv + \int_\Omega \rF |\bu^m|^{-1}(\bu^m \cdot \delta\bu) (\bu^m \cdot \bv), \]
and \[\langle\mathcal{H}_2(\bu^m)\btau,\varphi\rangle:=\int_{\Omega}(\bu^m\cdot\btau)\varphi.\]
Then, update the solution as \((\bu^{m+1},p^{m+1},\bzeta^{m+1},\varphi^{m+1}) = (\bu^m + \delta \bu, p^m + \delta p,\bzeta^m+\delta\bzeta,\varphi^m+\delta\varphi)\).
\noindent
Next, we follow the approach recently advanced in \cite[Section~3]{cai2026parameter} to propose a robust block-diagonal preconditioner for \eqref{eq:newton}.
We construct block-diagonal preconditioners for the Darcy--Forchheimer velocity and pressure blocks following
\cite{dhpr_sinum26}. On the other hand, and differently from \cite{cai2026parameter} and \cite{dhpr_sinum26}, we readily note that the sub-blocks associated with the ARD equations consist of a perturbed, \emph{non-symmetric} saddle-point of the form
\begin{equation}\label{abstract-block} \begin{pmatrix} \alpha^{-1}\bbA & \bbB^* + \alpha^{-1}\bbD \\ \bbB & - \varrho_0 \bbC\end{pmatrix},\end{equation}
where we are abusing of notation to emphasise the parameter dependence of the different blocks.
Let us first concentrate on the discrete case since it simplifies the discussion as $\bu_h^m \in \bL^{\infty}(\Omega)$ and so the Banach structure is not required. In this case, $\bbA:\bH(\vdiv,\Omega)\to \bH(\vdiv,\Omega)'$ denotes the  $\bL^2(\Omega)$ mass map in $\bH(\vdiv,\Omega)$, $\bbB:\bH(\vdiv,\Omega)\to \rL^2(\Omega)'$ is the divergence operator, $\bbD$ is the \textit{mixed} advection coupling operator ($\bbD\varphi \approx \bu^m\varphi$), and $
\bbC:\rL^2(\Omega) \to \rL^2(\Omega)'$ is the scalar mass map in $\rL^2(\Omega)$.
Therefore a typical Riesz map preconditioner, also based on a Schur complement approximation, is motivated by \cite[Sect. 3.1]{benzi2005numerical} (see also \cite[Section 1]{baerland2020observation} for the symmetric, unperturbed case). The guiding principle is to scale the flux space with $\alpha^{-1}$, the concentration space with $\varrho_0$, and ensure that the Schur complement is spectrally equivalent to a properly weighted  operator.
In this case, a block-diagonal preconditioner
reads as follows (using {a complete Schur complement structure for the (2,2)-block and an incomplete one for the (1,1)-block})
\begin{equation}\label{abstract-prec}
 \mathrm{diag}(\cR,\cS) = \begin{pmatrix}(\alpha^{-1}\bbA +  \bbB^*\bbB)^{-1} & 0 \\ 0 & ([1+\varrho_0]\,\bbC)^{-1}+(\varrho_0 \bbC + \alpha \bbB\bbA^{-1}\bbB^* + \bbB\bbA^{-1}\bbD)^{-1}
\end{pmatrix},
\end{equation}
or, equivalent Schur complement structure for the (1,1)-block
(and note that it is also possible to derive block-triangular versions, see for e.g., \cite{olshanskii2007pressure}).
For the the advection term the preconditioner part $\bbB\bbA^{-1}\bbD$ is not symmetric and it corresponds to   perturbation with respect to the diffusion term $\bbB\bbA^{-1}\bbB^*$ if $
\bu^m$ is bounded. However, for $\alpha \ll 1$,  the advection term $\bbD$ becomes dominant and the Schur complement changes from diffusion--reaction to transport--reaction and it can destroy the spectral equivalence. For a fully computable preconditioner associated with \eqref{abstract-prec}, we approximate the top diagonal block by $(\alpha^{-1} \cI - \alpha^{-1}\nabla(\vdiv) )^{-1}$,
where $\cI - \nabla(\vdiv)$
is the Raviart--Thomas full matrix.
\noindent
Finally, inspired by \eqref{abstract-prec}, {and considering another formulation inspired by the primal formulation,} we propose the following preconditioners {($\cB^{1,1}_{\mathsf{DFADR}}, \cB^{1,2}_{\mathsf{DFADR}}, \cB^{2,1}_{\mathsf{DFADR}},$ and $\cB^{2,2}_{\mathsf{DFADR}}$)} for the coupled monolithic problem~\eqref{eq:strong.linearised}, where
\begin{align}\label{precond.linear.B2}
\cB^{i,j}_{\mathsf{DFADR}}:=\begin{pmatrix}
\mathcal{P}_i & 0 & 0 & 0\\
0 &\mathcal{Q}_i& 0 & 0\\
0 & 0 & \mathcal{R}_j & 0\\
0 & 0 & 0 & \mathcal{S}_j
\end{pmatrix},
\end{align}
with \begin{align*}
\mathcal{P}_1&:={\bigl([\kappa^{-1}+\rF\mathcal{H}(\hat{\bu})]\,\mathcal{I}-[\kappa^{-1}+\rF\mathcal{H}(\hat{\bu})]\nabla\vdiv\bigr)^{-1}},\quad \text{and}\quad\mathcal{Q}_1:=[\kappa^{-1}+\rF\mathcal{H}(\hat{\bu})]^{-1}{\mathcal{I}}^{-1}, \\
\mathcal{P}_2&:={\bigl([\kappa^{-1}+\rF\mathcal{H}(\hat{\bu})]\mathcal{I}-\nabla\vdiv\bigr)^{-1}},\quad \text{and}\quad\mathcal{Q}_2:={\mathcal{I}^{-1}+(-\kappa\Delta)^{-1}-\bigl([\rF\mathcal{H}(\hat{\bu})]^{-1}\Delta_{\frac32}\bigr)^{-1}},
\\\mathcal{R}_1&:=(\alpha^{-1}\cI -  \nabla\vdiv+\alpha^{-1}\hat\bu\nabla\cdot)^{-1},\quad \text{and}\quad\mathcal{S}_1:={(1+\varrho_0
)\,\mathcal{I}^{-1}+ (\varrho_0\mathcal{I}-\alpha\Delta)^{-1}},
\\\mathcal{R}_2&:=(\alpha^{-1}\cI -  \nabla\vdiv)^{-1},\quad \text{and}\quad\mathcal{S}_2:={(1+\varrho_0
)\,\mathcal{I}^{-1}+ (\varrho_0\mathcal{I}-\alpha\Delta+\hat\bu\cdot\nabla)^{-1}}.
\end{align*}
}

\section{Numerical results}\label{sec:numer}
In this section, we illustrate the accuracy and performance of the proposed scheme and operator preconditioning through several numerical experiments. We verify that the convergence rates are robust. The numerical implementations are based on the open-source  finite element libraries \texttt{Gridap} \cite{badia22} and \texttt{FEniCS} \cite{alnaes2015fenics}.

\smallskip
\noindent\textbf{Example~1 (convergence against smooth solutions in 2D).}
We first verify convergence with  unity parameters $\varrho_0 = \alpha = \kappa = \rF= \alpha_0 = \alpha_1 =1.$ The domain is the unit square $\Omega = (0,1)^2$, with $\Gamma_{\bu}$ the bottom and left edges and $\Gamma_p$ the remainder of the boundary. The manufactured velocity, pressure, and concentration are
\begin{gather*}
 \bu = \begin{pmatrix}
     \cos(\pi x)\sin(\pi y)\\-\sin(\pi x) \cos(\pi y)
 \end{pmatrix}, \quad p = \sin(\pi x) \sin (\pi y),
 \quad \varphi = \cos(2\pi x) \cos (2\pi y).
\end{gather*}
The Darcy--Forchheimer velocity is divergence-free, and the forcing term is chosen as
$ \bf(\varphi) = (\alpha_0 + \alpha_1 \varphi) \bg$, where $\bg = (0,1)^{\tt t}$, which clearly fulfills the condition \eqref{eq:f.lipschitz} (with constant $\alpha_1|\bg|=\alpha_1$). With these solutions we construct an exact total flux, and additional source terms {(so that the momentum balance is satisfied)} and non-homogeneous essential and natural boundary conditions required for the convergence assessment. The error history associated with the
proposed mixed finite element  method on a sequence of successively refined partitions of the domain  are presented in Table~\ref{num:conv.rates-unity}, where we also tabulate rates of error decay computed as
$\texttt{rate}  =\log(e_{(\bullet)}/\tilde{e}_{(\bullet)})[\log(h/\tilde{h})]^{-1}$,
where $e,\tilde{e}$ denote errors generated on two
consecutive  meshes of sizes $h$ and~$\tilde{h}$, respectively. The results indicate an optimal $O(h^{k+1})$ convergence in all fields, where we are using the natural norms, without parameter-weighting (here $\|\bullet\|_{3,\vdiv,\Omega}:=\|\bullet\|_{0,3,\Omega}+\|\vdiv\bullet\|_{0,\Omega}$ and $\|\bullet\|_{\vdiv_{6/5},\Omega}:=\|\bullet\|_{0,\Omega}+\|\vdiv\bullet\|_{0,\frac65,\Omega}$, cf. Section~\ref{sec:intro}), and also the parameter-weighted norms. We remark that these optimal rates are better than those guaranteed by Theorem~\ref{th:convergence} (cf. Remark~\ref{rem:pressure,flux.robust}). {Sample numerical solutions (obtained with the second-order scheme and on the finest mesh resolution) are portrayed in Figure~\ref{fig:ex01}.}

\begin{table}[!t]
\caption{Example 1. Error history against smooth solutions for different polynomial degrees \((k=0,1)\), and iteration count for the nonlinear Newton--Raphson solver. Here we use unit parameters, and the non-weighted norms (top rows), and the parameter-weighted norms (bottom rows) for all fields.}
\label{num:conv.rates-unity}
\setlength{\tabcolsep}{3pt}
\centering
{\footnotesize\begin{tabular}{rcccccccccc}
\hline
DoF  &    $h$ & $\|\bu-\bu_h\|_{3,\vdiv,\Omega}\!$  &  \texttt{rate} &  $\|p-p_h\|_{0,\Omega}$  &  \texttt{rate} & \(\|\bzeta-\bzeta_h\|_{\vdiv_{6/5},\Omega}\)& \texttt{rate} &\(\|\varphi-\varphi_h\|_{0,6,\Omega}\)&\texttt{rate} &  \texttt{it} \\
\hline
\multicolumn{11}{c}{\textit{Errors and convergence rates for $k = 0$}}\\
\hline
       $40$& $0.7071$& $5.07\cdot 10^{-1}$& $\star$ & $2.44\cdot 10^{-1}$& $\star$ & $3.50\cdot 10^{1}$& $\star$ & $7.09\cdot 10^{-1}$& $\star$ & $5$\\
    $160$& $0.3536$& $2.92\cdot 10^{-1}$& $0.797$& $1.29\cdot 10^{-1}$& $0.921$& $1.88\cdot 10^{1}$& $0.900$& $3.43\cdot 10^{-1}$& $1.050$& $6$\\
    $640$& $0.1768$& $1.51\cdot 10^{-1}$& $0.946$& $6.52\cdot 10^{-2}$& $0.982$& $9.68$& $0.955$& $2.02\cdot 10^{-1}$& $0.764$& $6$\\
   $2560$& $0.0884$& $7.64\cdot 10^{-2}$& $0.987$& $3.27\cdot 10^{-2}$& $0.996$& $4.88$& $0.988$& $1.03\cdot 10^{-1}$& $0.967$& $5$\\
  $10240$& $0.0442$& $3.83\cdot 10^{-2}$& $0.997$& $1.64\cdot 10^{-2}$& $0.999$& $2.45$& $0.997$& $5.19\cdot 10^{-2}$& $0.992$& $5$\\
  $40960$& $0.0221$& $1.91\cdot 10^{-2}$& $0.999$& $8.18\cdot 10^{-3}$& $1.000$& $1.22$& $0.999$& $2.60\cdot 10^{-2}$& $0.998$& $6$\\
\hline
\multicolumn{11}{c}{\textit{Errors and convergence rates for $k = 1$}}\\
\hline
     $128$& $0.7071$& $1.80\cdot 10^{-1}$&$\star$& $7.62\cdot 10^{-2}$& $\star$ & $1.92\cdot 10^{1}$&$\star$  & $3.14\cdot 10^{-1}$& $\star$ & $7$\\
    $512$& $0.3536$& $4.78\cdot 10^{-2}$& $1.914$& $1.97\cdot 10^{-2}$& $1.954$& $5.20$& $1.884$& $1.23\cdot 10^{-1}$& $1.356$& $7$\\
   $2048$& $0.1768$& $1.24\cdot 10^{-2}$& $1.943$& $4.96\cdot 10^{-3}$& $1.986$& $1.37$& $1.923$& $3.33\cdot 10^{-2}$& $1.883$& $8$\\
   $8192$& $0.0884$& $3.15\cdot 10^{-3}$& $1.979$& $1.24\cdot 10^{-3}$& $1.997$& $3.47\cdot 10^{-1}$& $1.985$& $8.54\cdot 10^{-3}$& $1.962$& $7$\\
  $32768$& $0.0442$& $7.92\cdot 10^{-4}$& $1.993$& $3.11\cdot 10^{-4}$& $1.999$& $8.68\cdot 10^{-2}$& $1.997$& $2.15\cdot 10^{-3}$& $1.991$& $8$\\
 $131072$& $0.0221$& $1.98\cdot 10^{-4}$& $1.998$& $7.78\cdot 10^{-5}$& $2.000$& $2.17\cdot 10^{-2}$& $1.999$& $5.38\cdot 10^{-4}$& $1.998$& $8$\\
\hline
DoF  &    $h$ & $\|\bu-\bu_h\|_{\bV}\!$  &  \texttt{rate} &  $\|p-p_h\|_{\widehat{\rQ}(\cT_h)}$  &  \texttt{rate} & \(\|\bzeta-\bzeta_h\|_{\bW}\)& \texttt{rate} &\(\|\varphi-\varphi_h\|_{\widehat{\Psi}(\cT_h)}\)&\texttt{rate} &  \texttt{it} \\
\hline
\multicolumn{11}{c}{\textit{Errors and convergence rates for $k = 0$}}\\
\hline
    $40$& $0.7071$& $9.83\cdot 10^{-1}$& $\star$ & $6.05\cdot 10^{-2}$& $\star$ & $3.50\cdot 10^{1}$& $\star$ & $3.66\cdot 10^{-1}$& $\star$ & $6$\\
    $160$& $0.3536$& $5.56\cdot 10^{-1}$& $0.824$&$1.94\cdot 10^{-2}$& $1.640$& $1.88\cdot 10^{1}$& $0.900$& $1.77\cdot 10^{-1}$&$1.050$& $7$\\
    $640$& $0.1768$& $2.88\cdot 10^{-1}$& $0.947$&$4.96\cdot 10^{-3}$& $1.969$& $9.68$& $0.955$& $5.14\cdot 10^{-2}$& $1.782$& $7$\\
   $2560$& $0.0884$& $1.46\cdot 10^{-1}$& $0.986$&$1.25\cdot 10^{-3}$&$1.990$& $4.88$& $0.988$&$1.33\cdot 10^{-2}$&$1.948$& $7$\\
  $10240$& $0.0442$& $7.29\cdot 10^{-2}$& $0.996$& $3.13\cdot 10^{-4}$&$1.996$& $2.45$& $0.997$& $3.36\cdot 10^{-3}$&$1.987$& $7$\\
  $40960$& $0.0221$& $3.65\cdot 10^{-2}$& $0.999$&$7.84\cdot 10^{-5}$&$1.998$& $1.22$& $0.999$&$8.42\cdot 10^{-4}$& $1.997$& $7$\\
\hline
\multicolumn{11}{c}{\textit{Errors and convergence rates for $k = 1$}}\\
\hline
      $128$& $0.7071$& $3.31\cdot 10^{-1}$&$\star$& $1.55\cdot 10^{-1}$& $\star$ & $1.92\cdot 10^{1}$&$\star$  & $3.68\cdot 10^{-1}$& $\star$ & $7$\\
    $512$& $0.3536$& $8.86\cdot 10^{-2}$& $1.902$& $4.68\cdot 10^{-2}$& $1.730$& $5.20$& $1.884$&$1.81\cdot 10^{-1}$&$1.019$& $8$\\
   $2048$& $0.1768$& $2.30\cdot 10^{-2}$& $1.946$& $1.23\cdot 10^{-2}$& $1.931$& $1.37$& $1.923$& $5.61\cdot 10^{-2}$&$1.041$& $9$\\
   $8192$& $0.0884$& $5.82\cdot 10^{-3}$& $1.981$& $3.11\cdot 10^{-3}$& $1.981$& $3.47\cdot 10^{-1}$& $1.985$& $8.82\cdot 10^{-2}$& $1.938$& $8$\\
  $32768$&$0.0442$&$1.46\cdot 10^{-3}$&$1.994$&$7.80\cdot 10^{-4}$&$1.994$&$8.68\cdot 10^{-2}$&$1.997$&$5.81\cdot 10^{-3}$&$1.984$&$8$\\
  $131072$&$0.0221$&$3.66\cdot 10^{-4}$&$1.998$&$1.95\cdot 10^{-4}$&$1.998$&$2.17\cdot 10^{-2}$&$1.999$&$1.46\cdot 10^{-3}$&$1.994$&$9$\\
\hline
\end{tabular}}

\end{table}

\begin{figure}[t!]
    \centering
\includegraphics[width=0.24\linewidth]{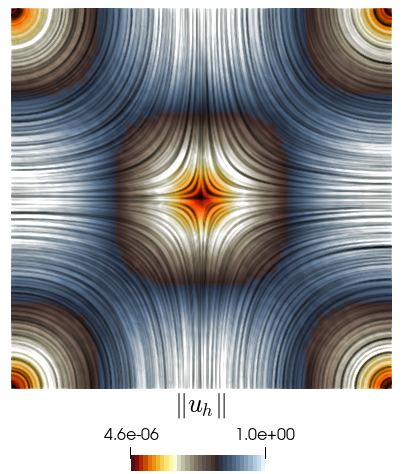}
\includegraphics[width=0.24\linewidth]{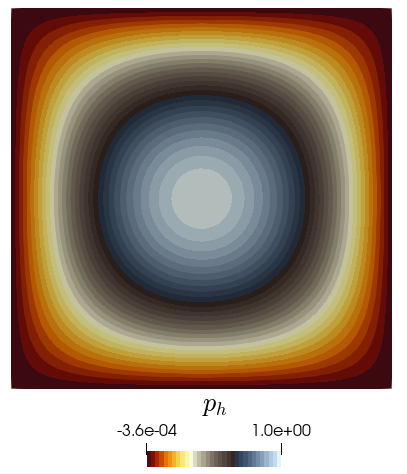}
\includegraphics[width=0.24\linewidth]{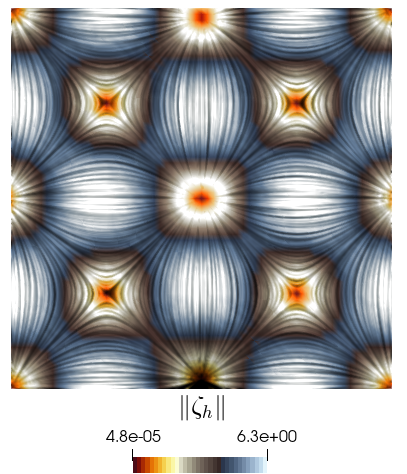}
\includegraphics[width=0.24\linewidth]{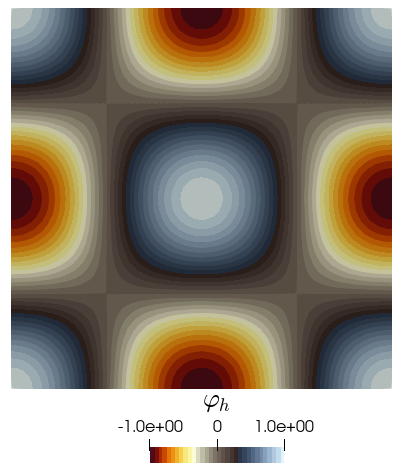}
    \caption{Example 1. Approximate velocity magnitude and line integral contours, pressure profile, flux magnitude and line integral contours, and concentration, all computed with a second-order scheme.}
    \label{fig:ex01}
\end{figure}

\smallskip
\noindent\textbf{Example 2  (convergence relative to fine-mesh reference solutions).}
Let $\Omega=(0,2)\times(0,1)$ and $(\bf,m)=(\mathbf{0},0)$. The boundary conditions are
$\bu\cdot\bn=-0.25\,y(1-y)$ and $\varphi = 2$ on the left edge, $\bu\cdot\bn=\bzeta\cdot\bn=0$ on the top and bottom edges, $p=0$  and $\varphi = 1$ on the right edge. Following a similar test in \cite{dhpr_sinum26},   we consider the following heterogeneous  permeability
\[
\kappa(\bx)=\kappa_0 \left(1+\exp\!\left(-\tfrac{1}{2}\,[5 -4x +7y + \cos x]^2\right)\right).
\]
Convergence rates are computed against reference solutions obtained via  additional mesh refinements. Errors are reported in Table~\ref{num:conv.rates}, in both the standard norms (as in Table~\ref{num:conv.rates-unity}) and the weighted norms $\mathbf{V}, \widehat{\rQ}, \mathbf{W}, \widehat{\Psi}$. While the experimental order of convergence seems affected by the extreme parameter values when using the usual norms, the optimality is restored when using the weighted norms. Close to optimal rates are observed for all variables (for the lowest-order scheme $k=0$). The Newton--Raphson iteration counts (Table~\ref{num:conv.rates}) remain low even for extreme parameter values. We show the approximate solutions on a coarse mesh in Figure \ref{fig:ex02}.

\begin{figure}[t!]
    \centering
\includegraphics[width=0.24\linewidth]{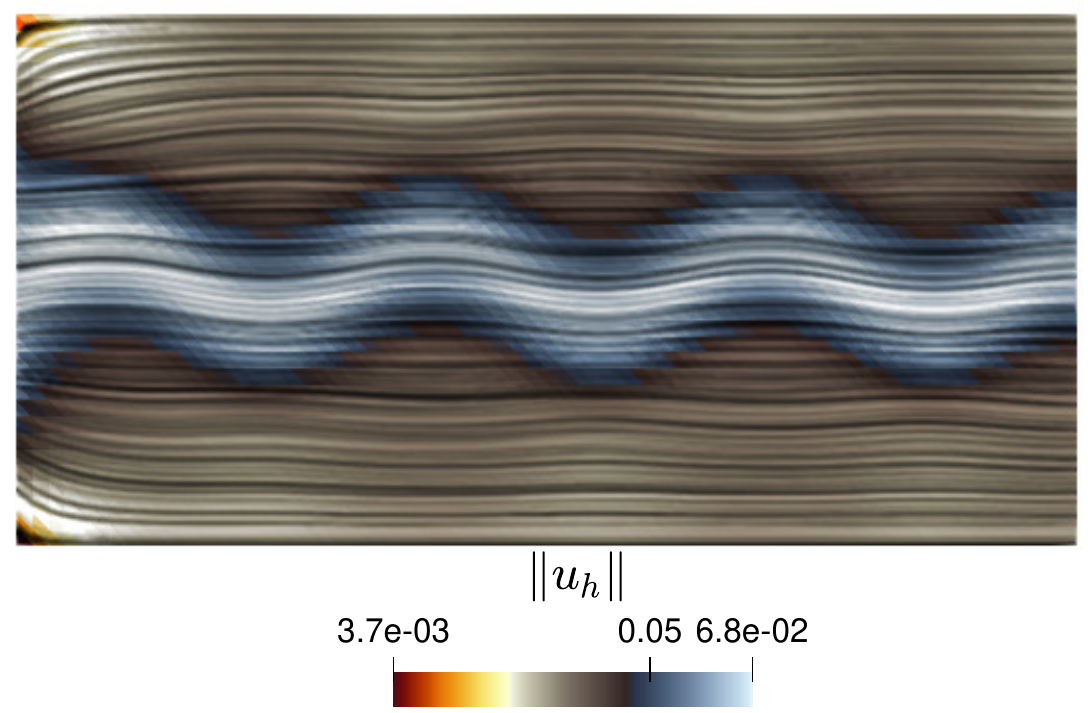}
\includegraphics[width=0.24\linewidth]{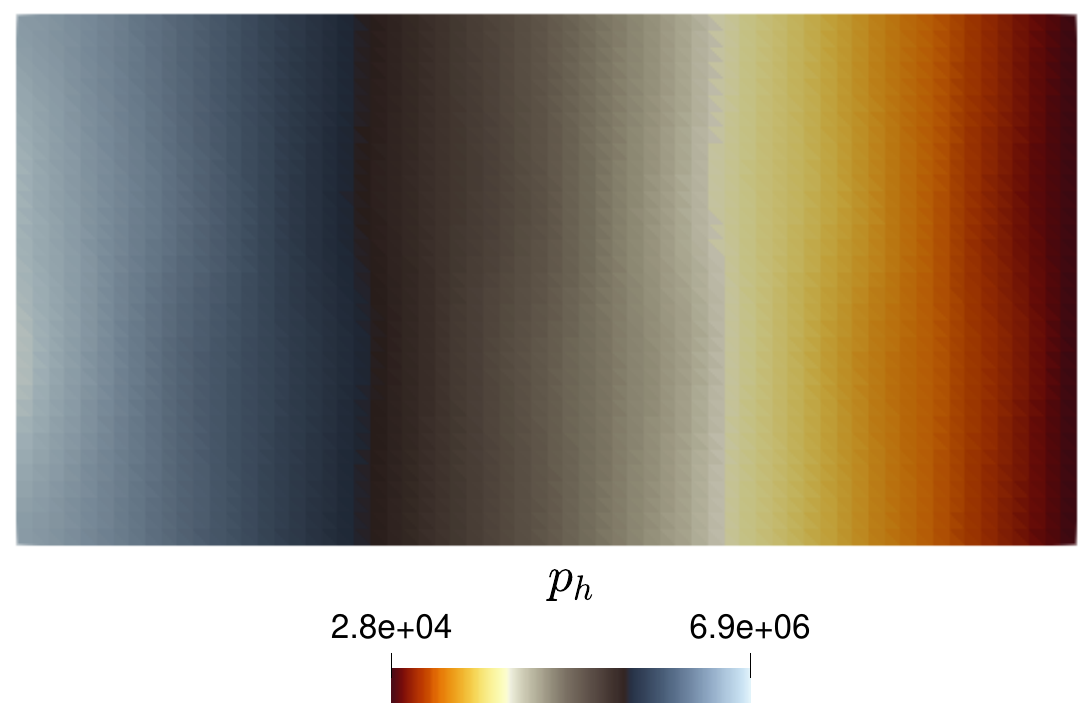}
\includegraphics[width=0.24\linewidth]{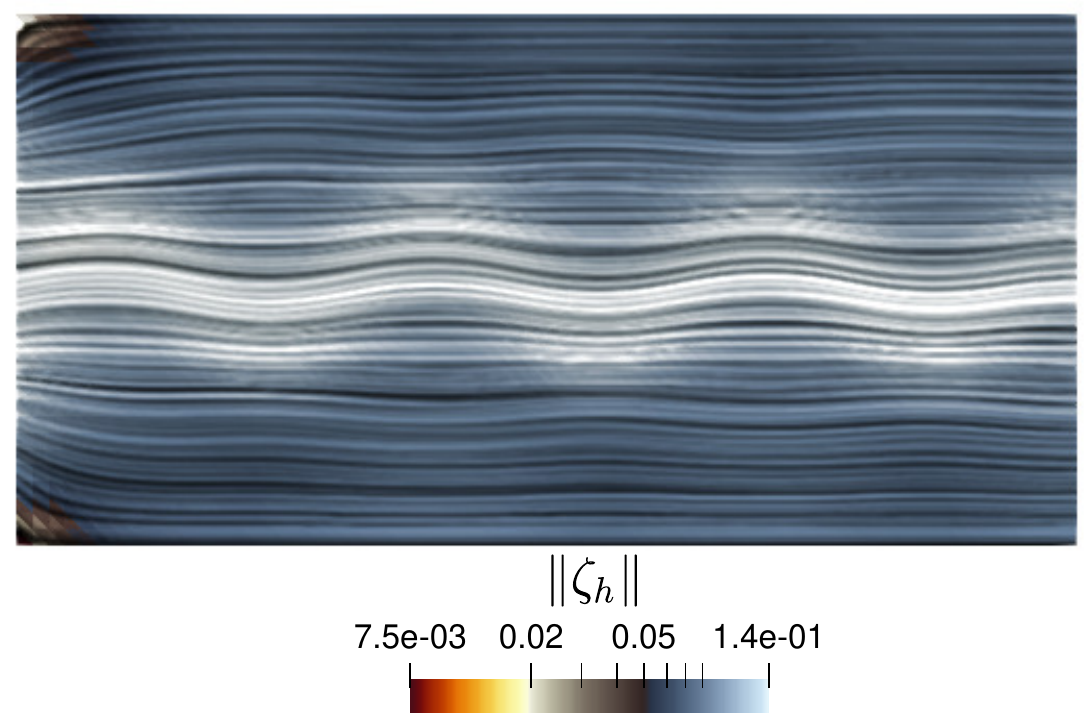}
\includegraphics[width=0.24\linewidth]{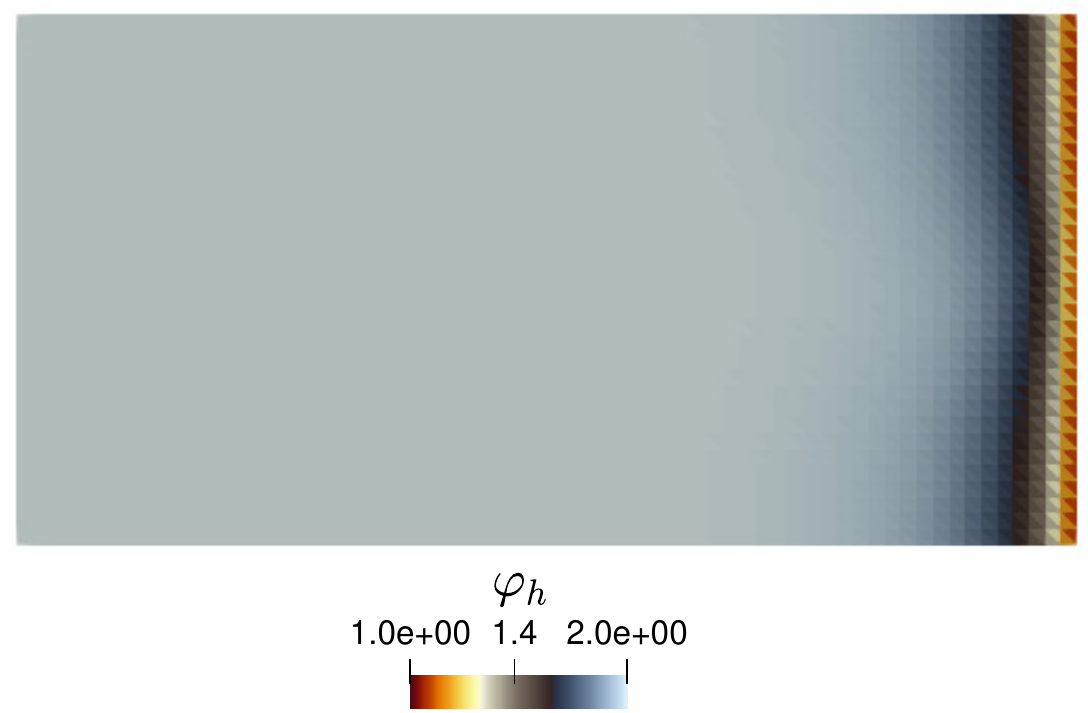}
    \caption{Example 2. Approximate velocity magnitude and line integral contours, pressure profile, flux magnitude and line integral contours, and concentration, all computed with a lowest-order scheme on a coarse mesh.}
    \label{fig:ex02}
\end{figure}

\begin{table}[!t]
\caption{Example 2. Error history relative to a fine-mesh reference solution: errors on successively refined grids, convergence rates, for $k=0$.  Newton--Raphson iteration counts corresponding to the nonlinear solver are also reported. Here we are using extreme values of the parameters \(\kappa_0=10^{-8},\rF=10^{3},\alpha=5\cdot 10^{-3},\varrho_0=10^{-5},\) and both the usual norms (top row), the weighted norms (bottom row) for velocity, pressure, flux, and concentration, respectively.}
\label{num:conv.rates}
\setlength{\tabcolsep}{2.5pt}
\centering
{\footnotesize \begin{tabular}{rcccccccccc}
\hline
DoF  &    $h$ & $\|\bu-\bu_h\|_{3,\vdiv}\!$  &  \texttt{rate} &  $\|p-p_h\|_{0,\Omega}$  &  \texttt{rate} & \(\|\bzeta-\bzeta_h\|_{\vdiv_{\frac65},\Omega}\)& \texttt{rate} &\(\|\varphi-\varphi_h\|_{0,6,\Omega}\)&\texttt{rate} &  \texttt{it} \\
\hline
  $72$& $0.7071$& $5.18\cdot 10^{-3}$& $\star$ & $1.02\cdot 10^{5}$& $\star$ & $1.22$& $\star$ & $3.71\cdot 10^{1}$&  $\star$ & $5$\\
    $168$& $0.3536$& $4.25\cdot 10^{-3}$& $0.285$& $3.86\cdot 10^{4}$& $1.403$& $1.04$& $0.222$& $2.04\cdot 10^{1}$& $0.862$&  $5$\\
    $480$& $0.1768$& $4.56\cdot 10^{-3}$& $-0.101$& $1.78\cdot 10^{4}$& $1.122$& $1.00$& $0.060$& $1.11\cdot 10^{1}$& $0.880$&  $5$\\
   $1584$& $0.0884$& $2.77\cdot 10^{-3}$& $0.718$& $6.29\cdot 10^{3}$& $1.497$& $5.17\cdot 10^{-1}$& $0.952$& $4.11$& $1.432$&  $5$\\
    $5712$& $0.0442$& $1.53\cdot 10^{-3}$& $0.857$& $1.93\cdot 10^{3}$& $1.707$& $2.83\cdot 10^{-1}$& $0.868$& $2.70$& $0.605$&  $5$\\
$21648$& $0.0221$& $1.19\cdot 10^{-3}$& $0.383$& $1.53\cdot 10^{3}$& $0.465$& $2.25\cdot 10^{-1}$& $0.182$& $2.14$& $0.278$& $5$\\
\hline

DoF  &    $h$ & $\|\bu-\bu_h\|_{\bV}\!$  &  \texttt{rate} &  $\|p-p_h\|_{\widehat{\rQ}_h}$  &  \texttt{rate} & \(\|\bzeta-\bzeta_h\|_{\bW}\)& \texttt{rate} &\(\|\varphi-\varphi_h\|_{\widehat{\Psi}_h}\)&\texttt{rate} &  \texttt{it} \\
\hline
     $72$& $0.7071$& $1.44\cdot 10^{1}$& $\star$ & $7.54$& $\star$ & $9.90$& $\star$ & $6.77$& $\star$ & $5$\\
    $168$& $0.3536$& $6.95$& $1.075$& $3.10$& $1.281$& $6.71$& $0.562$& $4.53$& $0.579$& $5$\\
    $480$& $0.1768$& $3.65$& $0.905$& $1.57$& $0.980$& $4.21$& $0.673$& $2.70$& $0.746$& $5$\\
   $1584$& $0.0884$& $1.64$& $1.153$& $5.40\cdot 10^{-1}$& $1.542$& $1.94$& $1.119$& $1.15$& $1.229$& $5$\\
   $5712$& $0.0442$& $5.26\cdot 10^{-1}$& $1.642$& $2.93\cdot 10^{-1}$& $0.883$& $7.63\cdot 10^{-1}$& $1.343$& $6.81\cdot 10^{-1}$& $0.961$& $5$\\
  $21648$& $0.0221$& $2.18\cdot 10^{-1}$& $1.297$& $1.17\cdot 10^{-1}$& $1.630$& $3.17\cdot 10^{-1}$& $1.405$& $3.50\cdot 10^{-1}$& $0.985$& $5$\\
\hline
\end{tabular}}

\end{table}

\smallskip
\noindent\textbf{Example~3 (operator preconditioning).} In this example, we test the robustness of the preconditioners~{$\cB^{i,j}_{\mathsf{DFADR}},i,j=1,2$} for the fully-coupled linearised Darcy--Forchheimer and ADR system.  We stick to the domain \(\Omega=(0,1)^2,\) the lowest order
finite element family, and mixed boundary conditions. We perform a sequence of experiments in which the permeability of the material \(\kappa\), the Forchheimer coefficient   \(\rF\), the diffusivity coefficient \(\alpha,\)  the reaction coefficient \(\varrho_0,\) and the mesh size $h$ vary by several orders of magnitude.

Note that in this context $-\Delta$ and  $ \bu^m\cdot\nabla$
are  the discrete diffusion and convection operators, and they both  require stabilisation. {To ensure robust invertibility, the classical consistency (average-times-jump) terms found in standard Symmetric Interior Penalty Galerkin (SIPG) methods \cite{arnold2002unified} are  omitted. This isolates a purely positive-definite interface penalty operator parametrised by the stabilisation penalty factor $\beta = 10$ and the local characteristic mesh sizes $h_e$
\begin{align*}
    \langle-\alpha\Delta\varphi_h,\psi_h\rangle=& \alpha\beta \sum_{K\in \cT_h} \int_K  \nabla \varphi_h \cdot \nabla \psi_h + \sum_{F\in \mathcal{F}_h}{\int_F} \frac{\alpha\beta}{h_F} [\![\varphi_h]\!]\,[\![\psi_h]\!].
\end{align*}
Hyperbolic stability for the advection term is governed by a pointwise upwind numerical flux along the interior interfaces   and boundaries, utilizing the velocity field $\bu^m$. By defining the positive and negative part operators as $[x]_+ = \max(x, 0)$ and $[x]_- = \min(x, 0)$, respectively, the corresponding block is
\begin{align*}
 \langle(\bu^m\cdot\nabla)\varphi_h,\psi_h \rangle  =&
-\sum_{K\in \cT_h} \int_K (\bu_h^m\cdot \nabla \psi_h)\,\varphi_h\\
&\quad + \sum_{F\in \mathcal{F}_h} \int_F
\bigl([(\bu^m)^+ \cdot \bn_F^+]_+ \varphi^+ + [(\bu^m)^+ \cdot \bn_F^+]_- \varphi^- \bigr) \llbracket \psi \rrbracket   \\
&\quad  + \sum_{F\in \Gamma_{\text{out}}\cup \Gamma_{\text{in}}}  \int_F [\bu^m \cdot \bn_F]_+ \varphi \psi,
\end{align*}
where  $\llbracket \phi \rrbracket = \phi^+ - \phi^-$ denotes the standard jump operator across the facet $F$ associated with the outward unit normal $\bn_F^+$. This   retains the essential upwind stability and diffusion scales while keeping a diagonally dominant structure   \cite{franca1992stabilized,kanschat2008multigrid}.
Even with this, we note}  that the tested preconditioners are not robust in the advection-dominated regime (large mesh Péclet number $\mathrm{Pe}_h = \frac{|\hat{\mathbf{u}}|h}{\alpha} \gg 1$). Therefore, we restrict our attention to the diffusion-dominated regime, where $\mathrm{Pe}_h \ll 1$, in which the preconditioners exhibit the desired robustness, see Table~\ref{tab:DF-ARD-Prec}.  The  performance of the preconditioners is assessed through condition numbers of the preconditioned system and the number of MINRES iterations needed to converge, all against  mesh refinement.

\begin{table}[t]
\caption{{Example~3. Condition numbers (and MINRES iteration counts) corresponding to the preconditioners in \eqref{precond.linear.B2}:~$\mathcal{B}^{1,1}_{\mathsf{DFADR}}$ (top left), $\mathcal{B}^{1,2}_{\mathsf{DFADR}}$ (top right), $\mathcal{B}^{2,1}_{\mathsf{DFADR}}$ (bottom left), and $\mathcal{B}^{2,2}_{\mathsf{DFADR}}$ (bottom right), under different parameter regimes of the coupled Darcy--Forchheimer and ADR model.}}
\label{tab:DF-ARD-Prec}
\centering
\footnotesize
\setlength{\tabcolsep}{3.5pt}
\renewcommand{\arraystretch}{1.15}

{
\begin{minipage}{0.48\textwidth}
\centering
\begin{tabular}{ccccccc}
\hline
$\varrho_0$ & $\kappa$ & $\alpha$ & $\rF$ & \multicolumn{3}{c}{$h$} \\
\cline{5-7}
 & & & & $2^{-2}$ & $2^{-3}$ & $2^{-4}$ \\
\hline

\multirow{3}{*}{$1$}
& $10^{-8}$ & $10^{-9}$ & $10^{6}$ & $1.8\;(6)$ & $1.8\;(6)$& $1.8\;(7)$ \\
& $10^{-4}$ & $10^{-5}$ & $1$ & $1.8\;(10)$& $1.8\;(12)$ & $1.8\;(14)$\\
& $1$ & $10^{-1}$ & $10^{-6}$ &$2.9\;(27)$ &$3.0\;(28)$ &$3.0\;(28)$ \\
\hline

\multirow{3}{*}{$10^{-5}$}
& $10^{-8}$ & $10^{-9}$ & $10^{6}$ & $1.2\;(8)$ &$1.3\;(10)$ &$1.6\;(14)$ \\
& $10^{-4}$ & $10^{-5}$ & $1$ &$3.6\;(31)$ &$4.0\;(35)$ &$4.2\;(35)$ \\
& $1$ & $10^{-1}$ & $10^{-6}$ &$2.1\;(22)$ &$2.1\;(22)$ &$2.1\;(22)$ \\
\hline

\multirow{3}{*}{$10^{-9}$}
& $10^{-8}$ & $10^{-9}$ & $10^{6}$ & $3.6\;(24)$&$4.1\;(26)$ &$4.3\;(28)$ \\
& $10^{-4}$ & $10^{-5}$ & $1$ &$3.7\;(32)$ &$4.1\;(35)$ &$4.2\;(35)$ \\
& $1$ & $10^{-1}$ & $10^{-6}$ &$2.1\;(22)$ &$2.1\;(22)$ &$2.1\;(22)$ \\
\hline
\end{tabular}
\end{minipage}}
\hfill
{\begin{minipage}{0.48\textwidth}
\centering
\begin{tabular}{ccccccc}
\hline
$\varrho_0$ & $\kappa$ & $\alpha$ & $\rF$ & \multicolumn{3}{c}{$h$} \\
\cline{5-7}
 & & & & $2^{-2}$ & $2^{-3}$ & $2^{-4}$ \\
\hline

\multirow{3}{*}{$1$}
& $10^{-8}$ & $10^{-9}$ & $10^{6}$ & $1.8\;(6)$ & $1.8\;(6)$& $1.8\;(7)$ \\
& $10^{-4}$ & $10^{-5}$ & $1$ & $1.8\;(10)$& $1.8\;(12)$ & $1.8\;(14)$\\
& $1$ & $10^{-1}$ & $10^{-6}$ &$2.9\;(28)$ &$3.0\;(29)$ &$3.0\;(29)$ \\
\hline

\multirow{3}{*}{$10^{-5}$}
& $10^{-8}$ & $10^{-9}$ & $10^{6}$ & $1.2\;(8)$ &$1.3\;(10)$ &$1.6\;(14)$ \\
& $10^{-4}$ & $10^{-5}$ & $1$ &$3.6\;(32)$ &$4.0\;(36)$ &$4.2\;(36)$ \\
& $1$ & $10^{-1}$ & $10^{-6}$ &$2.1\;(23)$ &$2.0\;(23)$ &$2.0\;(23)$ \\
\hline

\multirow{3}{*}{$10^{-9}$}
& $10^{-8}$ & $10^{-9}$ & $10^{6}$ & $3.6\;(25)$&$4.1\;(27)$ &$4.3\;(28)$ \\
& $10^{-4}$ & $10^{-5}$ & $1$ &$3.7\;(32)$ &$4.1\;(35)$ &$4.2\;(37)$ \\
& $1$ & $10^{-1}$ & $10^{-6}$ &$2.1\;(23)$ &$2.0\;(23)$ &$2.0\;(23)$ \\
\hline
\end{tabular}
\end{minipage}}

\vspace{0.4cm}

{\begin{minipage}{0.48\textwidth}
\centering
\begin{tabular}{ccccccc}
\hline
$\varrho_0$ & $\kappa$ & $\alpha$ & $\rF$ & \multicolumn{3}{c}{$h$} \\
\cline{5-7}
 & & & & $2^{-2}$ & $2^{-3}$ & $2^{-4}$ \\
\hline

\multirow{3}{*}{$1$}
& $10^{-8}$ & $10^{-9}$ & $10^{6}$ & $3.7\;(21)$ & $4.1\;(25)$& $4.4\;(25)$ \\
& $10^{-4}$ & $10^{-5}$ & $1$ & $3.6\;(30)$& $3.8\;(32)$ & $3.5\;(35)$\\
& $1$ & $10^{-1}$ & $10^{-6}$ &$2.9\;(29)$ &$3.0\;(29)$ &$3.0\;(29)$ \\
\hline

\multirow{3}{*}{$10^{-5}$}
& $10^{-8}$ & $10^{-9}$ & $10^{6}$ & $3.7\;(25)$ & $4.1\;(29)$ & $4.4\;(32)$ \\
& $10^{-4}$ & $10^{-5}$ & $1$ &$3.7\;(34)$ &$4.1\;(37)$ &$4.2\;(37)$ \\
& $1$ & $10^{-1}$ & $10^{-6}$ &$1.9\;(21)$ &$1.9\;(21)$ &$1.9\;(20)$ \\
\hline

\multirow{3}{*}{$10^{-9}$}
& $10^{-8}$ & $10^{-9}$ & $10^{6}$ & $3.7\;(25)$&$4.1\;(27)$ &$4.4\;(29)$ \\
& $10^{-4}$ & $10^{-5}$ & $1$ &$3.7\;(33)$ &$4.1\;(37)$ &$4.2\;(37)$ \\
& $1$ & $10^{-1}$ & $10^{-6}$ &$1.9\;(21)$ &$1.9\;(21)$ &$1.9\;(20)$ \\
\hline
\end{tabular}
\end{minipage}}
\hfill
{\begin{minipage}{0.48\textwidth}
\centering
\begin{tabular}{ccccccc}
\hline
$\varrho_0$ & $\kappa$ & $\alpha$ & $\rF$ & \multicolumn{3}{c}{$h$} \\
\cline{5-7}
 & & & & $2^{-2}$ & $2^{-3}$ & $2^{-4}$ \\
\hline

\multirow{3}{*}{$1$}
& $10^{-8}$ & $10^{-9}$ & $10^{6}$ & $3.7\;(21)$ & $4.1\;(25)$& $4.4\;(25)$ \\
& $10^{-4}$ & $10^{-5}$ & $1$ & $3.6\;(30)$& $3.8\;(32)$ & $3.5\;(35)$\\
& $1$ & $10^{-1}$ & $10^{-6}$ &$2.9\;(29)$ &$3.0\;(30)$ &$3.0\;(31)$ \\
\hline

\multirow{3}{*}{$10^{-5}$}
& $10^{-8}$ & $10^{-9}$ & $10^{6}$ & $3.7\;(25)$ &$4.1\;(29)$ &$4.4\;(32)$ \\
& $10^{-4}$ & $10^{-5}$ & $1$ &$3.6\;(34)$ &$4.1\;(37)$ &$4.2\;(37)$ \\
& $1$ & $10^{-1}$ & $10^{-6}$ &$1.8\;(21)$ &$1.8\;(21)$ & $1.8\;(22)$ \\
\hline

\multirow{3}{*}{$10^{-9}$}
& $10^{-8}$ & $10^{-9}$ & $10^{6}$ & $3.7\;(25)$&$4.1\;(27)$ &$4.4\;(29)$ \\
& $10^{-4}$ & $10^{-5}$ & $1$ &$3.7\;(34)$ &$4.1\;(37)$ &$4.2\;(37)$ \\
& $1$ & $10^{-1}$ & $10^{-6}$ &$1.8\;(21)$ &$1.8\;(21)$ &$1.8\;(22)$ \\
\hline
\end{tabular}
\end{minipage}}
\vspace{0.5cm}

\end{table}

\smallskip
\noindent\textbf{Example 4 (lid-driven cavity flow).} {Next, we present a numerical simulation following the classical lid-driven cavity problem for the Darcy--Forchheimer-transport coupling \cite{sayah2021finite}. The domain is $\Omega = (0,1)^2$ (discretised with a triangular mesh representing 82176 DoFs for the lowest-order method), the model parameters are taken as follows $\kappa = 1$, $\rF = 20$, $\alpha = 1$, $\bf(\varphi) = (10\varphi,10\varphi)^{\tt t}$, $m = 0$, $\varrho_0 = 0$, and as boundary conditions we impose $\bu \cdot \bn = 0$ everywhere on $\partial\Omega$ and so {for the sake of pressure uniqueness,  we impose $\int_\Omega p = 0$ (with a real Lagrange multiplier). We also} set $\varphi = 20$ on the top edge and $\varphi = 0$ on the remainder of the boundary. In the first two rows of Figure~\ref{fig:placeholder} we present the solution profiles, which match the results of \cite{sayah2021finite}. We also carried out an extension to 3D of the same test, taking $\Omega = (0,1)^3$ (discretised with a mesh representing 2323200 DoFs), $\bf(\varphi) = (10\varphi,10\varphi,10\varphi)^{\tt t}$, and maintained all other parameters unchanged. The obtained numerical solutions are shown in the bottom row of the same figure. For both cases we have used the lowest-order method, and the Newton--Raphson iterative scheme (using Hager--Zhang line search) took seven  iterations to reach the prescribed tolerance of $10^{-6}$.
At least in this flow-transport regime and with the chosen parametric values, the discrepancies between 3D and 2D simulations are unnoticeable.}

\begin{figure}[t!]
    \centering
\includegraphics[width=0.24\linewidth]{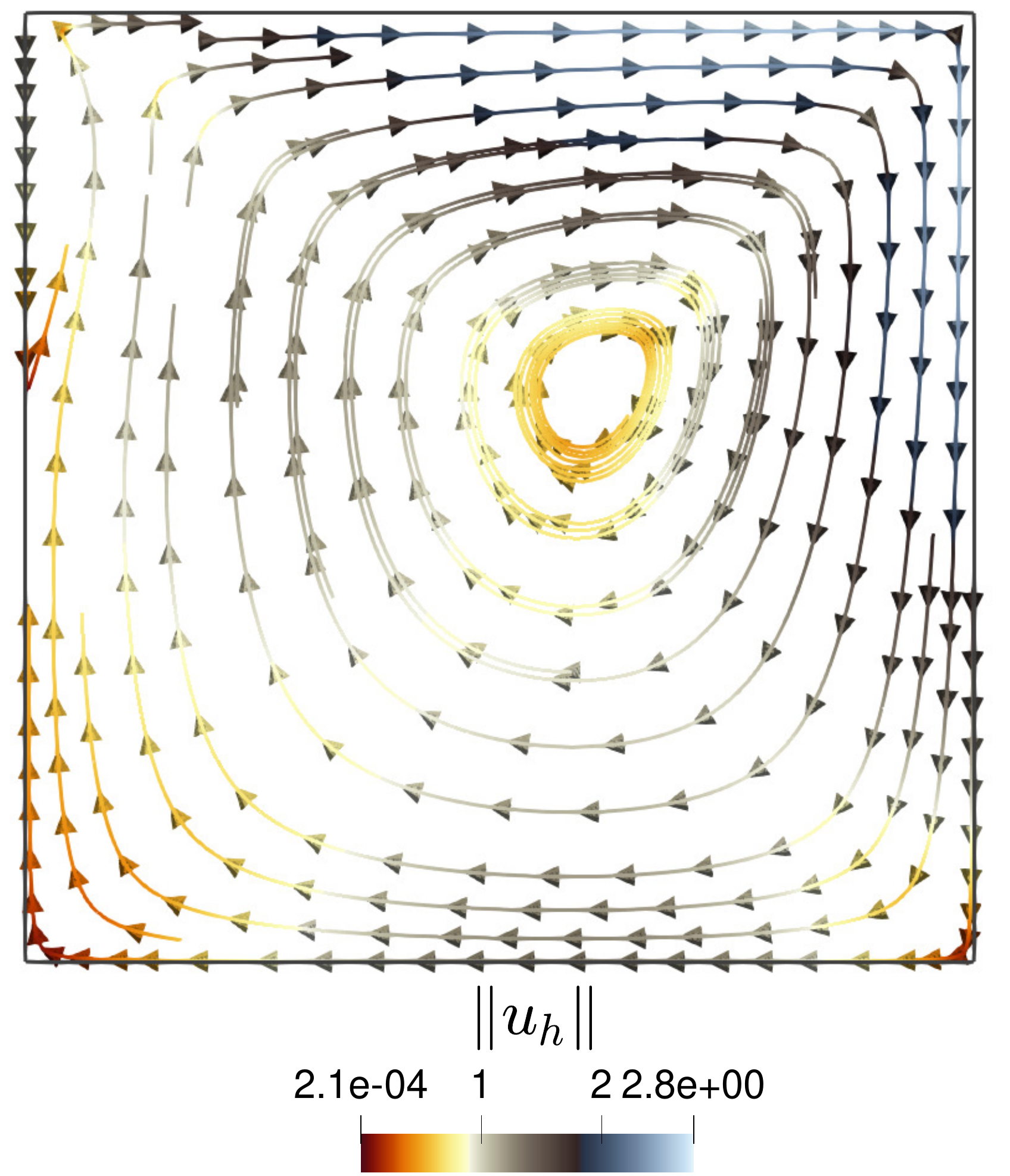}    \includegraphics[width=0.24\linewidth]{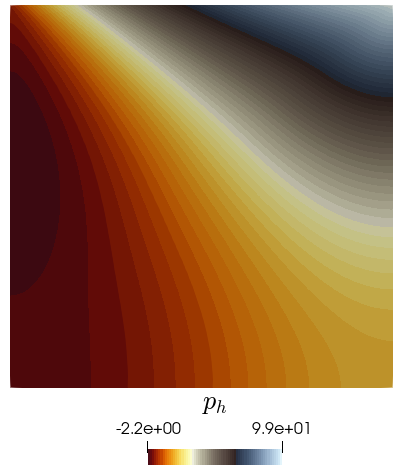}
\includegraphics[width=0.24\linewidth]{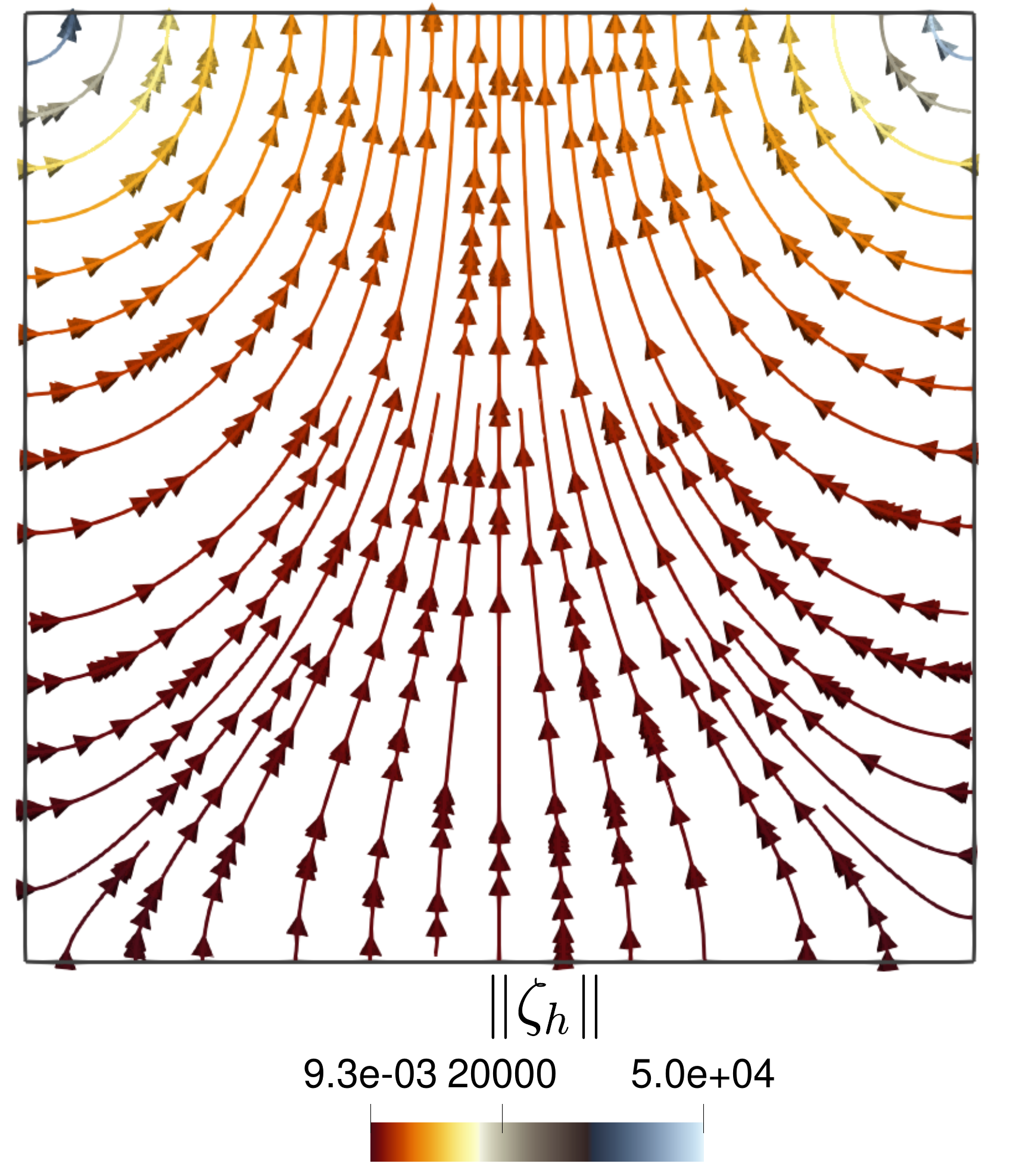}
\includegraphics[width=0.24\linewidth]{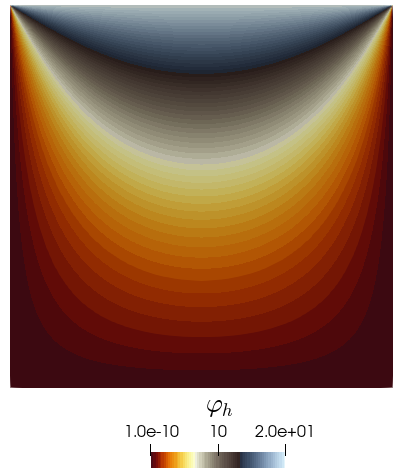}\\
\includegraphics[width=0.24\linewidth]{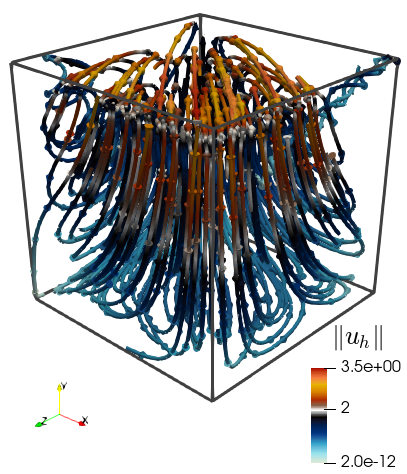}
\includegraphics[width=0.24\linewidth]{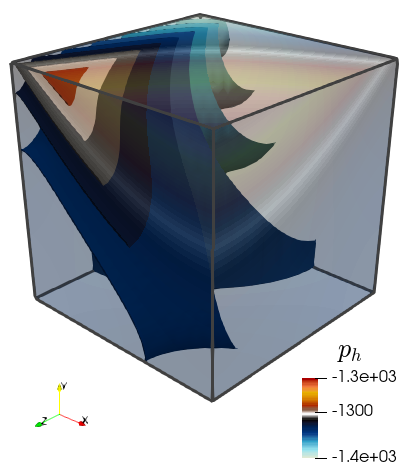}
\includegraphics[width=0.24\linewidth]{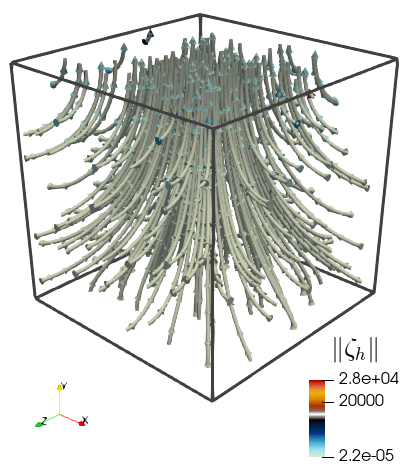}
\includegraphics[width=0.24\linewidth]{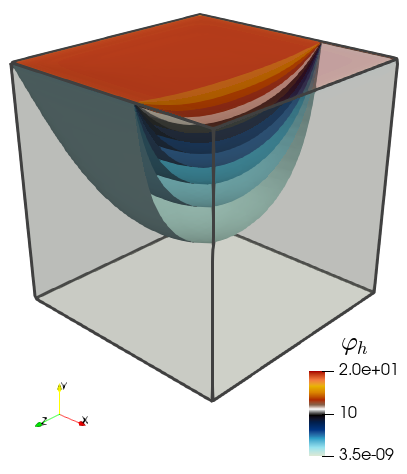}
    \caption{Example 4. Numerical solutions (discharge velocity streamlines, pressure profile, total flux, and concentration) for the lid-driven cavity problem. Parameters are taken as in \cite{sayah2021finite} for the 2D case (top and middle rows), and we also show its extension to 3D (bottom row).}
    \label{fig:placeholder}
\end{figure}

\smallskip
\noindent\textbf{Example 5 (transient miscible displacement).} To conclude, we apply a time-dependent version of \eqref{eq:strong} (replacing the reaction term $\varrho_0\varphi$ by $\varrho_0\partial_t \varphi$ in the concentration equation and the diffusion $-\alpha \Delta \varphi$ by $-\varrho_0\alpha \Delta \varphi$, where $\varrho_0$ is the porosity; the advective term $\bu\cdot\nabla\varphi$, which enters through the total flux $\bzeta=\alpha\nabla\varphi-\bu\varphi$, is not modified, since $\bu$ is the Darcy (superficial) velocity, which already accounts for the porosity, as in the classical miscible displacement model \cite{bear2013dynamics})
to the classical problem of viscous fingering in a Hele--Shaw cell. Here \(\varphi\) is interpreted as the solute concentration (in mol/m\(^3\)) of a resident fluid being displaced by a second, less viscous fluid (water).

\begin{figure}[t!]
    \centering
    \includegraphics[width=0.245\linewidth]{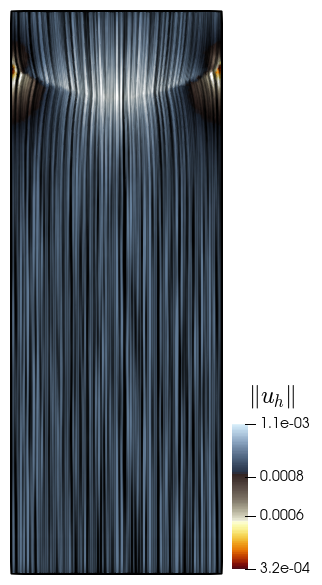}
    \includegraphics[width=0.245\linewidth]{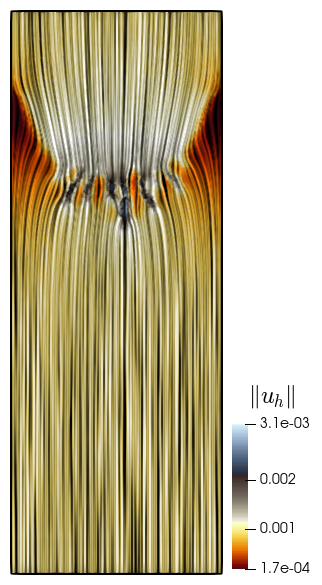}
   \includegraphics[width=0.245\linewidth]{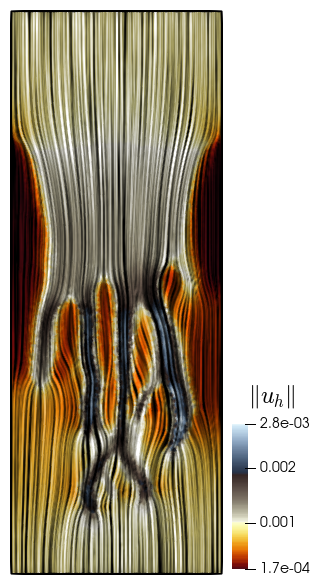}\\
   \includegraphics[width=0.245\linewidth]{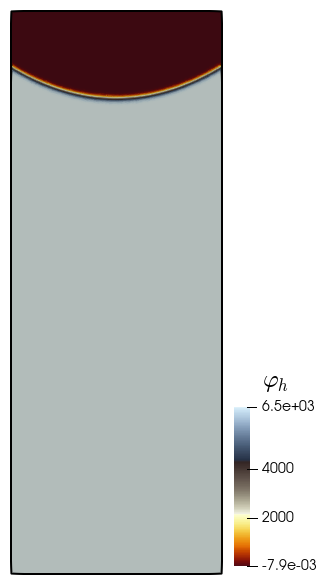}
    \includegraphics[width=0.245\linewidth]{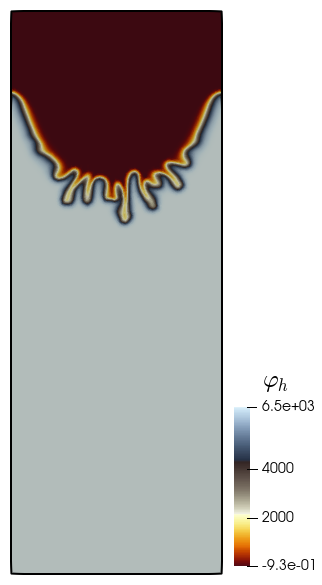}
   \includegraphics[width=0.245\linewidth]{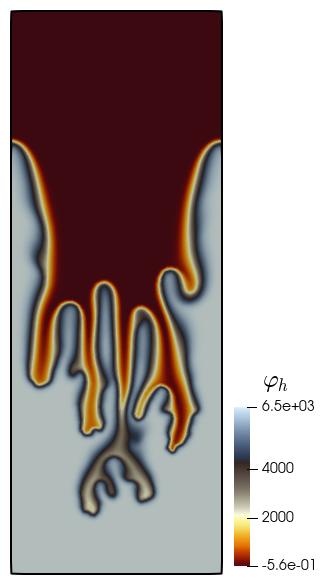}
    \caption{Example 5. Snapshots of the approximate discharge flux (top) and concentration profile (bottom) at times $t=0.01$ (left), $t=6$ (centre), and $t=18$ (right).}
    \label{fig:ex05}
\end{figure}

The viscosity of the fluid mixture follows an Arrhenius law
$\mu(\varphi) = \mu_1 \exp(R\,\frac{\varphi}{\varphi_2})$,
where \(\mu_1\) is the viscosity of the displacing fluid (water), \(\varphi_2 = 6500\) mol/m\(^3\) is a reference concentration, and \(R = \ln r\) is the log‑mobility ratio.
The viscosity ratio is \(r = \mu(\varphi=1)/\mu_1 = 2\), hence \(R = \ln 2\).
The hydraulic conductivity is then \(\kappa = K/\mu(\varphi)\), with \(K\) the intrinsic permeability; the  body force \(\bf(\varphi)\) is set to zero. The computational domain is the channel \(\Omega = (0,0.03)\times(0,0.08)\) m\(^2\).
The initial concentration field $\varphi(x,y,0)$ is a step function from 0 (water) on top of a parabolic profile located near the top boundary $y=0.0675+25(x-0.015)^2$, to $\varphi_2$ (resident fluid) below it.
As water is injected from the top  boundary (setting a prescribed normal discharge inflow velocity $\bu\cdot\bn = -9\cdot10^{-4}$\,m/s and fixed concentration $\varphi = 0$), the resident fluid is displaced and the viscous fingering instability develops spontaneously.
No artificial random perturbation is added; the unstructured mesh provides sufficient asymmetry to trigger the instability near the initial interface. On the vertical walls we impose \(\bu\cdot\bn = 0\) and \(\bzeta\cdot\bn = 0\). On the outlet (bottom segment) we set $p=0$ and zero diffusive flux (not total flux), this means that we need to include the remainder terms $-\alpha \langle \varphi, \bw\cdot \bn \rangle_{\mathrm{out}}$ and $\langle (\bzeta+\varphi\bu)\cdot \bn, \bw\cdot\bn\rangle_{\mathrm{out}}$ in equation \eqref{eq:weak,3}.

The model parameters (diffusivity \(\alpha = 4\cdot10^{-8}\)\,m$^2/$s, porosity \(\varrho_0=0.5\), permeability \(K=10^{-6}\)\,m$^2$, reference concentration \(\varphi_2=6500\)\,mol/m$^3$,  the viscosity ratio \(r=2\) (i.e., $R=\ln2$), and water viscosity $\mu_1=1$\,mPa$\cdot$s) are taken from \cite{bear2013dynamics,gatica2022p}; the source term \(m\) in \eqref{eq:strong} is zero.
The domain is discretised with an unstructured mesh of 37745 triangles.
The time derivative is discretised using backward Euler's method with constant time step \(\Delta t = 0.01\)\,s, and the simulation runs until the final time \(t = 18\)\,s.
The numerical results (snapshots of the approximate solutions at three different times, showing in this case only velocity and concentration)  are displayed in Figure~\ref{fig:ex05}. They reproduce the characteristic viscous fingering patterns observed in Hele--Shaw displacements. For this test the maximum number of nonlinear iterations  to convergence was four (and required at the initial time steps).

\section{Concluding remarks}\label{sec:concl}
We have extended the Brezzi--Braess theory for parameter-robust stability of perturbed saddle-point problems to reflexive Banach spaces, and applied it to the mixed formulation of the advection-diffusion-reaction equation coupled with the Darcy--Forchheimer equations. The resulting stability bounds and error estimates are robust with respect to the permeability, the Forchheimer coefficient and the reaction coefficient. The cubic monotonicity of the Forchheimer term leads to H\"older-type (rather than Lipschitz) continuity of the flow solution operator, and to error estimates that are suboptimal in the worst case, although optimal convergence is observed in all our numerical tests. Robustness with respect to the diffusivity, both in the analysis and in the preconditioners, is only partially achieved, mainly because of the non-symmetric advective term.

A natural strategy to address this last point consists in exploiting the incompressibility constraint, $\vdiv(\bu\varphi)=\bu\cdot\nabla\varphi$, to split the convective term in a skew-symmetric manner. This requires introducing the concentration gradient $\bxi:=\nabla\varphi$ as an additional unknown, together with a ``partial'' flux $\bsigma:=\alpha\bxi-\frac12\bu\varphi$ containing only half of the advective contribution (for Dirichlet boundary conditions on the concentration). Augmented formulations of this type, with the gradient of the transported quantity as an extra unknown, have already been analysed in the $\rL^p$-based mixed setting, for instance for the Poisson--Nernst--Planck part of the coupled Stokes/PNP system in \cite{correa2023new}. Combining such a formulation with the robust framework developed here is the subject of ongoing work.

\appendix
\section{Preconditioners for sub-systems}
\smallskip
\noindent\textbf{Robust preconditioners for the Darcy--Forchheimer system.}
Let us consider the Darcy--Forchheimer system alone from \eqref{eq:strong}
\begin{equation}\label{eq:strong.Darcy.Forchh}
\begin{aligned}
    \kappa^{-1}\bu + {\rF|\bu|\bu} + \nabla p & = \bf \quad && \text{in }\Omega,\\
    \vdiv \bu & = 0 \quad && \text{in }\Omega,\\
    \bu \cdot \bn & = 0 \quad && \text{on }\Gamma_{\bu},\\
    p & = 0 \quad && \text{on }\Gamma_p.
\end{aligned}
\end{equation}
The following two \((\kappa, \rF, h)\)-robust preconditioners for \eqref{eq:strong.Darcy.Forchh} have been mentioned in \cite[Section~5]{dhpr_sinum26}, and the robustness of these preconditioners can be seen in Table~\ref{tab:DF.prec}.
\begin{align}\label{precond.linear}
\mathcal{B}_{\mathsf{DF}}^1:=\begin{pmatrix}
        (\kappa^{-1}\mathcal{I}+\rF
        {\mathcal{H}({\bu})} {\mathcal{I}}-\nabla\vdiv)^{-1} & 0\\0 &\mathcal{I}^{-1}+ {\bigl(-\kappa\Delta\bigr)^{-1}-\bigl([\rF{\mathcal{H}({\bu})}]^{-1}{\Delta_{{\frac32}}}\bigr)^{-1}}
    \end{pmatrix},\end{align}
    and \begin{align}\label{precond.linear.new}
    \mathcal{B}_{\mathsf{DF}}^2:=\begin{pmatrix}
        \bigl([\kappa^{-1}\mathcal{I}+\rF
        {\mathcal{H}({\bu})}]\, ({\mathcal{I}}-\nabla\vdiv)\bigr)^{-1} & 0\\0 &([\kappa^{-1}\mathcal{I}+\rF
        {\mathcal{H}({\bu})}]^{-1}\,\mathcal{I})^{-1}
    \end{pmatrix},\end{align}
    where \(\mathcal{H}(\bu)\) is the G\^ateaux derivative of the nonlinear operator \(|\bu|\bu\) in the state \(\bu\) in the direction of \(\delta\bu\), and ${\Delta_{{\frac32}}}$ is the $\frac32$-Laplace operator associated with the space {$\rW^{1,\frac32}_{\Gamma_p}(\Omega)$}, defined as
  \[{\langle \Delta_{\frac32}\phi,\psi\rangle := \int_\Omega |\nabla\phi|^{-\frac12}\nabla \phi \cdot \nabla \psi, \qquad \forall \psi \in \rW^{1,3}(\Omega).} \]

  \begin{table}[t!]
\caption{Condition numbers of the preconditioned systems (\eqref{precond.linear} on left, and \eqref{precond.linear.new} on right) with MINRES iterations in parentheses for different values of $\kappa, \rF,$ and $h$.}
\label{tab:DF.prec}
\centering
{\footnotesize \begin{minipage}{0.48\textwidth}
\centering
\begin{tabular}{ccccc}
\hline
$\kappa$ & $\rF$ & \multicolumn{3}{c}{$h$} \\
\cline{3-5}
 & & $2^{-2}$ & $2^{-3}$ & $2^{-4}$ \\
\hline

\multirow{3}{*}{$10^{-10}$}
& $10^{-10}$ & $3.7\;(23)$ & $4.1\;(31)$ & $4.4\;(33)$ \\
& $1$& $5.7\;(23)$ & $4.1\;(31)$ & $4.4\;(33)$ \\
& $10^{5}$ & $3.7\;(23)$ & $4.1\;(31)$ & $4.4\;(33)$ \\

\hline

\multirow{3}{*}{$10^{-5}$}
& $10^{-10}$ & $3.7\;(25)$ & $4.1\;(31)$ & $4.2\;(33)$ \\
& $1$& $3.7\;(25)$ & $4.1\;(31)$ & $4.2\;(33)$ \\
& $10^{5}$ & $3.7\;(27)$ & $4.1\;(31)$ & $4.3\;(33)$ \\

\hline

\multirow{3}{*}{$1$}
& $10^{-10}$ & $1.2\;(11)$ & $1.2\;(11)$ & $1.2\;(11)$ \\
& $1$& $1.2\;(11)$ & $1.2\;(11)$ & $1.2\;(11)$ \\
& $10^{5}$ & $4.4\;(27)$ & $4.6\;(33)$ & $4.5\;(38)$ \\
\hline
\end{tabular}
\end{minipage}
\hfill
\begin{minipage}{0.48\textwidth}
\centering
\begin{tabular}{ccccc}
\hline
$\kappa$ & $\rF$ & \multicolumn{3}{c}{$h$} \\
\cline{3-5}
 & & $2^{-2}$ & $2^{-3}$ & $2^{-4}$ \\
\hline

\multirow{3}{*}{$10^{-10}$}
& $10^{-10}$ & $1.2\;(1)$ & $1.2\;(1)$ & $1.2\;(1)$ \\
& $1$        & $1.2\;(1)$ & $1.2\;(1)$ & $1.2\;(1)$ \\
& $10^{5}$   & $1.2\;(1)$ & $1.2\;(1)$ & $1.2\;(1)$ \\

\hline

\multirow{3}{*}{$10^{-5}$}
& $10^{-10}$ & $1.2\;(6)$ & $1.2\;(6)$ & $1.2\;(6)$ \\
& $1$        & $1.2\;(6)$ & $1.2\;(6)$ & $1.2\;(6)$ \\
& $10^{5}$   & $1.2\;(8)$ & $1.2\;(8)$ & $1.2\;(8)$ \\

\hline

\multirow{3}{*}{$1$}
& $10^{-10}$ & $1.2\;(6)$ & $1.2\;(6)$ & $1.2\;(6)$ \\
& $1$        & $1.2\;(8)$ & $1.2\;(8)$ & $1.2\;(8)$ \\
& $10^{5}$   & $1.4\;(13)$ & $1.4\;(13)$ & $1.3\;(12)$ \\
\hline
\end{tabular}
\end{minipage}
}

\end{table}

\smallskip
\noindent\textbf{A robust preconditioner for the reaction-diffusion system.}
Consider the mixed reaction-diffusion system
\begin{equation}\label{eq:strong.reac.diff}
\begin{aligned}
    \alpha^{-1}\bzeta - \nabla \varphi & = \bzero \quad && \text{in }\Omega,\\
    \vdiv \bzeta - \varrho_0 \varphi & = -m \quad && \text{in }\Omega,\\
    \varphi & = 0 \quad && \text{on }\Gamma_{\bu},\\
    \bzeta \cdot \bn & = 0 \quad && \text{on }\Gamma_p.
\end{aligned}
\end{equation}
The natural choice of a \((\alpha,\varrho_0,h)\)-robust preconditioner for \eqref{eq:strong.reac.diff} (following \cite[Equation~3.8]{boon2021robust}) is the following
  \begin{equation}\label{abstract-prec-reac.diff}
\mathcal{B}_{{\mathrm{RD}}}:= \begin{pmatrix} (\frac{1}{\alpha}\cI - \nabla(\vdiv) )^{-1} & 0 \\ 0 & ([1+\varrho_0]\,\mathcal{I})^{-1}+(\varrho_0\mathcal{I}- {\alpha}\Delta )^{-1}
\end{pmatrix}.
\end{equation}
The condition numbers and the MINRES iteration counts are reported in Table~\ref{tab:reac.diff}.

\begin{table}[t!]
\caption{Condition numbers of the preconditioned system (corresponding to \eqref{abstract-prec-reac.diff}) with MINRES iterations in parentheses for different parameters and mesh sizes.}
\label{tab:reac.diff}
\centering
{\footnotesize \begin{tabular}{ccccc}
\hline
$\alpha$ & $\varrho_0$ & \multicolumn{3}{c}{$h$} \\
\cline{3-5}
 & & $2^{-2}$ & $2^{-3}$ & $2^{-4}$ \\
\hline

\multirow{4}{*}{$10^{-9}$}
& $10^{-9}$ & $3.7\;(17)$ & $4.1\;(19)$ & $4.4\;(21)$ \\
& $10^{-6}$ & $1.2\;(10)$ & $1.7\;(13)$ & $2.5\;(18)$ \\
& $10^{-3}$ & $1.0\;(4)$  & $1.0\;(4)$  & $1.0\;(4)$  \\
& $1$       & $1.5\;(3)$  & $1.5\;(3)$  & $1.5\;(3)$  \\

\hline

\multirow{4}{*}{$10^{-3}$}
& $10^{-9}$ & $3.1\;(24)$ & $3.3\;(27)$ & $3.4\;(29)$ \\
& $10^{-6}$ & $3.1\;(24)$ & $3.3\;(27)$ & $3.4\;(29)$ \\
& $10^{-3}$ & $3.0\;(23)$ & $3.0\;(25)$ & $3.1\;(29)$ \\
& $1$       & $1.9\;(12)$ & $2.3\;(15)$ & $2.6\;(19)$ \\

\hline

\multirow{4}{*}{$1$}
& $10^{-9}$ & $1.2\;(11)$ & $1.2\;(11)$ & $1.2\;(11)$ \\
& $10^{-6}$ & $1.2\;(11)$ & $1.2\;(11)$ & $1.2\;(11)$ \\
& $10^{-3}$ & $1.2\;(11)$ & $1.2\;(11)$ & $1.2\;(11)$ \\
& $1$       & $2.6\;(14)$ & $2.6\;(14)$ & $2.6\;(14)$ \\
\hline
\end{tabular}}

\end{table}

\begin{table}[t!]
\caption{Condition numbers and MINRES iteration counts in parenthesis for different values of \(\alpha\) and mesh sizes $h$ corresponding to the preconditioner~ \eqref{Prec:B2.advec.diff}. The symbol $\star$ indicates that the MINRES iteration count exceeds 1000.}
\label{tab:advec.reac.diff}
\centering
{\footnotesize \begin{tabular}{cccccc}
\hline
$|\hat\bu|$ & $\varrho_0$ & $\alpha$ & \multicolumn{3}{c}{$h$} \\
\cline{4-6}
& & & $2^{-2}$ & $2^{-3}$ & $2^{-4}$ \\
\hline
\multirow{12}{*}{$10^{-10}$}
& \multirow{4}{*}{$10^{-8}$}
& $10^{-9}$ & $3.3\;(26)$ & $3.9\;(31)$ & $4.2\;(32)$ \\
& & $10^{-6}$ & $3.7\;(19)$ & $4.2\;(26)$ & $4.4\;(28)$ \\
& & $10^{-3}$ & $3.1\;(20)$  & $3.3\;(26)$  & $3.4\;(28)$  \\
& & $1$       & $1.2\;(11)$  & $1.2\;(11)$  & $1.2\;(11)$  \\

\cline{2-6}

& \multirow{4}{*}{$10^{-4}$}
& $10^{-9}$ & $1.0\;(5)$ & $1.0\;(6)$ & $1.0\;(7)$ \\
& & $10^{-6}$ & $2.2\;(20)$ & $3.0\;(29)$ & $3.7\;(33)$ \\
& & $10^{-3}$ & $3.1\;(20)$  & $3.2\;(25)$  & $3.3\;(28)$  \\
& & $1$       & $1.2\;(11)$  & $1.2\;(11)$  & $1.2\;(11)$  \\

\cline{2-6}

& \multirow{4}{*}{1}
& $10^{-9}$ & $1.5\;(3)$ & $1.5\;(3)$ & $1.5\;(3)$ \\
& & $10^{-6}$ & $1.5\;(3)$ & $1.5\;(3)$ & $1.3\;(5)$ \\
& & $10^{-3}$ & $1.9\;(11)$  & $2.6\;(14)$  & $2.6\;(18)$  \\
& & $1$       & $2.6\;(13)$  & $2.6\;(13)$  & $2.6\;(13)$  \\

\hline

\multirow{12}{*}{$10^{-5}$}
& \multirow{4}{*}{$10^{-8}$}
& $10^{-9}$ & $1567.6\;(63)$ & $2232.1\;(248)$ & $3189.9\;(\star)$ \\
& & $10^{-6}$ & $64.1\;(42)$ & $55.2\;(74)$ & $6.0\;(61)$ \\
& & $10^{-3}$ & $3.1\;(20)$ & $3.3\;(26)$ & $3.3\;(29)$ \\
& & $1$       & $1.2\;(11)$ & $1.2\;(11)$ & $1.2\;(11)$ \\

\cline{2-6}

& \multirow{4}{*}{$10^{-4}$}
& $10^{-9}$ & $3.9\;(32)$ & $64.6\;(128)$& $105.6\;(521)$ \\
& & $10^{-6}$ & $3.3\;(31)$ & $4.8\;(48)$ & $1.5\;(5)$ \\
& & $10^{-3}$ & $3.1\;(20)$ & $3.2\;(26)$ & $3.3\;(29)$ \\
& & $1$       & $1.2\;(11)$ & $1.2\;(11)$ & $1.2\;(11)$ \\
\cline{2-6}

& \multirow{4}{*}{1}
& $10^{-9}$ & $1.5\;(3)$ & $1.5\;(4)$ & $1.5\;(5)$ \\
& & $10^{-6}$ & $1.5\;(3)$ & $1.5\;(5)$ & $1.5\;(5)$ \\
& & $10^{-3}$ & $1.9\;(11)$  & $2.3\;(14)$  & $2.6\;(18)$  \\
& & $1$       & $2.6\;(13)$  & $2.6\;(13)$  & $2.6\;(13)$  \\

\hline

\multirow{12}{*}{$10^{-1}$}
& \multirow{4}{*}{$10^{-8}$}
& $10^{-9}$ & $1.4\cdot 10^{7}\;(61)$ & $1.1\cdot 10^{7}\;(243)$ & $8.6\cdot 10^{6}\;(991)$ \\
& & $10^{-6}$ & $22040.3\;(64)$ & $17537.9\;(250)$ & $12259.2\;(998)$ \\
& & $10^{-3}$ & $21.1\;(65)$ & $190.7\;(230)$ & $356.6\;(377)$ \\
& & $1$       & $1.2\;(11)$ & $1.2\;(11)$ & $1.2\;(11)$ \\

\cline{2-6}

& \multirow{4}{*}{$10^{-4}$}
& $10^{-9}$ & $24970.7\;(62)$ & $30598.7\;(245)$ & $44066.2\;(988)$ \\
& & $10^{-6}$ &$3711.7\;(64)$ & $4484.1\;(248)$ & $3435.6\;(996)$ \\
& & $10^{-3}$ & $21.0\;(65)$ & $186.5\;(230)$ & $345.8\;(377)$ \\
& & $1$       & $1.2\;(11)$ & $1.2\;(11)$ & $1.2\;(11)$ \\

\cline{2-6}

& \multirow{4}{*}{1}
& $10^{-9}$ & $1.9\;(43)$ & $7.4\;(143)$ & $9.8\;(494)$ \\
& & $10^{-6}$ & $1.9\;(43)$ & $7.5\;(143)$ & $9.4\;(491)$ \\
& & $10^{-3}$ & $1.8\;(39)$  & $6.0\;(84)$  & $18.2\;(126)$  \\
& & $1$       & $2.6\;(14)$  & $2.5\;(14)$  & $2.5\;(13)$  \\
\hline
\end{tabular}}

\end{table}

\smallskip
\noindent\textbf{Robust preconditioners for the {ADR} system.}
Consider now the ADR system
\begin{equation}\label{eq:strong.advection.diff}
\begin{aligned}
    \alpha^{-1}\bzeta - \nabla \varphi + \alpha^{-1}\hat\bu\,\varphi & = \bzero \quad && \text{in }\Omega,\\
    \vdiv \bzeta -\varrho_0\varphi& = -m \quad && \text{in }\Omega,\\
    \varphi & = 0 \quad && \text{on }\Gamma_{\bu},\\
    \bzeta \cdot \bn & = 0 \quad && \text{on }\Gamma_p.
\end{aligned}
\end{equation}
The preconditioners~\eqref{precond.linear.B2} for the full system already include the blocks needed for the sub-system~\eqref{eq:strong.advection.diff}.  Here, we vary the magnitude of the Darcy--Forchheimer velocity~$\hat{\bu}$ to illustrate the deterioration in performance of these preconditioners in the advection-dominated regime. Specifically, we restrict our attention to $\mathrm{diag}({\cR}_2,{\cS}_2)$ from \eqref{precond.linear.B2}, noting that the remaining preconditioners exhibit similar behaviour:
\begin{equation}\label{Prec:B2.advec.diff}
    \mathcal{B}^{1,2}_{\mathsf{DFADR}}:=\begin{pmatrix}
        (\alpha^{-1}\cI-\nabla\vdiv)^{-1}&0\\0&{(1+\varrho_0)\,\cI^{-1}+(\varrho_0\mathcal{I}-\alpha\Delta + \hat{\bu}\cdot \nabla)^{-1}}
    \end{pmatrix}.
\end{equation}
{A conclusion from Table~\ref{tab:advec.reac.diff} is that whenever $\frac{|\hat{\mathbf{u}}|h}{\alpha} \ll 1$, both the condition numbers of the preconditioned system and the MINRES iteration counts remain well controlled. Outside this regime, however, the performance of the preconditioner deteriorates and becomes inconsistent.
Possible remedies include a better stabilisation for the advection term~$\bu^m \varphi$ (pure upwind in general will not yield robust preconditioners for extreme regimes).
It is well established \cite{ernst2000residual} that classical discretisations of convection-dominated systems give rise to highly nonnormal discrete operators. Such nonnormality is not adequately reflected by the eigenvalues of the preconditioned systems alone; instead, their behavior deteriorates markedly  by increasing  P\'eclet number and mesh refinement. As a consequence, Krylov subspace methods will struggle for this type of problems.}

Another difficulty already mentioned in Section \ref{sec:precond} is the nonsymmetric nature of the underlying discrete operators. Operator preconditioning and the usual  Schur complement structures are not yet  systematically developed. Here we have adopted a structure motivated by symmetric elliptic operators, which do not seem to provide definitive or robust conclusions.
See \cite{wathen2015preconditioning} for a detailed review of the challenges associated with preconditioning nonsymmetric systems.

\section*{Funding}
RD acknowledges support from the Industrial Research and Consultancy Centre (IRCC), IIT Bombay (grant no. RI/0519-10001788-001). HH acknowledges financial support from the Anusandhan National Research Foundation (ANRF), Government of India (grant no. ANRF/ARG/2025/011989/MS). RRB acknowledges support from  the Australian Research Council through the \textit{Future Fellowship} grant FT220100496, and from the Center of Advanced Study (CAS) at the Norwegian Academy of Science and Letters under the program \textit{Mathematical Challenges in Brain Mechanics}.

\end{document}